\documentclass[10pt]{article}

\usepackage{hyperref}

\usepackage{titlesec}

\usepackage{amsmath}
\usepackage{epsfig, amssymb, amsfonts, dsfont}
\usepackage{amsthm}
\usepackage{amstext}
\usepackage{amsopn}
\usepackage{mathrsfs}
\usepackage{subfigure}
\usepackage{graphicx}
\usepackage[left=2.3cm,top=2cm,right=2.3cm,bottom=2cm, nohead,foot=1cm]{geometry}
\usepackage[makeroom]{cancel}
\usepackage{color}
\usepackage{enumerate}
\usepackage{comment}
\usepackage{stackrel}
\usepackage{stmaryrd}
\usepackage{mathtools}
\usepackage{amsbsy}

\usepackage{tikz}
\usepackage{tikz-cd}

\SetSymbolFont{stmry}{bold}{U}{stmry}{m}{n}

\def\stackrel#1#2{\mathrel{\mathop{#2}\limits^{#1}}}

\numberwithin{equation}{section}

\newtheorem{convention}{Convention}[section]

\newtheorem{hypothesis}{Hypothesis}

\newtheorem{theorem}{Theorem}[section]

\newtheorem{lemma}[theorem]{Lemma}

\newtheorem{corollary}[theorem]{Corollary}
\newtheorem{proposition}[theorem]{Proposition}

\newtheorem{notation}[theorem]{Notation}
\theoremstyle{definition}
\newtheorem{definition}[theorem]{Definition}
\newtheorem{remark}[theorem]{Remark}

\newcommand{\R}{{\mathbb R}}

\DeclareFontFamily{U}{mathx}{\hyphenchar\font45}
\DeclareFontShape{U}{mathx}{m}{n}{
      <5> <6> <7> <8> <9> <10>
      <10.95> <12> <14.4> <17.28> <20.74> <24.88>
      mathx10
      }{}
\DeclareSymbolFont{mathx}{U}{mathx}{m}{n}
\DeclareMathSymbol{\bigtimes}{1}{mathx}{"91}

\usepackage{relsize}

\newcommand{\N}{{\mathbb N}}
\newcommand{\Tr}{{\textup{Tr}}}

\newcommand{\C}{{\mathbb C}}

\title{Planar diffusions with a point interaction on a finite time horizon}

\date{  }

 \author{\textbf{Barkat Mian}\footnote{ {\tt bmian@utk.edu}} \vspace{.1cm}  \\  University of Tennessee Knoxville, Department of Mathematics   }

\begin{document}
\maketitle

\begin{abstract}
The skew-product diffusion~[Ann. Appl. Probab. \textbf{35}, 3150--3214 (2025)] and exponentially tilted planar Brownian motion~[Electron. J. Probab. \textbf{30}, 1--97 (2025)] are canonical examples of planar diffusions with a point interaction at the origin in the sense that their drifts are singular only at the origin and allow visits there with positive probability. However, in this article we propose an axiomatic framework for such diffusions on a finite time horizon. We isolate admissibility conditions and additional regularity hypotheses on a general driving family under which the associated diffusion, constructed as a Doob transform of point-interaction Schr\"odinger semigroup kernels, exhibits the same point interaction structure. In particular, for the ground-state driving family, we obtain a heuristic alternative construction of the skew-product diffusion based on Kolmogorov continuity arguments. We also consider formal applications of this framework to some driving families generated by measures, including the Lebesgue-driven diffusion, which is the exponentially tilted planar Brownian motion.

\end{abstract}

\vspace{.2cm}

\section{Introduction}
The article~\cite{Chen2} develops a probabilistic framework for planar diffusions describing the motion of a particle in two dimensions subject to a point (zero-range) interaction at the origin; see also recent related works by Chen. In~\cite{CM}, related planar diffusions with a point interaction are constructed on a finite time interval, and their laws are related to the second moment of the continuum polymer measures corresponding to the critical two-dimensional stochastic heat flow; see~\cite[Sec.~5.3]{CM2}. From an analytic perspective, the dynamics of such systems are governed by the singular two-dimensional Schr\"odinger Hamiltonian with zero-range potential at the origin, which can be written formally as
\begin{align}\label{Hamiltonian}
\mathscr{H}
\,=\,
-\frac{1}{2}\,\Delta
\,-\,
\boldsymbol{\epsilon}\,\delta(x)\,,
\qquad x\in\R^2 \,.
\end{align}
Here $\Delta = \partial_{x_1}^2 + \partial_{x_2}^2$ denotes the Laplacian on $\R^2$,
$\delta(x)$ is the Dirac distribution at the origin, and $\boldsymbol{\epsilon}\in\R$ is a coupling constant. The above operator arises naturally as the Hamiltonian governing the relative motion of two interacting particles in the two-dimensional delta-Bose gas. Indeed, consider the formal two-body Hamiltonian $-\frac12\Delta_{x_1}-\frac12\Delta_{x_2}-\boldsymbol{\epsilon}\,\delta(x_1-x_2)$ describing a system of two identical bosons with positions $x_1,x_2\in\R^2$ interacting through a zero-range potential. In center of mass and relative coordinates $z:=(x_1+x_2)/2$ and $y:=x_1-x_2$, it separates as $-\frac14\Delta_z +\big(-\frac12\Delta_y-\boldsymbol{\epsilon}\,\delta(y)\big)$, so the center of mass motion evolves freely and the interaction is entirely captured by the Hamiltonian~\eqref{Hamiltonian} acting on the relative coordinate $y$. Formally, the semigroup associated with $\mathscr{H}$ admits the Feynman-Kac representation: for every bounded Borel $f$ and $t\ge0$,
\begin{align}\label{FKFormal}
\big(e^{-t\mathscr{H}}f\big)(x)
\,\stackrel{\mathrm{formal}}{=}\,
\mathbb{E}_x\!\bigg[
\exp\!\bigg\{\boldsymbol{\epsilon}\int_0^t \delta(W_s)\,ds\bigg\}\,f(W_t)
\bigg]\,,
\end{align}
where $\{W_t\}_{t\in [0,\infty)}$ is a Brownian motion in $\R^2$ started at $x$ under $\mathbb P_x$. In two dimensions the functional $\int_0^t \delta(W_s)\,ds$ is not defined without renormalization since a two-dimensional Brownian motion will almost surely never visit the origin after time zero, so~\eqref{FKFormal} is only a heuristic guide.

A rigorous interpretation of the formal Hamiltonian~\eqref{Hamiltonian} is obtained via the renormalization procedure introduced by Albeverio, Gesztesy, H{\o}egh-Krohn, and Holden~\cite{AGHH}. Fix $\vartheta, \varepsilon>0$ and let $\phi\in C_c(\R^2)$ be a normalized mollifier. Consider the regular Schr\"odinger operators with short-range potentials of the form
\begin{align}\label{WDHamiltonian}
    \mathscr{H}_\varepsilon^{\vartheta}
\,:=\,
-\frac12\Delta \,-\, \boldsymbol{\epsilon}^{\vartheta}_{\varepsilon}\,\delta_{\varepsilon}(x)\,,
\end{align}
where $\delta_{\varepsilon}(x)
=
\varepsilon^{-2}\phi(x/\varepsilon)$ and the coupling constant $\boldsymbol{\epsilon}^{\vartheta}_{\varepsilon}$ decays logarithmically according to the asymptotic expansion
\begin{align}
\boldsymbol{\epsilon}^{\vartheta}_{\varepsilon}
\,:=\,
\frac{\pi}{\log \frac{1}{\varepsilon}}
\bigg(
1
\,+\,
\frac{\frac12\log\frac{\vartheta}{2}
\,+\,
\gamma_{\mathrm{EM}}
\,+\,
I_{\phi}}
{\log \frac{1}{\varepsilon}}
\bigg)
\,+\,
\mathit{o}\bigg(\frac{ 1}{ \log^2 \frac{1}{\varepsilon} } \bigg)\,,
\qquad \varepsilon\downarrow0\,, \nonumber
\end{align}
where $I_{\phi} = \int_{(\R^2)^2}\log|x-y|\,\phi(x)\phi(y)\,dx\,dy$ and $\gamma_{\mathrm{EM}}$ denotes the Euler-Mascheroni constant. Then the operators $\mathscr{H}_\varepsilon^{\vartheta}$ converge, in the sense of norm resolvent convergence, to a self-adjoint operator
$-\tfrac12\Delta^{\vartheta}$ on $L^2(\R^2)$, meaning that for every $z\in\C\setminus\R$ the resolvent operators
$(\mathscr{H}_\varepsilon^{\vartheta}-zI)^{-1}$ converge in the operator norm on $L^2(\R^2)$ to $(-\tfrac12\Delta^{\vartheta}-zI)^{-1}$ as $\varepsilon\downarrow0$. The Feynman-Kac formula gives the following integral representation: for every bounded Borel $f$ on $\R^2$ and $t\ge0$,
\begin{align}
\big(e^{-t\mathscr{H}_{\varepsilon}^{\vartheta}}f\big)(x)
\,=\,
\mathbb{E}_x\!\bigg[
\exp\!\bigg\{
\int_0^t \boldsymbol{\epsilon}^{\vartheta}_{\varepsilon}\,\delta_{\varepsilon}(W_s)\,ds
\bigg\}
\,f(W_t)
\bigg]\,. \nonumber
\end{align}
The above mollification scheme also provides a rigorous interpretation of the formal
Feynman-Kac expression~\eqref{FKFormal}. In particular, Chen~\cite{Chen1} proved that,
for every bounded Borel measurable function $f$ on $\R^2$ and every $x\in\R^2\setminus\{0\}$,
\begin{align}
\big(e^{-t\mathscr{H}_\varepsilon^{\vartheta}}f\big)(x)
\hspace{.5cm}\stackrel{\varepsilon\rightarrow 0}{\longrightarrow} \hspace{.5cm}
\big(e^{\,\frac{t}{2}\Delta^{\vartheta}}f\big)(x)\,,
\qquad t\ge0\,. \nonumber
\end{align}
Each operator $-\tfrac12\Delta^{\vartheta}$ generates a strongly continuous semigroup
$\{e^{\,\frac{t}{2}\Delta^{\vartheta}}\}_{t\in[0,\infty)}$ with integral kernel
$\mathsf{K}_t^{\vartheta}:\R^2\times\R^2\to[0,\infty)$.
For $t>0$ and $x,y\in\R^2$, this kernel is given explicitly by the formula below;
see~\cite[Thm.~2.2]{AGHH}.
\begin{align}\label{DefFullKer}
\mathsf{K}_t^{\vartheta}(x,y)
\,:=\,
g_t(x-y)
\,+\, 2\,\pi\, \vartheta
\int_{0 < r < s < t}
g_{r}(x)\,
\nu'\!\big(\vartheta(s-r)\big)\,
g_{t-s}(y)\,ds\,dr \,,
\end{align}
where $g_t(x):= (2\pi t)^{-1}\,e^{-\frac{|x|^2}{2t}}$ denotes the two-dimensional Gaussian heat kernel and $\nu'$ is the derivative of the Volterra function  $\nu:\C\to\C$ defined by $\nu(a) := \int_0^{\infty}\frac{a^{s}}{\Gamma(s+1)}ds $. See~\cite[(3.11)]{Albeverio},~\cite[(2.7)]{GQT},~\cite[(3.56)]{CSZ5},
and~\cite[(2.17)]{Chen2} for some equivalent representations of the integral kernel $\mathsf{K}^{\vartheta}_t(x,y)$.

\subsection{Overview of the model}

The current planar diffusions with a point interaction at the origin are
constructed using entirely different methods. The diffusions in~\cite{CM} are obtained through a Doob transform of the family of integral kernels $\{\mathsf{K}_t^{\vartheta} \}_{t\in[0,\infty)}$ using the driving functions $h_t^{\vartheta,\textup{Leb}} := \mathsf{K}_t^{\vartheta}\mathbf{1}$, which correspond to integrating $\mathsf{K}_t^{\vartheta}$ against the Lebesgue measure; see~\eqref{Lebfamily}. One objective of the present work is to explore the possibility of an alternative construction of the skew-product diffusion introduced in~\cite{Chen2} using this Doob transform based technique; see Section~\ref{GStDiffusion} for a formal implementation of this approach and a discussion of the technical difficulties that arise in this case. Another objective is to extend this Doob transform based construction to more general driving families generated by measures; see, for example,~\eqref{Dirfamily} and~\eqref{Gaufamily} for the driving functions obtained by integrating $\mathsf{K}_t^{\vartheta}$ against the Dirac mass at the origin and a Gaussian probability measure, respectively. In the present paper, our focus is on
developing an abstract axiomatic framework for planar diffusions with a point interaction and on identifying the structural properties of the associated driving functions. Since a complete verification of the required integrability and asymptotic conditions for each individual family would require a separate, model-specific analysis, we apply our results to the above examples only at a formal level; see Section~\ref{ExamplesDiffusion}. Below we summarize the Doob transform construction underlying planar diffusions with a point interaction and highlight the structural features that distinguish these processes from standard planar diffusions; see Figure~\ref{FigBasicPropertiesPD}
for a summary of the basic properties. In particular, we emphasize the role of the driving family $h^{\vartheta}$, the resulting singular drift, and the pathwise mechanisms that encode the point interaction. These considerations motivate the axiomatic definition of such diffusions.

Fix $T,\vartheta>0$. The family $\{\mathsf{K}_t^{\vartheta}(x,y)\}_{t\in[0,T]}$
consists of the integral kernels of the semigroup
$\{e^{\frac{t}{2}\Delta^{\vartheta}}\}_{t\in[0,\infty)}$, restricted to
$t\in[0,T]$, and therefore satisfies the semigroup property
\begin{align}\label{SemigroupP}
\int_{\R^2}\mathsf{K}_{t}^{\vartheta}(x,z)\,
\mathsf{K}_{T-t}^{\vartheta}(z,y)\,dz
\,=\,
\mathsf{K}_{T}^{\vartheta}(x,y)\,,
\qquad 0<t<T\,.
\end{align}
However, $\mathsf{K}_t^{\vartheta}(x,\cdot)$ is not a probability density, since $(\mathsf{K}_t^{\vartheta}\mathbf{1})(x)>1$ due to the presence of the interaction term in~\eqref{DefFullKer}. It is therefore natural to view planar diffusions with a point interaction as Doob transforms of the family of integral kernels $\{\mathsf{K}_t^{\vartheta}(x,y) \}_{ t \in[0,T]}$ with respect to an appropriate choice of driving family $h^{\vartheta}$. In particular, the process $X$ has transition density kernels given, for $0\le s<t\le T$ and $x,y\in\R^2\setminus\{0\}$, by
\begin{align}\label{FirstTrans}
\mathlarger{\mathsf{d}}_{s,t}^{T,\vartheta,h}(x,y)
\,:=\,
\frac{h_{T-t}^{\vartheta}(y)}{h_{T-s}^{\vartheta}(x)}\,
\mathsf{K}_{t-s}^{\vartheta}(x,y) \,.
\end{align}
For $x=0$, the density $y\mapsto\mathlarger{\mathsf{d}}_{s,t}^{T,\vartheta,h}(0,y)$ is defined by taking the limit of the right hand side above as $x\to0$. For fixed $s\in[0,T]$ and $x\in\R^2$, the map
$(t,y)\mapsto \mathlarger{\mathsf{d}}_{s,t}^{T,\vartheta,h}(x,y)$, defined on
$(s,T]\times\R^2\setminus\{0\}$, satisfies the forward Kolmogorov equation
\begin{align}\label{KolmogorovForD}
\partial_t\,
\mathlarger{\mathsf{d}}_{s,t}^{T,\vartheta,h}(x,y)
\,=\,
\frac12\,\Delta_y\,\mathlarger{\mathsf{d}}_{s,t}^{T,\vartheta,h}(x,y)
\,-\,
\nabla_y\!\cdot\!\Big[
b_{T-t}^{\vartheta,h}(y)\,
\mathlarger{\mathsf{d}}_{s,t}^{T,\vartheta,h}(x,y)
\Big]\,,
\end{align}
where $\nabla_y=(\partial_{y_1},\partial_{y_2})$ denotes the gradient in the space variable and the drift $b_t^{\vartheta,h} : \R^2 \setminus \{0\}\to\R^2$ is given by the logarithmic gradient of the deriving function,
\begin{align}\label{DefDriftFunD}
b_t^{\vartheta,h}(x)
\,:=\,
\nabla_x\log h_t^{\vartheta}(x)
\,=\,
\frac{\nabla_x h_t^{\vartheta}(x)}{h_t^{\vartheta}(x)} \,,
\qquad x\neq0 \,.
\end{align}
We consider the driving family $h^{\vartheta}$ with the logarithmic blowup near the origin in~\eqref{hSmallaAsympt} so that the associated drift
field $b^{\vartheta,h}$ exhibits a \textit{critical} singular behavior at the
origin that is strong enough for the resulting diffusion to reflect the attractive nature of the point interaction
encoded in the negative potential in~\eqref{WDHamiltonian}.\footnote{The current models describing the two-dimensional delta interaction
have drifts with the singularity~\ref{Bee}, which motivates our choice of the logarithmic singularity at the origin for the functions $h_t^{\vartheta}$.} More precisely, for $x\neq0$ the vector $b_t^{\vartheta,h}(x)$ points toward the origin, its Euclidean norm decays super-exponentially as $|x|\to\infty$, and it blows up as $|x|\downarrow0$ according to the asymptotic
\begin{align}\label{Bee}
|b_t^{\vartheta,h}(x)| \,\sim\, \frac{1}{|x|\log(1/|x|)}
\qquad\text{as }|x|\downarrow 0\,.
\end{align}
In particular, the process $X$ can reach the origin, unlike planar Brownian motion, and hence the associated path measure $\mathbb{P}_x^{T,\vartheta,h}$ on $C([0,T],\R^2)$ is not absolutely continuous with respect to the planar Wiener measure $\mathbb{P}_x^{T}$. However, the
origin does not act as a boundary state for the process, and the interaction at the origin is of zero range.

An alternative characterization of the process \(X\) is provided by the following SDE
\begin{align}\label{SDEToSolve}
dX_t
\,=\,
dW_t^{T,\vartheta}
\,+\,
b_{T-t}^{\vartheta,h}(X_t)\,dt \,,
\qquad
t\in[0,T] \,,
\end{align}
where \(\{W_t^{T,\vartheta}\}_{t\in[0,T]}\) is a two-dimensional Brownian
motion under \(\mathbb{P}_x^{T,\vartheta,h}\), and
\(b^{\vartheta,h}\) is the singular drift field defined in~\eqref{DefDriftFunD}.
Since the path measure \(\mathbb{P}_x^{T,\vartheta,h}\) is singular with respect to \(\mathbb{P}_x^{T}\), Girsanov's theorem
(e.g.\ \cite[Thm.~3.51]{Karatzas}) cannot be applied directly to construct a weak solution to~\eqref{SDEToSolve}. In our present framework, the drift field \(b^{\vartheta,h}\) is time-dependent. However, explicit time dependence in the coefficient \(b_{T-t}^{\vartheta}\) does not, by itself, pose a fundamental difficulty, since its magnitude decreases as \(t\uparrow T\). Rather, the principal analytical challenge stems from the singular behavior of the drift at the origin. Stochastic differential equations with singular drift coefficients have been studied extensively; see, for example,~\cite{Chen, Bass,Krylov,Zhang,Jin,Kinzebulatov,Rockner}, which focus on multidimensional SDEs with time-dependent singular drifts. In the absence of time dependence, the process $X$ governed by~\eqref{SDEToSolve} belongs to the class of \emph{distorted Brownian motions}, that is, diffusions with drift of the form $\nabla_x \log h^{\vartheta}(x)$ independent of time. Distorted Brownian motions form a well-studied class of diffusions; see, for example,~\cite{Ezawa,Fukushima,Streit,Kondratiev,Trutnau}. However, no general result from either of these theories applies to the present setting, due to the particularly strong norm blowup of the drift at the origin. In the present setting the dynamics are regular away from the origin, while the singular behavior at the origin reflects the zero-range nature of the interaction. In~\cite{Chen2}, a solution to an SDE of the above type with time-independent drift is constructed; see Remark~\ref{RemarkChen2}. In the present work, we generalize the Kolmogorov continuity and martingale based approach of~\cite{CM}, developed there for the specific driving family \(h_t^{\vartheta, \textup{Leb}} :=\mathsf{K}_t^{\vartheta} \mathbf{1}\), to more general families \(h^{\vartheta}\) satisfying suitable analytic properties, and we formulate the corresponding model dependent requirements as hypotheses.

The discussion above highlights that the diffusion $X$ with a point interaction at the origin admits explicit
descriptions in terms of transition densities and an SDE with a singular drift, whose singular behavior is concentrated at the origin. These descriptions provide concrete representations of the process, but they do not by themselves capture the nature of the delta interaction. Rather, the defining features of the model are pathwise and structural: $X$ is a continuous planar diffusion whose dynamics are regular on $\R^2\setminus\{0\}$, while all effects of the interaction are localized at a single point. The origin influences the motion through repeated, instantaneous contacts that accumulate over time; however, since the interaction is point-localized and acts only through randomness, the process does not dwell at the origin, is not killed, is not reflected in the sense of being deterministically pushed away—the random fluctuations allow it to move away despite the inward drift—exhibits no jumps, as the paths remain continuous even in the presence of a highly singular drift, and gives rise to no nonlocal effects, since the interaction acts only at the origin. It is this structural behavior that we isolate and formalize in the axiomatic definition below.

\begin{definition}[Planar diffusion with a point interaction]
\label{DefRelativehMotion}
Fix $T,\vartheta>0$ and let
$h^\vartheta=\{h_t^\vartheta\}_{t\in[0,T]}$
be a family of functions
$h_t^\vartheta:\R^2\setminus\{0\}\to(0,\infty)$
such that
$h_t^\vartheta(x)\to+\infty$ as $x\to0$ for each $t\in(0,T]$.
We say that $h^\vartheta$ is an \emph{admissible family} if there exists a
two-dimensional stochastic process
$X=\{X_t\}_{t\in[0,T]}$
such that, for every $x\in\R^2\setminus\{0\}$, under the law
$\mathbb{P}_x^{T,\vartheta,h}$ of $X$ started from
$X_0=x$, the following axioms are satisfied.
\begin{enumerate}
\item[(A1)]
$X$ is a continuous Markov process on
$\R^2$ whose transition density kernels on $\R^2\setminus\{0\}$ are given by the
Doob transform \eqref{FirstTrans}.

\item[(A2)]
On $\R^2\setminus\{0\}$, the infinitesimal generator of $X$ at time
$t\in[0,T)$ is given below.
\[
\mathscr{L}_t^{T,\vartheta,h}
\,= \,
\frac12\,\Delta
\,+ \,
\nabla\log h_{T-t}^{\vartheta}\cdot\nabla  \,
\]

\item[(A3)]
There exists a bounded, continuous function
$p^{\vartheta,h}:[0,T]\times\R^2\to[0,1]$ such that, defining
\[
\mathbf{S}_t^{T,\vartheta,h}
\, := \,
p^{\vartheta,h}(T-t,X_t) \,,
\qquad t\in[0,T] \,,
\]
the process $\mathbf{S}^{T,\vartheta,h}$ is a continuous submartingale.
Moreover, in the Doob-Meyer decomposition $\mathbf{S}^{T,\vartheta,h}
=
\mathbf{M}^{T,\vartheta,h}
+
\mathbf{A}^{T,\vartheta,h}$ the process $\mathbf{A}^{T,\vartheta, h}$ is continuous, nondecreasing, starts from the origin, and is supported at the origin in the sense that
\[
\int_0^{T} \mathbf{1}_{\{X_t\neq 0\}}\,d\mathbf{A}_t^{T,\vartheta,h} \,=0 \,,
\qquad
\mathbb{P}_x^{T,\vartheta,h}\text{-a.s.}
\]
\end{enumerate}
We refer to any such process
$X=\{X_t\}_{t\in[0,T]}$
as a \emph{planar diffusion with a point interaction at the origin driven by
$h^\vartheta$}, or equivalently as a \emph{planar $h^\vartheta$-diffusion with
point interaction}. When no confusion arises, we also simply call it a
\emph{planar $h^\vartheta$-diffusion}, or use the terms
\emph{planar diffusion with point potential} or
\emph{zero-range planar diffusion}.

\end{definition}

For the existence of a planar diffusion with delta interaction, the blowup of the driving family $h^\vartheta$ near the origin must be carefully balanced: it must be strong enough to induce repeated visits to the origin through the singular drift, yet not so strong as to trap the process or cause it to spend a nonnegligible amount of time there. Under such minimal analytic requirements on $h^\vartheta$, together with
additional regularity assumptions, existence can be established. In particular, the Markov family in \textup{(A1)} is obtained in
Propositions~\ref{PropCPD}, the stochastic dynamics corresponding to
\textup{(A2)} are constructed in Proposition~\ref{PropStochPreD}, and the existence and structural properties of the submartingale in \textup{(A3)} are established in Proposition~\ref{PropSubMartD}. For clarity and future reference, we summarize these results in the following existence theorem.

\begin{theorem}[Existence of point-interaction planar diffusions]
\label{ThmExistenceRelativeh}
Fix $T,\vartheta>0$.
There exist an admissible family
$h^\vartheta=\{h_t^\vartheta\}_{t\in[0,T]}$ and a corresponding stochastic process $X=\{X_t\}_{t\in[0,T]}$
such that $X$ is a planar diffusion with a point interaction at the origin driven by $h^\vartheta$ in the sense of Definition~\ref{DefRelativehMotion}.
\end{theorem}

An admissible driving family is given by \(h_t^{\vartheta}(x)=(\mathsf{K}_t^{\vartheta}\mathbf{1})(x)\), for which the corresponding delta-potential diffusion was constructed in~\cite{CM}; see Section~\ref{LebDiffusion}. In Definition~\ref{SpaceTimeHEFDef} we specify the minimal conditions required
for pre-admissibility, while Hypotheses~\ref{H2}-\ref{H4} list the additional
regularity assumptions under which the associated zero-range diffusion exists.
Accordingly, throughout this work, by an \textbf{admissible driving family} we
mean a pre-admissible driving family satisfying these standing assumptions.
Under this convention, the structural results stated below are proved in
Section~\ref{SubsectionThmSubMARTD} via the more technical
Theorem~\ref{ThmSubMARTD}.

\begin{theorem}[Structural properties of planar $h^\vartheta$-diffusions]
\label{ThmPropertiesRelativeh}
Fix $T,\vartheta>0$ and let $h^\vartheta=\{h_t^\vartheta\}_{t\in[0,T]}$ be an
admissible family, with corresponding point interaction planar diffusion $X=\{X_t\}_{t\in[0,T]}$ and law $\mathbb{P}_{x}^{T,\vartheta,h}$ started from
$x\in\R^2\setminus\{0\}$. Let
\[
\tau  \,:=  \, \inf\{t \in [0,\infty)  \,:  \,X_t = 0\}
\]
denote the first hitting time of the origin. Then the following statements hold.
\begin{enumerate}[(i)]
\item $\displaystyle \mathbb{P}_{x}^{T,\vartheta,h}[\tau>T]
\,=\,
\frac{(h_0^{\vartheta}*g_T)(x)}{h_T^{\vartheta}(x)}\,$, here and throughout, $*$ denotes convolution in the spatial variable.

\item
Let
\(
\mathring{\mathbb{P}}_{x}^{T,\vartheta,h}
:=
\mathbb{P}_{x}^{T,\vartheta,h}[\,\cdot \mid \tau>T]
\).
Then, under $\mathring{\mathbb{P}}_{x}^{T,\vartheta,h}$, the process $X$
satisfies the SDE~\eqref{SDEToSolve} with a
$\mathring{\mathbb{P}}_{x}^{T,\vartheta,h}$-Brownian motion
$\mathring{W}^{T,\vartheta,h}$ and the drift given below
\[
\mathring{b}_{t}^{\vartheta,h}(x)
\,=  \,
\nabla_x\log\big(h_0^{\vartheta}*g_t\big)(x) \,,
\qquad x\neq0 \,, \, t\in[0,T] \,.
\]

\item
Let $\mathsf{S}$ be a stopping time such that
$\mathbb{P}_{x}^{T,\vartheta,h}[\mathsf{S}<\tau]>0$, and define
\(
\mathring{\mathbb{P}}_{x}^{T,\vartheta,h, \mathsf{S}}
:=
\mathbb{P}_{x}^{T,\vartheta,h}[\,\cdot \mid \mathsf{S}<\tau]
\).
Then, on the event $\{\mathsf{S}<\tau\}$, the stopped process
$\{X_{t\wedge \mathsf{S}}\}_{t\in[0,T]}$
has the same law under
$\mathring{\mathbb{P}}_{x}^{T,\vartheta,h, \mathsf{S}}$
as it does under the weighted path measure
\[
\mathring{\mathbb{E}}_{x}^{T,\vartheta,h}\!\bigg[
\frac{h_{T-\mathsf{S}}^{\vartheta}\!\big(X_{\mathsf{S}}\big)}
{(h_0^{\vartheta} * g_{T-\mathsf{S}})\!\big(X_{\mathsf{S}}\big)}
\bigg]^{-1}
 \, \frac{h_{T-\mathsf{S}}^{\vartheta}\!\big(X_{\mathsf{S}}\big)}
{(h_0^{\vartheta} * g_{T-\mathsf{S}})\!\big(X_{\mathsf{S}}\big)}
\,
\mathring{\mathbb{P}}_{x}^{T,\vartheta,h} \,.
\]

\item
The conditional distribution of $\tau$ under
$\mathbb{P}_{x}^{T,\vartheta,h}$ given $\tau<T$ is absolutely continuous on
$(0,T)$, with density
\begin{align}\label{DistTau}
\frac{\mathbb{P}_x^{T,\vartheta,h}\!\big[\,\tau\in dt\,\big|\,\tau<T\big]}{dt}
\,=\,
\frac{-1}{h_T^{\vartheta}(x)-(h_0^{\vartheta} * g_T)(x)}\,
\frac{d}{dt}\,(h_{T-t}^\vartheta*g_t)(x) \,,
\qquad t\in(0,T) \,.
\end{align}

\end{enumerate}
\end{theorem}
Part~(i) shows that the survival probability up to the finite horizon $T$ is
strictly between zero and one. In particular, starting from any
$x\in\R^2\setminus\{0\}$, the process hits the origin before time $T$ with
positive probability, but it also survives without hitting the origin up to time
$T$ with positive probability. Thus the origin is neither absorbing nor
avoidable over the time interval $[0,T]$, reflecting the critical nature of the
two-dimensional point interaction. Part~(ii) identifies the dynamics under the
conditional law $\mathring{\mathbb P}_x^{T,\vartheta,h}$, under which the
singular point interaction disappears and the process $X$ is a regular planar
diffusion with transition density
\begin{align}\label{SecondTrans}
\mathlarger{\mathsf{\mathring{d}}}_{s,t}^{T,\vartheta,h}(x,y)
\,:=\,
\frac{(h_0^\vartheta*g_{T-t})(y)}{(h_0^\vartheta*g_{T-s})(x)}\, g_{t-s}(x-y)\,,
\qquad 0\le s<t\le T\,.
\end{align}

\begin{figure}[t]
\centering
\begin{tikzpicture}[font=\small,
  card/.style={draw, rounded corners, inner sep=7pt, align=left}
]

\node[card] (CARD) {%
\parbox{0.92\linewidth}{%
\centering \textbf{Planar diffusion $X$ with a point interaction at the origin driven by $h^\vartheta=\{h_t^\vartheta\}_{t\in[0,T]}$}\\[4pt]
\raggedright
\setlength{\tabcolsep}{0pt}
\begin{tabular}{@{}p{0.27\linewidth}p{0.65\linewidth}@{}}
\textbf{Base semigroup} &
$\displaystyle
\mathsf{K}_t^{\vartheta}(x,y)
\overset{\eqref{DefFullKer}}{:=}
g_t(x-y)
+ 2\,\pi\, \vartheta
\int_{0 < r < s < t}
g_{r}(x)\,
\nu'\!\big(\vartheta(s-r)\big)\,
g_{t-s}(y)\,ds\,dr
$
\end{tabular}
\\[4pt]

\setlength{\tabcolsep}{0pt}
\renewcommand{\arraystretch}{1.05}
\begin{tabular}{@{}p{0.27\linewidth}p{0.65\linewidth}@{}}

\textbf{Transition density} &
$\displaystyle
\mathsf d^{T,\vartheta,h}_{s,t}(x,y)
\overset{\eqref{FirstTrans}}{=}
\frac{h_{T-t}^\vartheta(y)}{h_{T-s}^\vartheta(x)}\,\mathsf{K}_{t-s}^\vartheta(x,y),
\qquad 0\le s<t\le T
$\\[9pt]

\textbf{Law} &
$\mathbb{P}^{T,\vartheta,h}_x$ on $C([0,T],\R^2)$ \qquad \big(Prop.\ \ref{PropCPD} \big) \\[8pt]

\textbf{Drift} &
$\displaystyle
b_t^{\vartheta,h}(x)
\overset{\eqref{DefDriftFunD}}{=}
\nabla\log h_t^\vartheta(x),
\qquad x\neq 0
$\\[8pt]

\textbf{Generator} &
$\displaystyle
\mathscr{L}^{T,\vartheta,h}_t
=
\frac12\,\Delta
+
b_{T-t}^{\vartheta,h}\cdot\nabla,
\qquad x\neq0
$\\[8pt]

\textbf{SDE} &
$\displaystyle
dX_t
=
dW_t^{T,\vartheta}
+
b_{T-t}^{\vartheta,h}(X_t)\,dt,
\qquad t\in[0,T] 
$ \qquad \big(Prop.\ \ref{PropStochPreD} \big)
\\[10pt]

\textbf{Radial SDE} &
$\displaystyle
dR_t
=
d\bar W_t^{T,\vartheta}
+
\Big(
\frac{1}{2R_t}
+
\bar b_{T-t}^{\vartheta,h}(R_t)
\Big)\,dt
$\\[10pt]

\textbf{Conditional density} &
$\displaystyle
\mathring{\mathsf d}^{\,T,\vartheta,h}_{s,t}(x,y)
\,\overset{\ref{SecondTrans}}{=}\,
\frac{(h_0^\vartheta*g_{T-t})(y)}{(h_0^\vartheta*g_{T-s})(x)}\, g_{t-s}(x,y),
\qquad 0\le s<t\le T
$\\[11pt]

\textbf{Conditional law} &
$\mathring{\mathbb{P}}_{x }^{T,\vartheta,h}
:=
\mathbb{P}_{x }^{T,\vartheta,h}[\,\cdot \mid \tau>T]$ \qquad \big(Thm.\ \ref{ThmPropertiesRelativeh}(ii) \big)
\\[8pt]

\textbf{Conditional drift} &
$\displaystyle
\mathring b_t^{\vartheta,h}(x)
=
\nabla_x\log (h_0^{\vartheta} * g_t)(x),
\qquad x\neq 0 
$ \qquad \big(Thm.\ \ref{ThmPropertiesRelativeh}(ii) \big)
\\[8pt]

\textbf{Conditional generator} &
$\displaystyle
\mathring{\mathscr{L}}^{T,\vartheta,h}_t
=
\frac12\,\Delta
+
\mathring b_{T-t}^{\vartheta,h}\cdot\nabla,
\qquad x\neq0
$
\\[8pt]

\textbf{Conditional SDE} &
$\displaystyle
dX_t
=
d\mathring W_t^{T,\vartheta}
+
\mathring b_{T-t}^{\vartheta,h}(X_t)\,dt,
\qquad t\in[0,T],
$ \qquad \big(Thm.\ \ref{ThmPropertiesRelativeh}(ii) \big)
\\[10pt]

\textbf{Survival} &
$\displaystyle
\mathbb{P}_{x}^{T,\vartheta,h}[\tau>T]
=
\frac{(h_0^\vartheta*g_T)(x)}{h_T^\vartheta(x)}
$ \qquad \big(Thm.\ \ref{ThmPropertiesRelativeh}(i) \big)
\\[10pt]

\textbf{Hit-time given hit} &
$\displaystyle
\frac{\mathbb{P}_x^{T,\vartheta,h}\!\big[\,\tau\in dt\,\big|\,\mathcal{O}_T\big]}{dt}
\,=\,
\frac{-\partial_t\,(h_{T-t}^\vartheta*g_t)(x)}{h_T^{\vartheta}(x)-(h_0^{\vartheta} * g_T)(x)}\, $ \qquad \big(Thm.\ \ref{ThmPropertiesRelativeh}(iv) \big)
\\[-2pt]
\end{tabular}
}%
};
\end{tikzpicture}
\caption{Here $R_t:=|X_t|$ is the radial process and $\tau:=\inf\{t\ge0:X_t=0\}$. The one-dimensional Brownian motion $\bar W^{T,\vartheta}$ is defined by radial projection of $W^{T,\vartheta}$ via $d\bar W_t^{T,\vartheta}=\frac{X_t}{|X_t|}\cdot dW_t^{T,\vartheta}$ for $t<\tau$, and the radial drift is the corresponding projection $\bar b_t^\vartheta(r):=\frac{x}{|x|}\cdot b_t^\vartheta(x)$ for $x\neq0$ with $|x|=r$.}
\label{FigBasicPropertiesPD}
\end{figure}

\subsection{Structure of the paper}
The remainder of the paper is organized as follows.

\begin{itemize}

\item In Section~\ref{ModelFormulation} we formulate an abstract model for planar diffusions with a point interaction on a finite time horizon and introduce admissible driving families together with standing hypotheses. In particular, Sections~\ref{MarkoveFamily}--\ref{SectionOriginEventD} describe the construction of the associated Markov family, the stochastic dynamics, and the submartingale characterization.

\item In Section~\ref{ExamplesDiffusion} we apply this abstract framework formally to the ground-state diffusion of~\cite{Chen2} and to some measure-driven diffusions, including the Lebesgue-driven diffusion of~\cite{CM}. Figure~\ref{GStFigure} provides a formal summary of the basic characteristics of the ground-state diffusion and its survival-conditioned law on a finite time horizon. A comparative summary of the laws obtained in this formal setting, their associated driving families, and the corresponding conditional laws is given in Figure~\ref{LawsRelations}.

\item  Sections~\ref{MarkoveFamilyProof}--\ref{SectionOriginEventDProof} contain the proofs of the results stated in Sections~\ref{MarkoveFamily}--\ref{SectionOriginEventD}, respectively.

\item The Appendix~\ref{AuxLemmas} contains auxiliary analytic lemmas.

\end{itemize}

\begin{notation}
If a function $F:\R^2\to\R$ is radially symmetric,
we write $\bar F:[0,\infty)\to\R$ for its radial representative, defined by $\bar F(r):=F(x)$ for any $x\in\R^2$ such that $ |x|=r$. When no ambiguity arises, we implicitly identify $F(x)$ with $\bar F(|x|)$.
\end{notation}

\section{Model formulation: point-interaction planar diffusions} \label{ModelFormulation}

We now isolate the analytic structure underlying planar diffusions with a point interaction at the origin and formulate an abstract class of driving families suitable for a finite-horizon Doob transform construction. Rather than fixing a specific stochastic differential equation, we take the driving family as the primary object and encode the properties required for the resulting diffusion to exhibit a point interaction.

\begin{definition}\label{SpaceTimeHEFDef}
Fix $\vartheta>0$ and a finite time horizon $T>0$.
For each $t\in[0,T]$, let $h_t^\vartheta:\R^2\to(0,\infty]$ be a function satisfying
the following properties.
\begin{enumerate}[(i)]

\item
The function $h_t^\vartheta$ is $C^2$ on $\R^2\setminus\{0\}$, radially symmetric,
and the radial profile $\bar h_t^\vartheta:(0,\infty)\to(0,\infty)$ is decreasing
and satisfies:

\begin{itemize}
\item
There exists a continuous function $c^{T,\vartheta}:[0,T]\to[0,\infty)$ such that
\begin{align}\label{hSmallaAsympt}
\bar h_t^\vartheta(a)
\,\sim\,
c^{T,\vartheta}(t)\,\log\frac{1}{a} \,,
\qquad\text{as }a\downarrow0,\quad t\in(0,T]  \,.
\end{align}
\item
There exists a radially symmetric function $\Psi^{T,\vartheta}:\R^2\to(0,\infty)$,
independent of $t$, with radial profile $\bar\Psi^{T,\vartheta}$ such that for
every $L>0$ there exist constants
$0<c_{L}^{T,\vartheta}\le C_{L}^{T,\vartheta}<\infty$ for which
\begin{align}\label{hInfinityAsympt}
c_{L}^{T,\vartheta}\,\bar\Psi^{T,\vartheta}(a)
\,\le\,
\bar h_t^\vartheta(a)
\,\le\,
C_{L}^{T,\vartheta}\,\bar\Psi^{T,\vartheta}(a) \,,
\qquad a\ge1 \,,\quad t\in(0,T] \,,\quad \vartheta t\le L \,.
\end{align}
\end{itemize}

\item
The family $\{h_t^\vartheta\}_{t\in[0,T]}$ is space-time harmonic with respect to the family of integral kernels $\{\mathsf{K}_t^\vartheta\}_{t\in[0,T]}$, in the sense that
\begin{align}\label{hspaceTimeHarmonic}
\int_{\R^2}
\mathsf{K}_{t-s}^{\vartheta}(x,y)\,h_{T-t}^{\vartheta}(y)\,dy
\,=\,
h_{T-s}^{\vartheta}(x) \,,
\qquad 0\le s<t\le T \,,\ x\in\R^2 \,.
\end{align}

\item
For all $(x,t)\in(\R^2\setminus\{0\})\times (0,T)$, the function $h_t^\vartheta$
satisfies the diffusion equation
\begin{align}\label{DEh}
\partial_t h_t^\vartheta(x)
\,=\,
\frac12\,\Delta h_t^\vartheta(x) \,.
\end{align}

\item
The map $(x,\vartheta,t)\mapsto h_t^\vartheta(x)$ is jointly continuous on
$(\R^2\setminus\{0\})\times(0,\infty)\times[0,T]$.
\end{enumerate}
We call any family $h^\vartheta=\{h_t^\vartheta\}_{t\in[0,T]}$ satisfying {\rm(i)}-{\rm(iv)} a \emph{pre-admissible driving family} for $\{\mathsf{K}_t^\vartheta\}_{t\in[0,T]}$.

\end{definition}

Notice that the class of pre-admissible driving families is stable under
positive linear combinations: if $h^{\vartheta}$ and $\hbar^{\vartheta}$ satisfy
the above definition, then so does $\alpha\,h^{\vartheta}+\beta\,\hbar^{\vartheta}$
for any constants $\alpha,\beta\ge 0$, provided that the resulting family remains
strictly positive. Two non-trivial examples satisfying the above definition are
the ground-state-driven family
$h_t^{\vartheta,\textup{GSt}}(x):=e^{\vartheta t}
K_0(\sqrt{2\vartheta}\,|x|)$, where $K_0$ denotes the modified Bessel function given by~\eqref{DefBessel}, and the Lebesgue-driven family
$h_t^{\vartheta,\textup{Leb}}(x):=1+ \mathsf{V}_t^{\vartheta,\textup{Leb}}(x)$, where
$\mathsf{V}_t^{\vartheta,\textup{Leb}}$ is the Volterra-type kernnel defined in~\eqref{DefH}; see Sections~\ref{GStDiffusion} and~\ref{LebDiffusion}.

For the remainder of the paper, we fix a pre-admissible driving family $h^\vartheta=\{h_t^\vartheta\}_{t\in[0,T]}$ and let $\mathlarger{\mathsf d}_{s,t}^{T,\vartheta}:\R^2\times\R^2\to(0,\infty]$ denote the associated Doob transform transition densities defined in~\eqref{FirstTrans} for all $0\le s<t\le T$. Although $\mathlarger{\mathsf d}_{s,t}^{T,\vartheta}$ depends on the choice of the driving family, this dependence is suppressed throughout unless multiple driving families are under consideration, in which case we write $\mathlarger{\mathsf d}_{s,t}^{T,\vartheta,h}$. The same convention applies to related objects, such as the vector field $b_t^{\vartheta}$ and the probability measures $\mathbb P_x^{T,\vartheta}$. We also adopt the following analytic conventions, which will be used throughout.

\begin{convention}\label{ConventionAnalytic}
We implicitly assume that all differentiations under the integral sign and
integrations by parts with respect to the spatial variable are justified, and
that boundary terms at infinity and at the origin vanish. In addition, for every
$\varphi\in C_c^2(\R^2)$,
\[
\lim_{h\downarrow 0} \,\int_{\R^2} \,
\mathlarger{\mathsf{d}}_{t,t+h}^{T,\vartheta}(x,y)\,\varphi(y)\,dy
 \,= \,
\varphi(x) \,,
\qquad t\in[0,T) \,,\ x\in\R^2\setminus\{0\} \,.
\]
\end{convention}
The proof of the following lemma follows from the semigroup property~\eqref{SemigroupP} of the family of integral kernels $\{\mathsf{K}_t^{\vartheta}(x,y)\}_{t\in[0,T]}$, together with the space-time harmonicity relation~\eqref{hspaceTimeHarmonic} of the driving family $h^\vartheta$. We therefore omit the proof.

\begin{lemma}\label{LemTranKern}
Fix $T,\vartheta>0$. Then $\{\mathlarger{\mathsf{d}}_{s,t}^{T,\vartheta}\}_{0\le s<t\le T}$ forms a family of transition probability densities on $[0,T]$. In particular, for all $0\le s<u<t\le T$ and $x,y\in\R^2$,
\[
\int_{\R^2} \,
\mathlarger{\mathsf{d}}_{s,u}^{T,\vartheta}(x,z)\,
\mathlarger{\mathsf{d}}_{u,t}^{T,\vartheta}(z,y)\,dz
 \,= \,
\mathlarger{\mathsf{d}}_{s,t}^{T,\vartheta}(x,y) \,,
\qquad
\int_{\R^2} \,
\mathlarger{\mathsf{d}}_{s,t}^{T,\vartheta}(x,y)\,dy
 \,= \,1 \,.
\]
\end{lemma}

\subsection{Markov family} \label{MarkoveFamily}

We now construct, for a fixed pre--admissible driving family $h^\vartheta$ and the associated Doob--transformed transition densities $\{\mathlarger{\mathsf d}_{s,t}^{T,\vartheta}\}_{0\le s<t\le T}$ introduced above, the corresponding finite-horizon Markov family on the path space $\boldsymbol{\Upsilon} := C([0,T];\R^2)$ of continuous $\R^2$--valued paths on the finite time interval $[0,T]$. Endowed with the metric
\[
d(p,p')
 \,:= \,
\sum_{n=1}^{\infty} \,\frac{1\wedge d_n(p,p')}{2^n} \,,
\qquad
d_n(p,p')
 \,:= \,
\sup_{t\in[0,n\wedge T]}|p(t) \,- \,p'(t)| \,,
\]
the space $\boldsymbol{\Upsilon}$ is Polish. We write
$\mathcal B(\boldsymbol{\Upsilon})$ for the associated Borel $\sigma$--algebra,
denote by $X=\{X_t\}_{t\in[0,T]}$ the canonical coordinate process defined by
$X_t(p):=p(t)$, and let $\{\mathcal{F}_t\}_{t\in[0,T]}$ be the filtration generated
by $X$.


For $t>0$ and $x\in\R^2\setminus\{0\}$, define the map
$\mathsf{H}_t^\vartheta:\R^2\setminus\{0\}\to\R^2$ by
\[
\mathsf{H}_t^\vartheta(x)
 \,:= \,
\frac{(\hat h_0^\vartheta * g_t)(x)}{h_t^\vartheta(x)} \,,
\qquad\text{where}\qquad
\hat h_0^\vartheta(x) \,:= \,x\,h_0^\vartheta(x) \,.
\]
Since $h_t^\vartheta(x)\to\infty$ as $x\to0$ due to the logarithmic singularity
\eqref{hSmallaAsympt}, while $(\hat h_0^\vartheta*g_t)(x)$ is finite for every
fixed $t>0$, it follows that $\mathsf{H}_t^\vartheta(x)\to0$ as $x\to0$; we therefore
set $\mathsf{H}_t^\vartheta(0):=0$. Using the identity
$\nabla_x g_t(x-y)=-t^{-1}(x-y)g_t(x-y)$,
it follows that
\(
\nabla_x(h_0^\vartheta * g_t)(x)
=
- t^{-1}\!\int_{\R^2}(x-y)\,h_0^\vartheta(y)\,g_t(x-y)\,dy,
\)
and hence a simple rearrangement yields the equivalent representation
\begin{align}\label{DefUpsilonD}
\mathsf{H}_t^\vartheta(x)
\,=\,
\frac{(h_0^\vartheta * g_t)(x)}{h_t^\vartheta(x)}\,x
\,+\,
t\,\frac{\nabla_x(h_0^\vartheta * g_t)(x)}{h_t^\vartheta(x)} \,.
\end{align}
Moreover, $\mathsf{H}_t^\vartheta$ is radially symmetric as a vector field, and its
radial profile $\bar{\mathsf{H}}_t^\vartheta:(0,\infty)\to\R$ is given by
\begin{align}\label{DefUpsilonDRadial}
\bar{\mathsf{H}}_t^\vartheta(|x|)
 \,= \,
\frac{(h_0^\vartheta*g_t)(x)}{h_t^\vartheta(x)}
 \,+ \,
t\,\frac{x\cdot\nabla_x(h_0^\vartheta*g_t)(x)}
{|x|^2\,h_t^\vartheta(x)} \,,
\qquad x\neq0  \,.
\end{align}
The family $\mathsf{H}^\vartheta=\{\mathsf{H}_t^\vartheta\}_{t \in [0,T]}$ admits a well-defined Jacobian on
$\R^2\setminus\{0\}$ with a simple radial-tangential structure
(see~\eqref{JacobianRepresentation}) and uniform eigenvalue bounds, which will be
used repeatedly in the sequel to control quadratic variations. A direct computation using Lemma~\ref{LemmaSpaceTimeHarmonicHath}  shows that $\{\mathsf{H}_t^\vartheta\}_{t \in [0,T]}$ is space-time harmonic
with respect to the family of transformed kernels
$\mathlarger{\mathsf d}^{T,\vartheta}
=\{\mathlarger{\mathsf d}_{s,t}^{T,\vartheta}\}_{0\le s<t\le T}$ defined in~\eqref{FirstTrans}, in the sense that
\begin{align}\label{SpaceTimeharmonicD}
\int_{\R^2} \mathlarger{\mathsf{d}}_{s,t}^{T,\vartheta}(x,y)\,\mathsf{H}_{T-t}^\vartheta(y)\,dy
 \,= \,
\mathsf{H}_{T-s}^\vartheta(x) \,,\qquad 0\le s<t\le T \,, x\in\R^2 \,.    
\end{align}

The next lemma which we prove in Section~\ref{LemmaKolmogorovDProof} is used to obtain the moment estimate required for
Kolmogorov's continuity criterion for the process $Y_t^{T, \vartheta} := \mathsf{H}_{T-t}^{\vartheta}(X_t)$.
\begin{lemma} \label{LemmaKolmogorovD}
Fix $T>0$ and $\vartheta>0$. For every $L>0$ there exists a constant $\mathbf{c}_L>0$
such that for all $x\in\R^2$, $m\in\mathbb{N}_0$, $0\le s<t\le T$, and
$\vartheta T\le L$,
\begin{align*}
    \int_{\R^2}  \,
       \mathlarger{\mathsf{d}}_{s,t}^{T,\vartheta}(x,y)\,
       \big|
           \mathsf{H}_{T-t}^{\vartheta}(y)
            \,-  \, \mathsf{H}_{T-s}^{\vartheta}(x)
       \big|^{2m}\,
    dy
    \,\le\,
    \mathbf{c}_L^{\,m}\,(m!)^2\,(t-s)^m \,.
\end{align*}
\end{lemma}

To obtain a path measure under which the coordinate process $X$ has continuous
sample paths, we must transfer the regularity established for the transformed
process $Y^{T,\vartheta}$ back to $X$. This requires the transformation maps
$\mathsf{H}_t^\vartheta$ to admit continuous, locally Lipschitz inverses, uniformly
over finite time horizons. We record this requirement in
Hypothesis~\ref{H2} below. More generally, to keep the presentation flexible and
to separate structural considerations from model-specific analytic estimates,
we introduce such hypotheses only when they become relevant. These assumptions
encode quantitative regularity properties that are not automatic consequences
of the abstract framework and must be verified in each concrete setting.

\begin{hypothesis}\label{H2}
Fix $T>0$ and $\vartheta>0$.
For each $t\in[0,T]$, the map
$\mathsf{H}_t^\vartheta:\R^2\setminus\{0\}\to\R^2\setminus\{0\}$ is a bijection with inverse $\mathsf{G}_t^\vartheta$. The families
\[
(t,x) \,\mapsto  \, \mathsf{H}_t^\vartheta(x)
\quad\text{and}\quad
(t,y) \,\mapsto  \,\mathsf{G}_t^\vartheta(y)
\]
are jointly continuous on $[0,T]\times(\R^2\setminus\{0\})$, and for every compact
set $K\subset\R^2\setminus\{0\}$ there exists $L_K>0$ such that
\[
|\mathsf{G}_t^\vartheta(y_1) \,- \,\mathsf{G}_t^\vartheta(y_2)|
 \,\le  \, L_K|y_1 \,- \,y_2| \,,
\qquad y_1,y_2\in K,\ t\in[0,T] \,.
\]
Moreover, defining $\mathsf{H}_t^\vartheta(0):=0$, the map
$(t,x)\mapsto\mathsf{H}_t^\vartheta(x)$ extends jointly continuously to
$(0,T]\times\R^2$.
\end{hypothesis}

The proof of the next proposition is in Section~\ref{PropCPDProof}.

\begin{proposition}\label{PropCPD}
Fix $T,\vartheta>0$ and $x\in\R^2\setminus \{0\}$, and assume Hypothesis~\ref{H2}.
There exists a unique probability measure $\mathbb{P}^{T,\vartheta}_{x}$
on $\big(\boldsymbol{\Upsilon},\mathcal{B}(\boldsymbol{\Upsilon})\big)$, where
$\boldsymbol{\Upsilon}:=C([0,T];\R^2)$, under which the coordinate process
$\{X_t\}_{t\in[0,T]}$ has initial distribution $\delta_x$ and is Markov with
transition density function $\mathlarger{\mathsf{d}}_{s,t}^{T,\vartheta}$
with respect to the filtration $\{\mathcal{F}_t\}_{t\in[0,T]}$. Moreover, the family $\{\mathbb{P}^{T,\vartheta}_{x}\}_{(x,\vartheta,T)}$
depends continuously on $(x,\vartheta,T)\in\R^2\times(0,\infty)\times(0,\infty)$
with respect to the weak topology on probability measures on $\boldsymbol{\Upsilon}$.
Finally, as $\vartheta\searrow 0$, $\mathbb{P}^{T,\vartheta}_{x}$
converges weakly to the Wiener measure $\mathbb{P}_{x}$ on $\boldsymbol{\Upsilon}$.
\end{proposition}

Observe that the law $\mathbb{P}_{x}^{T,\vartheta}$ in the above proposition depends on the choice of the driving family
$h^{\vartheta}$, which is fixed throughout the sequel; as mentioned above, this
dependence is therefore suppressed in the notation. The following remark
describes the precise relationship between the corresponding path measures
arising from different choices of driving families.

\begin{remark} \label{RemarkRNDDoobTransformRelationship}
Fix $T,\vartheta>0$ and $x\in\R^2 \setminus \{0\}$.
Let $h^\vartheta=\{h_t^\vartheta\}_{t\in[0,T]}$ and
$\hbar^{\vartheta}=\{\hbar_t^{\vartheta}\}_{t\in[0,T]}$
be two families satisfying Definition~\ref{SpaceTimeHEFDef}, and denote by
$\mathbb{P}^{T,\vartheta,h}_x$ and $\mathbb{P}^{T,\vartheta,\hbar}_x$
the corresponding path measures on $\boldsymbol{\Upsilon}$.
Then the measures $\mathbb{P}^{T,\vartheta,h}_x$ and
$\mathbb{P}^{T,\vartheta,\hbar}_x$ are mutually absolutely continuous on
$\mathcal{F}_T$, and their Radon-Nikodym derivative is given by
\begin{align}\label{RNDDoobTransformRelationship}
\frac{d\mathbb{P}^{T,\vartheta,h}_x}{d\mathbb{P}^{T,\vartheta,\hbar}_x}
\bigg|_{\mathcal{F}_T}
\, = \,
\frac{h_0^\vartheta(X_T)}{\hbar_0^\vartheta(X_T)}\,
\frac{\hbar_T^\vartheta(x)}{h_T^\vartheta(x)}\, .
\end{align}
\end{remark}

\subsection{Stochastic dynamics} \label{SubsectionStochdDiff}

We now describe the stochastic dynamics associated with the finite-horizon Markov family $\{\mathbb{P}^{T, \vartheta}_{x}\}_{x\in\R^2}$ constructed above and formulate conditions under which the coordinate process admits a weak SDE representation.

Given a Borel probability measure $\mu$ on $\R^2$, we define the law
$\mathbb{P}^{T,\vartheta}_{\mu}$ with initial distribution $\mu$ by
$\mathbb{P}^{T,\vartheta}_{\mu}
:= \int_{\R^2} \mathbb{P}^{T,\vartheta}_{x}\,\mu(dx)$, and we denote by
$\mathbb{E}^{T,\vartheta}_{\mu}$ the expectation with respect to
$\mathbb{P}^{T,\vartheta}_{\mu}$.
As a preliminary step toward constructing a weak solution to the SDE associated
with the path measure $\mathbb{P}^{T,\vartheta}_{\mu}$, we introduce the usual
augmentation of the underlying $\sigma$-algebras.
For fixed $T,\vartheta>0$ and a Borel probability measure $\mu$ on $\R^2$, let
$\mathcal{N}^{T,\vartheta}_{\mu}$ denote the collection of all subsets of
$\mathbb{P}^{T,\vartheta}_{\mu}$-null sets. We then define the augmented
$\sigma$-algebras
\begin{align}\label{AugmentedD}
\mathcal{F}_t^{T,\vartheta,\mu}
\,:=\,
\sigma\!\left( \mathcal{F}_t \,\cup\, \mathcal{N}^{T,\vartheta}_{\mu} \right)
\hspace{.5cm}\text{and}\hspace{.5cm}
\mathcal{B}^{T,\vartheta}_{\mu}
\,:=\,
\sigma\!\left( \mathcal{B}(\boldsymbol{\Upsilon}) \,\cup\, \mathcal{N}^{T,\vartheta}_{\mu} \right) \,.
\end{align}
The probability measure $\mathbb{P}^{T,\vartheta}_{\mu}$ extends uniquely to the
augmented $\sigma$-algebra $\mathcal{B}^{T,\vartheta}_{\mu}$, and we use the same
notation for this extension. In the special case $\mu=\delta_x$, we adopt the
simplified notations
\[
\mathcal{F}_t^{T,\vartheta,x}  \,:=  \,\mathcal{F}_t^{T,\vartheta,\mu} \,, 
\qquad
\mathcal{B}^{T,\vartheta}_{x} \, :=  \,\mathcal{B}^{T,\vartheta}_{\mu} \,,
\qquad
\mathbb{P}^{T,\vartheta}_{x}  \,:=  \,\mathbb{P}^{T,\vartheta}_{\mu} \,.
\]
The augmented filtration
$\mathcal{F}^{T,\vartheta,\mu}
:= \{\mathcal{F}_t^{T,\vartheta,\mu}\}_{t\in[0,T]}$
is right-continuous and complete, and
$\{\mathbb{P}^{T,\vartheta}_{x}\}_{x\in\R^2}$ remains Markov with respect to $\mathcal{F}^{T,\vartheta,\mu}$; see \cite[Section~2.7(A-B)]{Karatzas}. For fixed $\varepsilon>0$  let $\{\rho_{n}^{\varepsilon\downarrow 0}\}_{n\in \mathbb{N}}$ and $\{\rho_{n}^{0\uparrow\varepsilon}\}_{n\in \mathbb{N}_0}$ be the  sequences of stopping times such that $\rho_{0}^{0\uparrow\varepsilon}=0$ and for $n\in \mathbb{N}$
\begin{align}\label{VARRHOSD}
\rho_{n}^{\varepsilon\downarrow 0}
\,:=\,\inf\big\{s\in(\rho_{n-1}^{0\uparrow\varepsilon},T] \,:\, X_s=0\big\}\,, \quad \textup{ and } \quad
\rho_{n}^{0\uparrow\varepsilon}
\,:=\,\inf\big\{s\in(\rho_{n}^{\varepsilon\downarrow 0},T] \,:\, |X_s|=\varepsilon\big\}\,,
\end{align}
where we interpret $\inf \emptyset =\infty$.  Moreover, for $t\geq 0$ let $N_t^\varepsilon
:=
\sum_{n\ge1}\mathbf{1}_{\{\rho_{n}^{\varepsilon\downarrow 0}\le t\}}$ denote the number of indices $n\ge1$ such that
$\rho_{n}^{\varepsilon\downarrow 0}\in[0,t]$.

\begin{hypothesis}\label{H3}
Fix $T,\vartheta>0$. The following conditions hold.
\begin{enumerate}[(i)]
\item 
For every Borel probability measure $\mu$ on $\R^2$,
\begin{align}
\mathbb{E}_{\mu}^{T,\vartheta}\!\bigg[
\int_0^T \big| b_{T-s}^{\vartheta}(X_s) \big|\,ds
\bigg]
 \,< \,\infty \,. \nonumber
\end{align}

\item
There exist a constant $C=C(T,\vartheta)>0$ and functions
$\psi,\eta:(0,1]\to(0,\infty)$ such that
\begin{align}\label{H3PsiCond}
\lim_{\varepsilon\downarrow 0}\,\varepsilon\,\psi(\varepsilon) \,= \,0 \,,
\qquad\text{and}\qquad
\lim_{\varepsilon\downarrow 0}\,\eta(\varepsilon) \,= \,0 \,,
\end{align}
and for every $\varepsilon\in(0,1]$ and every Borel probability measure $\mu$ on $\R^2$,
\begin{align}
&\mathbb{E}_{\mu}^{T,\vartheta}\!\big[\,N_T^{\varepsilon}\,\big]
\,\le\,
C\,\psi(\varepsilon) \,,
\label{H3Upcross}
\\
&\mathbb{E}_{\mu}^{T,\vartheta}\!\bigg[\int_0^T \mathbf{1}_{\{|X_s|\le \varepsilon\}}\,ds\bigg]
\,+\,
\mathbb{E}_{\mu}^{T,\vartheta}\!\bigg[\int_0^T \mathbf{1}_{\{|X_s|\le \varepsilon\}}\,
\big|b_{T-s}^{\vartheta}(X_s)\big|\,ds\bigg]
\,\le\,
C\,\eta(\varepsilon) \,,
\label{H3OccAndDriftSmallBall}
\end{align}
where $N_T^\varepsilon$ is the number of excursions of $X$ from $0$ to the level
$\{|x|=\varepsilon\}$ defined in~\eqref{VARRHOSD}.
\end{enumerate}
\end{hypothesis}

In the proof of the following proposition, given in
Appendix~\ref{AppendixWeakSolConstD}, no additional It\^o-admissibility assumptions for the
family of inverse maps $\{\mathsf{G}_t^\vartheta\}_{t\in(0,T]}$ are required.
Indeed, on $\{Y_s^{T,\vartheta}\neq0\}=\{X_s\neq0\}$ we have
$Y_s^{T,\vartheta}=\mathsf{H}_{T-s}^\vartheta(X_s)$, and the identities
\eqref{PDEUp1D_corr} and \eqref{PDEUp2D_corr} imply that
\[
\big(\nabla^\dagger \mathsf{G}_{T-s}^\vartheta\big)\big(Y_s^{T,\vartheta}\big)\,
\mathsf J_{T-s}^\vartheta(X_s) \,= \,I_2
\quad\text{and}\quad
\partial_s \mathsf{G}_{T-s}^\vartheta\big(Y_s^{T,\vartheta}\big)
 \,+ \,\frac12\,\Tr_2\!\big[\cdots\big]
 \,= \, b_{T-s}^\vartheta(X_s) \,,
\]
where $\Tr_2[\cdots]$ denotes the trace term appearing in~\eqref{PDEUp2D_corr}
evaluated at $(t,x)=\big(T-s,X_s\big)$ (so that $\mathsf{H}_t^\vartheta(x)=Y_s^{T,\vartheta}$).
As a consequence, the stochastic integral term in It\^o's formula for
$t\mapsto \mathsf{G}_{T-t}^\vartheta(Y_t^{T,\vartheta})$ is automatically
well-defined, while the finite-variation term reduces to
$\int_0^t b_{T-s}^\vartheta(X_s)\,ds$, whose integrability follows from~(i) above.

\begin{proposition}\label{PropStochPreD}
Fix $T,\vartheta>0$ and let $\mu$ be a Borel probability measure on $\R^2$. Assume Hypothesis~\ref{H2}. There exists an $\R^2$-valued adapted process $\{W_t^{T,\vartheta}\}_{t\in[0,T]}$ on the probability space $(\boldsymbol{\Upsilon},\mathcal{B}^{T,\vartheta}_\mu,\mathbb P^{T,\vartheta}_\mu)$ such that $W^{T,\vartheta}$ is a two-dimensional standard Brownian motion with respect to the filtration $\{\mathcal{F}_t^{T,\vartheta,\mu}\}_{t\in[0,T]}$. If, in addition, Hypothesis~\ref{H3} holds, then the coordinate process $X$ is a weak solution, under $\mathbb P^{T,\vartheta}_\mu$, of the SDE~\eqref{SDEToSolve}, with drift field $b_t^\vartheta$ given by
\eqref{DefDriftFunD}.
\end{proposition}

\subsection{Hitting behavior at the origin} \label{SectionOriginEventD}

In this section, we study the behavior of the diffusion under
$\mathbb{P}_x^{T,\vartheta}$ at the origin and analyze the structure induced by visits to the interaction point. Fix $T,\vartheta>0$. For each $t\in[0,T]$, define the function
$\mathsf{p}_t^{\vartheta}:\R^2\to[0,1]$ by
\begin{align}\label{DefPD}
\mathsf{p}_t^{\vartheta}(x)
\,:=\,
\frac{(h_0^{\vartheta}*g_t)(x)}{h_t^{\vartheta}(x)} \, .
\end{align}
The function $\mathsf{p}_t^{\vartheta}$ satisfies the PDE below for $t\geq 0$ 
\begin{align}\label{PartialForPD}
\partial_t\,\mathsf{p}_t^{\vartheta}(x)
\,=\,
\frac{1}{2}\,\Delta \mathsf{p}_t^{\vartheta}(x)
\,+\,
b_t^{\vartheta}(x)\cdot\nabla \mathsf{p}_t^{\vartheta}(x) \,,
\qquad x\in\R^2\setminus\{0\} \,,
\end{align}
where the drift field is given by
$b_t^{\vartheta}(x):=\nabla\log h_t^{\vartheta}(x)$; see~\eqref{DefDriftFunD}. The above  follows from the identities
$\partial_t h_t^{\vartheta}(x)=\frac{1}{2}\Delta h_t^{\vartheta}(x)$ and
$\partial_t (h_0^{\vartheta}*g_t)(x)=\frac{1}{2}\Delta (h_0^{\vartheta}*g_t)(x)$,
which hold on $\R^2\setminus\{0\}$. Moreover, for every $x\neq 0$ and $t>0$, the spatial gradient of
$\mathsf{p}_t^{\vartheta}$ admits the representation
\begin{align}\label{GradPD}
\nabla \mathsf{p}_{t}^{\vartheta}(x)
\,=\,
\mathsf{p}_{t}^{\vartheta}(x)\,\mathring{b}_t^{\vartheta}(x)
\,-\,
\mathsf{p}_{t}^{\vartheta}(x)\,b_t^{\vartheta}(x)\, ,
\end{align}
where $\mathring b_t^{\vartheta}(x):=\nabla\log (h_0^{\vartheta}*g_t)(x)$ is the
drift associated with the conditioned law
$\mathring{\mathbb P}_x^{T,\vartheta}$ appearing in
Theorem~\ref{ThmPropertiesRelativeh}(ii).

\begin{hypothesis}\label{H4}
Fix $T>0$ and $\vartheta>0$.  
The family $\mathsf{p}^\vartheta=\{\mathsf{p}_t^\vartheta\}_{t\in[0,T]}$
defined in~\eqref{DefPD} satisfies the following properties.
\begin{enumerate}[(i)]

\item For every Borel probability measure $\mu$ on $\R^2$,
\begin{align}\label{H4GradL2}
\mathbb{E}_\mu^{T,\vartheta}\!\bigg[
\int_0^T \,
\mathbf{1}_{\{X_s\neq 0\}}\,
\big|\nabla \mathsf{p}_{T-s}^\vartheta(X_s)\big|^2\,ds
\bigg]
 \,< \,\infty  \,.
\end{align}

\item There exists a function $\kappa:(0,1]\to(0,\infty)$ with
$\lim_{\varepsilon\downarrow 0}\kappa(\varepsilon)=0$
such that for every $\varepsilon\in(0,1]$ and every Borel probability measure $\mu$,
\begin{align}\label{H4SmallBallGrad}
\mathbb{E}_\mu^{T,\vartheta}\!\bigg[
\int_0^T \,
\mathbf{1}_{\{|X_s|\le \varepsilon\}}\,
\big|\nabla \mathsf{p}_{T-s}^\vartheta(X_s)\big|^2\,ds
\bigg]
 \,\le  \,\kappa(\varepsilon) \,.
\end{align}

\item The process $\mathbf{1}_{\{s<\tau\}}\,
\nabla\log \mathsf{p}_{T-s}^\vartheta(X_s)$ is progressively measurable and locally square-integrable with respect to
$\mathbb{P}_\mu^{T,\vartheta}$, so that the stochastic integral
\[
\int_0^t \,
\mathbf{1}_{\{s<\tau\}}\,
\nabla\log \mathsf{p}_{T-s}^\vartheta(X_s)\cdot dW_s^{T,\vartheta}
\]
is well defined for all $t\in[0,T]$.

\end{enumerate}
\end{hypothesis}

The proof of the following proposition is in Section~\ref{PropSubMartDProof}.
\begin{proposition}\label{PropSubMartD}
Fix $T,\vartheta>0$ and a Borel probability measure $\mu$ on $\R^2$. Assume Hypotheses~\ref{H2}-\ref{H4}. Let $\mathsf{p}^\vartheta = \{ \mathsf{p}_t^\vartheta \}_{t\in[0,T]}$ be the family defined in~\eqref{DefPD}, and define the process $\mathbf{S}^{T, \vartheta} =\{\mathbf{S}_t^{T, \vartheta}\}_{t\in[0,T]}$ by $\mathbf{S}_t^{T, \vartheta}:= \mathsf{p}_{T-t}^{\vartheta}(X_t) $ for $t\in[0,T]$. Then $\mathbf{S}^{T,\vartheta}$ is a bounded, continuous $\mathbb{P}^{T,\vartheta}_\mu$-submartingale with respect to the filtration $\{\mathcal{F}_t^{T, \vartheta, \mu}\}_{t \in[0,T] }$. Moreover, in the Doob-Meyer decomposition
\[
\mathbf{S}_t^{T,\vartheta}
 \,= \,
\mathbf{M}_t^{T,\vartheta}
 \,+ \,
\mathbf{A}_t^{T,\vartheta} \,,
\qquad t\in[0,T] \,,
\]
the martingale component is given by
\begin{align}\label{Martingale}
    \mathbf{M}_{t}^{T,\vartheta}
\,:=\,
\mathsf{p}_{T}^{\vartheta}(X_0)
\,+\,
\int_0^t \nabla\mathsf{p}_{T-s}^{\vartheta}(X_s)\cdot dW_{s}^{T,\vartheta},
\end{align}
and satisfies $\mathbf{M}^{T,\vartheta}\in L^2\!\big(\mathbb{P}^{T,\vartheta}_\mu\big)$.
The increasing component $\mathbf{A}^{T,\vartheta}$ is constant during excursions
of $X$ away from $0$, in the sense that
\[
\int_0^{\infty}  \,\mathbf{1}_{\{X_t\neq 0\}}\,d\mathbf{A}_t^{T,\vartheta} \,= \,0 \,,
\qquad \mathbb{P}^{T,\vartheta}_\mu\text{-a.s.}
\]
\end{proposition}

The fact that the increasing component $\mathbf{A}^{T,\vartheta}$ in the above proposition is supported on the set $\{t\in[0,T]:X_t=0\}$ indicates that $\mathbf{S}^{T,\vartheta}$ behaves as a
martingale during excursions of $X$ away from the origin. Although this
structure suggests a connection with local time at the origin, such an interpretation is not pursued here and is not required for the results of this paper.

The proof of the following theorem is in Section~\ref{SubsectionThmSubMARTD}.

\begin{theorem} \label{ThmSubMARTD}
Fix $T,\vartheta>0$ and $x\in\R^2\setminus\{0\}$.
Assume Hypotheses~\ref{H2}-\ref{H4}. Then the assertions \textup{(i)}-\textup{(iv)} of
Theorem~\ref{ThmPropertiesRelativeh} are valid under
$\mathbb{P}_x^{T,\vartheta}$.
\end{theorem}

\section{Examples: finite-horizon planar diffusions with a point interaction} \label{ExamplesDiffusion}

We now illustrate the abstract Doob-transform framework by applying it to some concrete driving families. We first revisit the ground-state diffusion introduced by Chen in~\cite{Chen2} and examine its restriction to a finite time horizon from the present perspective. Although the finite-horizon law is well defined as the restriction of the infinite-horizon diffusion, the present
framework provides an alternative way of understanding its basic properties. We do not, however, verify the hypotheses of our abstract results in this case, as the required estimates are technically delicate; instead, we apply the results formally and compare the resulting predictions with known properties of the
infinite-horizon ground-state law. These characteristics and their
survival-conditioned counterparts are summarized in Figure~\ref{GStFigure}. We also consider additional driving families generated by measures, including the Lebesgue-driven family of~\cite{CM}. The formal relationships among the resulting laws are summarized in Figure~\ref{LawsRelations}.

\subsection[Chen]{Ground-state diffusion~\cite{Chen2}: a heuristic construction via Kolmogorov continuity arguments} \label{GStDiffusion}

Fix $T,\vartheta>0$ and define the ground-state family
$h^{\vartheta, \textup{GSt}}=\{h_t^{\vartheta, \textup{GSt}}\}_{t\in[0,T]}$ by
\begin{align}\label{GStFamily}
    h_t^{\vartheta, \textup{GSt}}(x)\,:=\,e^{\vartheta t}\,K_0(\sqrt{2\vartheta}\,|x|),
    \qquad t\in[0,T],\ x\in\mathbb R^2\setminus\{0\} \,,
\end{align}
where $K_0$ denotes the modified Bessel function of the second kind of order zero defined in~\eqref{DefBessel}. Then $h^{\vartheta, \textup{GSt}}$ satisfies the assumptions of Definition~\ref{SpaceTimeHEFDef}. In particular, $K_0(x)\sim\log x^{-1}$ as $x\searrow0$, and~\eqref{hInfinityAsympt} holds with the choice $\bar\Psi^\vartheta(a)=a^{-1/2}e^{-\sqrt{2\vartheta}\,a}$ for $a\ge1$, since $K_0(x)\sim \sqrt{\frac{\pi}{2x}}\, e^{-x}$ as $x\nearrow\infty$; these asymptotics of $K_0$ can be found in~\cite[p.~136]{Lebedev}. Moreover,~\eqref{hspaceTimeHarmonic} follows from the fact that
$K_0(\sqrt{2\vartheta}|\cdot|)$ is a positive ground state of the semigroup $\{e^{\frac{t}{2}\Delta^\vartheta}\}_{t\in[0,\infty)}$, in the sense that it is an eigenfunction of the operator $e^{\frac{t}{2}\Delta^\vartheta}$ with kernel $\mathsf{K}_t^\vartheta$ associated with the eigenvalue $e^{\vartheta t}$. This property is made precise in Lemma~\ref{LemRealUnnormEigen}(iii), which implies that the function $\mathlarger{\mathsf d}_t^{\vartheta, \textup{GSt}}:\R^2\times\R^2\to(0,\infty]$
defined by
\begin{align}\label{FirstTransBessel}
\mathlarger{\mathsf d}_t^{\vartheta, \textup{GSt}}(x,y)
\,=\,
e^{-\vartheta t}\,
\frac{K_0\!\big(\sqrt{2\vartheta}|y|\big)}
     {K_0\!\big(\sqrt{2\vartheta}|x|\big)}
\,\mathsf{K}_t^{\vartheta}(x,y) \,,
\qquad t\in[0,\infty) \,,\ x,y\in\mathbb R^2\setminus\{0\} \,,
\end{align}
defines a time-homogeneous transition probability density. In other words, the family
of transition kernels $\{\mathlarger{\mathsf d}_t^{\vartheta, \textup{GSt}}\}_{t\in[0,T]}$ is
obtained as the Doob transform~\eqref{FirstTrans} of $\mathsf{K}_t^{\vartheta}$
driven by $h_t^{\vartheta, \textup{GSt}}(x)=e^{\vartheta t}K_0(\sqrt{2\vartheta}\,|x|)$. For the function $\mathsf{H}_t^{\vartheta, \textup{GSt}}:\R^2\to\R^2$ defined by
\begin{align}
\mathsf{H}_t^{\vartheta, \textup{GSt}}(x)
\,=\,
\frac{e^{-\vartheta t}}{K_0(\sqrt{2\vartheta}|x|)}
\int_{\R^2}  \,y\,K_0(\sqrt{2\vartheta}|y|)\,g_t(x-y)\,dy \,,
\nonumber
\end{align}
we have the space-time harmonic relation
$\int_{\R^2}\mathlarger{\mathsf d}_t^{\vartheta, \textup{GSt}}(x,y)\,
\mathsf{H}_{T-t}^{\textup{GSt}}(y)\,dy=\mathsf{H}_t^{\vartheta, \textup{GSt}}(x)$ for $0<t<T$ and $x\in\R^2\setminus\{0\}$.  A direct computation using Lemma~\ref{LemRealUnnormEigen} yields the following
equivalent representation
\begin{align}
\mathsf{H}_{t}^{\vartheta, \textup{GSt}} (x)
\,=\,
\frac{K_0\!\big(\sqrt{2\vartheta}|x| ,\vartheta t \big)}
     {K_0\!\big(\sqrt{2\vartheta}|x|\big)}\,
x
\,-\,
\sqrt{2\vartheta}\,t\,
\frac{K_{1}\!\big(\sqrt{2\vartheta}|x|,\vartheta t \big)}
     {K_{0}\!\big(\sqrt{2\vartheta}|x| \big)}\,
\frac{x}{|x|},
\qquad x\neq 0,
\nonumber
\end{align}
where $K_{\nu}(z,y)$ denotes the incomplete modified Bessel function of the second kind of order $\nu$ defined in~\eqref{DefIncBessel}.

However, we do not verify Hypotheses~\ref{H2}-\ref{H4} in order to establish the existence of the finite-horizon ground-state-driven planar diffusion with a point interaction at the origin. Indeed, such a diffusion exists, and its law coincides with the restriction to $\mathcal{F}_T$ of the infinite-horizon ground-state law $\mathbb{P}_x^{\vartheta, \textup{GSt}}$. Moreover, under $\mathbb{P}_x^{\vartheta, \textup{GSt}}|_{\mathcal{F}_T}$, axioms \emph{(A1)} and \emph{(A2)} of Definition~\ref{DefRelativehMotion} are satisfied; see Remark~\ref{RemarkChen2}. We do not verify \emph{(A3)} either, since the computations in this case are more delicate. For instance, using Lemma~\ref{LemRealUnnormEigen}(i), the function $\mathsf{p}_t^{\vartheta}:\R^2\to[0,1]$ defined in~\eqref{DefPD} can be written as
\begin{align}
\mathsf{p}_{t}^{\vartheta, \textup{GSt}}(x)
\,:=\,
\frac{K_0\big(\sqrt{2\vartheta}|x|,\vartheta t \big)}
     {K_0\big(\sqrt{2\vartheta}|x|\big)} \,,
\nonumber
\end{align}
and satisfies the PDE~\eqref{PartialForPD} for the drift vector $b^{\vartheta,\textup{GSt}}(x):=\nabla \log K_0\!\big(\sqrt{2\vartheta}\,|x|\big)$. Moreover, the gradient of
$\mathsf{p}_{t}^{\vartheta, \textup{GSt}}$ has the form
\begin{align}
\nabla \mathsf{p}_t^{\vartheta, \textup{GSt}}(x)
\,=\,
\frac{\nabla K_0\!\big(\sqrt{2\vartheta}|x|,\vartheta t\big)}
     {K_0\!\big(\sqrt{2\vartheta}|x|\big)}
\,-\,
\frac{K_0\!\big(\sqrt{2\vartheta}|x|,\vartheta t\big)}
     {K_0\!\big(\sqrt{2\vartheta}|x|\big)}\,
\frac{\nabla K_0\!\big(\sqrt{2\vartheta}|x|\big)}
     {K_0\!\big(\sqrt{2\vartheta}|x|\big)} \,,
\qquad x\in\R^2\setminus\{0\} \,.
\nonumber
\end{align}
Therefore, the integrability conditions in Hypotheses~\ref{H4} are not
straightforward to verify. Below, we apply Theorem~\ref{ThmPropertiesRelativeh} heuristically and compare the resulting formal observations with the
restriction of the law $\mathbb{P}_{x}^{\vartheta, \textup{GSt}}$. To this end, suppose that
$X$ is the diffusion with law $\mathbb{P}_{x}^{T,\vartheta,\textup{GSt}}$
corresponding to the transition densities~\eqref{FirstTransBessel}, so that the
associated SDE is given by
\begin{align}\label{SDEBessel}
dX_t \,=\, dW_t \,+\,
\nabla \log K_0\!\big(\sqrt{2\vartheta}\,|X_t|\big)\,dt \,,
\end{align}
where $W$ is a two-dimensional Brownian motion. Then a \textbf{formal} application of
Theorem~\ref{ThmPropertiesRelativeh} implies that, under
$\mathbb{P}_{x}^{T,\vartheta,\textup{GSt}}$, the diffusion $X$ has the following characteristics. Figure~\ref{GStFigure} summarizes these and related basic properties of the laws
$\mathbb{P}_{x}^{T,\vartheta,\textup{GSt}}$ and
$\mathbb{P}_{x}^{T,\vartheta,\textup{GSt}}\!\left[\,\cdot\,\middle|\,
\mathcal{O}_T^c\right]$, where $\mathcal{O}_T^c:=\{\tau>T\}$.

\begin{enumerate}[(i)]
\item The terminal survival probability is given by
    \begin{align}\label{SurProbBessel}
         \mathbb{P}_{x}^{T,\vartheta,\textup{GSt}}[\mathcal{O}_T^c]\,=\,\frac{K_0\big(\sqrt{2\vartheta}|x|\,,\vartheta T \big)}{ K_0\big(\sqrt{2\vartheta}|x|\big)}\,,
    \end{align}
    where $K_0(z,y)$ is the incomplete Bessel function given in~\eqref{DefIncBessel}. Notice that the above probability is the tail mass one would obtain if the hitting time $\tau$ were distributed as a $\mathrm{GIG}(0;|x|,\sqrt{2\vartheta})$ random variable; see Remark~\ref{RemarkChen2}.

\item The diffusion $X$ under
$\mathring{\mathbb{P}}_{x}^{T,\vartheta,\textup{GSt}}
:=
\mathbb{P}_{x}^{T,\vartheta,\textup{GSt}}\!\left[\,\cdot\,\middle|\,\mathcal{O}_T^c\right]$
satisfies the SDE
\[
dX_t
\,=\,
dW_t^{T,\vartheta}
\,+\,
\nabla_x\log K_0\!\big(\sqrt{2\vartheta}\,|X_t|\,,\,\vartheta (T-t)\big)\,dt \,,
\qquad t\in[0,T] \,,
\]
with respect to some two-dimensional
$\mathring{\mathbb{P}}_{x}^{T,\vartheta,\textup{GSt}}$-Brownian motion $W^{T,\vartheta}$. In particular, conditioning on the survival event $\mathcal{O}_T^c$ removes the instantaneous effect of the point interaction at the origin: the singular contribution of the delta interaction is absorbed into the conditioning, and the resulting dynamics are described by a time-inhomogeneous Doob transform of planar Brownian motion. Equivalently, the conditional law
$\mathring{\mathbb{P}}_{x}^{T,\vartheta,\textup{GSt}}$ has transition density kernels given by the Doob transform of the Gaussian heat kernel driven by $\mathring{h}_t^{\textup{GSt}}(x):=e^{\vartheta t}
K_0\!\big(\sqrt{2\vartheta}\,|x|\,,\,\vartheta t\big)$; that is, for
$0\le s<t\le T$ and $x,y\in\R^2\setminus\{0\}$,
\begin{align}\label{SecondTransBessel}
\mathlarger{\mathsf{\mathring d}}_{s,t}^{\vartheta, \textup{GSt}}(x,y)
\,=\,
e^{-\vartheta (t-s)}\,
\frac{K_0\!\big(\sqrt{2\vartheta}|y|\,,\vartheta(T-t)\big)}
     {K_0\!\big(\sqrt{2\vartheta}|x|\,,\vartheta(T-s)\big)}
\,g_{t-s}(x-y)\,.
\end{align}

\item If $\mathsf{S}$ is an $\mathcal{F}$-stopping time such that $\mathbb{P}_{x}^{T,\vartheta,\textup{GSt}}[  \mathsf{S}<\tau ]>0$, then  the stopped  process $\{X_{t\wedge \mathsf{S}} \}_{t\in [0,T]}$ has the same law under $\mathbb{P}_{x}^{T,\vartheta,\textup{GSt}}[\,\cdot\,|\,\mathsf{S}<\tau ]$ as it does under the path measure
\begin{align}
\mathring{\mathbb{E}}_{x}^{T,\vartheta,\textup{GSt}}\!\bigg[
\frac{K_0\!\big(\sqrt{2\vartheta}\,|X_{\mathsf{S}}|\big)}
     {K_0\!\big(\sqrt{2\vartheta}\,|X_{\mathsf{S}}|\,,\,\vartheta (T-\mathsf{S}\big)}\bigg]^{-1}
     \,
\frac{K_0\!\big(\sqrt{2\vartheta}\,|X_{\mathsf{S}}|\big)}
     {K_0\!\big(\sqrt{2\vartheta}\,|X_{\mathsf{S}}|\,,\,\vartheta (T-\mathsf{S}\big)}
     \, \mathring{\mathbb{P}}_{x}^{T,\vartheta,\textup{GSt}} \,.    \nonumber
\end{align}

\item Since we can write $h_t^{\vartheta, \textup{GSt}}(x)=e^{\vartheta t}\,h_0^{\vartheta, \textup{GSt}}(x)$ with
$h_0^{\vartheta, \textup{GSt}}(x):=K_0(\sqrt{2\vartheta}\,|x|)$, we have
\begin{align}
(h_0^{\vartheta, \textup{GSt}}*g_T)(x)
    \,=\,
\int_{\R^2} g_T(x-y)\, K_0\!\big(\sqrt{2\vartheta}|y|\big)\,dy
    \,=\,
e^{\vartheta T}\, K_0\big(\sqrt{2 \vartheta} |x|\,,\vartheta T\big) \,, \nonumber
\end{align}
where the last equality follows from Lemma~\ref{LemRealUnnormEigen}(i). Similarly,
$(h_{T-t}^{\textup{GSt}}*g_t)(x)=e^{\vartheta T}\,K_0\big(\sqrt{2 \vartheta} |x|\,,\vartheta t\big)$.
Hence, a formal application of~\eqref{DistTau} implies that the distribution of the
random variable $\tau$ under $\mathbb{P}_{x}^{T,\vartheta,\textup{GSt}}$, conditioned on the
event $\mathcal{O}_T:=\{\tau<T\}$, has the density
\begin{align}\label{DistTauBessel}
\frac{\mathbb{P}_{x}^{T,\vartheta,\textup{GSt}}\big[ \,\tau \in dt  \,\big|\, \mathcal{O}_T\, \big] }{dt}
\,=\,
\frac{1
}{
K_0\big(\sqrt{2\vartheta}|x|\big)
\,-\,
K_0\big(\sqrt{2\vartheta}|x|,\vartheta T\big)
}\,\frac{1}{2t}e^{-\vartheta t-\frac{|x|^2}{2t} }\, ,
\qquad 0<t<T \, .
\end{align}
Since $K_0(\sqrt{2\vartheta}|x|\,,\vartheta T)\nearrow K_0(\sqrt{2\vartheta}|x|)$ as
$T\uparrow\infty$, the above expression may be viewed as the density of the
$\mathrm{GIG}(0;|x|,\sqrt{2\vartheta})$ law truncated to $(0,T)$ and renormalized; see the following remark.

\end{enumerate}

\begin{figure}[t]
\centering
\begin{minipage}{\linewidth}
\centering
\begin{tikzpicture}[font=\small,
  card/.style={draw, rounded corners, inner sep=7pt, align=left}
]

\node[card] (CARD) {%
\parbox{0.92\linewidth}{%
\centering \textbf{\footnotemark Ground-state diffusion~\cite{Chen2} and its survival-conditioned law on $[0,T]$}\\[4pt]
\raggedright
\setlength{\tabcolsep}{0pt}
\renewcommand{\arraystretch}{1.05}
\begin{tabular}{@{}p{0.27\linewidth}p{0.65\linewidth}@{}}

\textbf{Base semigroup} &
$\displaystyle
\mathsf{K}_t^{\vartheta}(x,y)
\overset{\eqref{DefFullKer}}{:=}
g_t(x-y)
+ 2\,\pi\, \vartheta
\int_{0 < r < s < t}
g_{r}(x)\,
\nu'\!\big(\vartheta(s-r)\big)\,
g_{t-s}(y)\,ds\,dr
$
\\[9pt]

\textbf{Driving family} &
$\displaystyle
h_t^{\vartheta, \textup{GSt}}(x)
\overset{\eqref{GStFamily}}{=}
e^{\vartheta t}\,K_0\!\big(\sqrt{2\vartheta}\,|x|\big),
\qquad t\in[0,T],\ x\neq0
$
\\[9pt]

\textbf{Transition density} &
$\displaystyle
\mathsf d_{t}^{\textup{GSt}}(x,y)
\overset{\eqref{FirstTransBessel}}{=}
e^{-\vartheta t}\,
\frac{K_0\!\big(\sqrt{2\vartheta}|y|\big)}
     {K_0\!\big(\sqrt{2\vartheta}|x|\big)}\,
\mathsf{K}_t^{\vartheta}(x,y),
\qquad t\in(0,T]
$
\\[9pt]

\textbf{Law} &
$\mathbb{P}^{T,\vartheta,\textup{GSt}}_x$ on $C([0,T],\R^2)$ 
\\[8pt]

\textbf{Drift} &
$\displaystyle
b^{\vartheta,\textup{GSt}}(x)
=
\nabla \log K_0\!\big(\sqrt{2\vartheta}\,|x|\big),
\qquad x\neq0
$
\\[8pt]

\textbf{Generator} &
$\displaystyle
\mathscr{L}^{\vartheta,\textup{GSt}}
=
\frac12\,\Delta
+
b^{\vartheta,\textup{GSt}}\cdot\nabla,
\qquad x\neq0
$
\\[8pt]

\textbf{SDE} &
$\displaystyle
dX_t
=
dW_t
+
b^{\vartheta,\textup{GSt}}(X_t)\,dt,
\qquad t\in[0,T]
$
\\[10pt]

\textbf{Radial SDE} &
$\displaystyle
dR_t
=
d\bar W_t
+
\Big(
\frac{1}{2R_t}
+
\bar b^{\vartheta,\textup{GSt}}(R_t)
\Big)\,dt
$
\\[12pt]

\textbf{Survival} &
$\displaystyle
\mathbb{P}_{x}^{T,\vartheta,\textup{GSt}}[\tau>T]
\overset{\eqref{SurProbBessel}}{=}
\frac{K_0\!\big(\sqrt{2\vartheta}|x|\,,\vartheta T\big)}
     {K_0\!\big(\sqrt{2\vartheta}|x|\big)}
$
\\[11pt]

\textbf{Conditional law} &
$\displaystyle
\mathring{\mathbb{P}}_{x}^{T,\vartheta,\textup{GSt}}
\,:=\,
\mathbb{P}_{x}^{T,\vartheta,\textup{GSt}}\!\left[\,\cdot\,\middle|\,\tau>T\right]
$
\\[8pt]

\textbf{Conditional density} &
$\displaystyle
\mathsf{\mathring d}^{\vartheta,\textup{GSt}}_{s,t}(x,y)
\,\overset{\eqref{SecondTransBessel}}{=}\,
e^{-\vartheta (t-s)}\,
\frac{K_0\!\big(\sqrt{2\vartheta}|y|\,,\vartheta(T-t)\big)}
     {K_0\!\big(\sqrt{2\vartheta}|x|\,,\vartheta(T-s)\big)}
\,g_{t-s}(x-y)
$
\\[11pt]

\textbf{Conditional drift} &
$\displaystyle
\mathring b_t^{\vartheta,\textup{GSt}}(x)
=
\nabla_x\log K_0\!\big(\sqrt{2\vartheta}\,|x|\,,\vartheta t\big),
\qquad x\neq0
$
\\[8pt]

\textbf{Conditional generator} &
$\displaystyle
\mathring{\mathscr{L}}^{\,T,\vartheta,\textup{GSt}}_t
=
\frac12\,\Delta
+
\mathring b^{\vartheta,\textup{GSt}}_{T-t}\cdot\nabla,
\qquad x\neq0
$
\\[8pt]

\textbf{Conditional SDE} &
$\displaystyle
dX_t
=
dW_t^{T,\vartheta}
+
\nabla_x\log K_0\!\big(\sqrt{2\vartheta}\,|X_t|\,,\vartheta(T-t)\big)\,dt,
\qquad t\in[0,T]
$
\\[11pt]

\textbf{Hit-time given hit} &
$\displaystyle
\frac{\mathbb{P}_{x}^{T,\vartheta,\textup{GSt}}\!\big[\,\tau\in dt\,\big|\,\tau<T\big]}{dt}
\overset{\eqref{DistTauBessel}}{=}
\frac{(2t)^{-1}\,
e^{-\vartheta t-\frac{|x|^2}{2t}}}{
K_0\!\big(\sqrt{2\vartheta}|x|\big)
-
K_0\!\big(\sqrt{2\vartheta}|x|\,,\vartheta T\big)
}\,\mathbf{1}_{(0,T)}(t)
$
\\[-2pt]

\end{tabular}
}%
};
\end{tikzpicture}

\caption{Here $\bar b^{\vartheta,\textup{GSt}}(|x| ):= \frac{x}{|x|}\cdot b^{\vartheta,\textup{GSt}}(x)$ for $x\neq0$ and
$\bar W_t=\int_0^{t\wedge\tau}\frac{X_s}{|X_s|}\cdot dW_s$ for $t\in[0,T]$.
}
\footnotetext{%
We identify $\mathbb{P}_{x}^{T,\vartheta,\textup{GSt}}
=\mathbb{P}_{x}^{\vartheta, \textup{GSt}}\big|_{\mathcal{F}_T}$. The existence and
characteristics of $\mathring{\mathbb{P}}_{x}^{T,\vartheta,\textup{GSt}}$ listed
above should be understood as heuristic applications of
Theorem~\ref{ThmPropertiesRelativeh}, although some of these statements can be justified alternatively, as explained in Remark~\ref{RemarkChen2}.}
\label{GStFigure}
\end{minipage}
\end{figure}

\begin{remark}\label{RemarkChen2}
The heuristic derivations~\eqref{SurProbBessel} and~\eqref{DistTauBessel} are in fact correct and can be justified as follows.

Consider the time-homogeneous ground-state diffusion
$X=\{X_t\}_{t\in[0,\infty)}$ with its infinite-horizon law
$\mathbb{P}_{x}^{\vartheta, \textup{GSt}}$ introduced in~\cite{Chen2}. For each
$x\in\mathbb R^2\setminus\{0\}$, the process $X$ satisfies the
SDE~\eqref{SDEBessel} under $\mathbb P_x^{\vartheta, \textup{GSt}}$. Chen in~\cite{Chen2}
shows that the semigroup $\{\mathsf{K}_t^{\vartheta}\}_{t\in[0,\infty)}$
admits the following representation: for every bounded measurable function
$\varphi:\mathbb R^2\to\mathbb R$ and every $x\neq 0$,
\begin{align}
\int_{\mathbb R^2}
\mathsf{K}_t^{\vartheta}(x,y)\,\varphi(y)\,dy
\,=\,
\mathbb E_x^{\vartheta, \textup{GSt}}\!\bigg[
\frac{e^{\vartheta t}\,
K_0\!\big(\sqrt{2\vartheta}\,|x|\big)}
{K_0\!\big(\sqrt{2\vartheta}\,|X_t|\big)}
\,\varphi(X_t)
\bigg]\, . \nonumber
\end{align}
In particular, $X$ has transition density kernels of the
form~\eqref{FirstTransBessel} for all $t\in[0,\infty)$. Consequently, when restricted to the finite time interval $[0,T]$, the diffusion $X$ under the measure
$\mathbb{P}_{x}^{\vartheta, \textup{GSt}}\big|_{\mathcal{F}_T}$ satisfies axioms
\emph{(A1)} and \emph{(A2)} of Definition~\ref{DefRelativehMotion} with the
driving family
$h_t^{\vartheta, \textup{GSt}}(x)=e^{\vartheta t}K_0(\sqrt{2\vartheta}\,|x|)$.\footnote{%
The law $\mathbb{P}_{x}^{\vartheta, \textup{GSt}}|_{\mathcal{F}_T}$ agrees with the finite-horizon law $\mathbb{P}_{x}^{T,\vartheta,\textup{GSt}}$. We keep a distinct notation to emphasize that the infinite-horizon law $\mathbb{P}_{x}^{\vartheta, \textup{GSt}}$ was constructed in~\cite{Chen2} by methods different from the Kolmogorov continuity-based approach discussed here.}

Next, recall that the hitting time $\tau$ is distributed as a generalized inverse Gaussian
random variable $\textup{GIG}(0;|x|,\sqrt{2\vartheta})$ under $\mathbb{P}_{x}^{\vartheta, \textup{GSt}}$;
see, for instance, \cite[p.~884]{Donati} or \cite[p.~3206]{Chen2}. Equivalently,
\begin{align}
\frac{\mathbb{P}_{x}^{\vartheta, \textup{GSt}}[\tau\in dt]}{dt}
\,=\,
\frac{1}{K_0\big(\sqrt{2\vartheta}|x|\big)}\,
\frac{1}{2t}\,
e^{-\vartheta t-\frac{|x|^2}{2t}}\,,
\qquad t>0 \,. \nonumber
\end{align}
Fix $T>0$ and set $\mathcal O_T^c:=\{\tau>T\}$. Then
\begin{align}
\mathbb{P}_{x}^{\vartheta, \textup{GSt}}\big|_{\mathcal{F}_T}[\mathcal O_T^c]
\,=\,
\mathbb{P}_{x}^{\vartheta, \textup{GSt}}[\tau>T]
\,=\,
\frac{1}{K_0\big(\sqrt{2\vartheta}|x|\big)}
\int_T^\infty \frac{1}{2t}\,
e^{-\vartheta t-\frac{|x|^2}{2t}}\,dt
\,=\,
\frac{K_0\!\big(\sqrt{2\vartheta}|x|,\vartheta T\big)}
{K_0\!\big(\sqrt{2\vartheta}|x|\big)}\, , \nonumber
\end{align}
where the last equality follows from~\eqref{DefIncBessel}, and thus~\eqref{SurProbBessel} holds. Moreover, the conditional density of $\tau$ given $\mathcal O_T$ is obtained as
\begin{align}
\frac{\mathbb{P}_{x}^{\vartheta, \textup{GSt}}\big|_{\mathcal{F}_T}\!\big[\,\tau\in dt\,\big|\,\mathcal O_T\,\big]}{dt}
\,=\,&
\frac{1}
{\mathbb{P}_{x}^{\vartheta, \textup{GSt}}[\tau\le T]}\,
\frac{\mathbb{P}_{x}^{\vartheta, \textup{GSt}}[\tau\in dt]}{dt}\,
\mathbf{1}_{(0,T)}(t) \nonumber\\
\,=\,&
\frac{1}{1-\mathbb{P}_{x}^{\vartheta, \textup{GSt}}\big|_{\mathcal{F}_T}[\mathcal O_T^c]}\,
\frac{\mathbb{P}_{x}^{\vartheta, \textup{GSt}}[\tau\in dt]}{dt}\,
\mathbf{1}_{(0,T)}(t)\, . \nonumber
\end{align}
Substituting the values in the expressions above yields~\eqref{DistTauBessel}.

\end{remark}

\subsection{Finite-horizon measure-driven diffusions}
In this section, we consider a general class of driving families generated by measures. Fix $T,\vartheta>0$ and let $\mathsf{K}_t^{\vartheta}$ be defined as
in~\eqref{DefFullKer}. Given a positive Borel measure $\mu$ on $\R^2$, define
\begin{align}
h_t^{\vartheta, \mu}(x)
\,:=\,
\int_{\R^2}\mathsf{K}_t^{\vartheta}(x,y)\,\mu(dy) \,,
\qquad (t,x)\in[0,T]\times\R^2 \,. \nonumber
\end{align}
Then $h_t^{\vartheta, \mu}(x)>0$ for all $(t,x)\in[0,T]\times\R^2$, and the family $h^{\vartheta, \mu}=\{h_t^{\vartheta, \mu}\}_{t\in[0,T]}$ satisfies the space--time harmonicity relation~\eqref{hspaceTimeHarmonic}. Indeed, for $0\le s<t\le T$ and $x\in\R^2$, by Tonelli's theorem and the semigroup property of $\{\mathsf{K}_t^{\vartheta}\}_{t\in[0,\infty)}$,
\begin{align*}
\int_{\R^2}\mathsf{K}_{t-s}^{\vartheta}(x,y)\,h_{T-t}^{\mu}(y)\,dy
\,=\,
\int_{\R^2}\int_{\R^2}
\mathsf{K}_{t-s}^{\vartheta}(x,y)\,
\mathsf{K}_{T-t}^{\vartheta}(y,z)\,dy\,\mu(dz) 
\,=\,
\int_{\R^2}\mathsf{K}_{T-s}^{\vartheta}(x,z)\,\mu(dz)
\,=\,
h_{T-s}^{\mu}(x)\,.
\end{align*}
Consequently, the Doob transform~\eqref{FirstTrans} driven by $h^{\vartheta, \mu}$ defines a family of transition density kernels on $[0,T]$. We write $h_t^{\vartheta, \mu}:=\mathsf{K}_t^{\vartheta}\mu$, where $(\mathsf{K}_t^{\vartheta}\mu)(x) := \int_{\R^2}\mathsf{K}_t^{\vartheta}(x,y)\,\mu(dy)$. Next, for $t\in[0,T]$ and $x,y\in\R^2$, define $\mathsf v_t^{\vartheta}:\R^2\times\R^2\to[0,\infty]$ by
\begin{align}\label{DefFullKerSecondTerm}
\mathsf v_t^{\vartheta}(x,y)
\,:=\,
2\pi\,\vartheta
\int_{0<r<s<t}
g_r(x)\,
\nu'\!\big(\vartheta(s-r)\big)\,
g_{t-s}(y)\,ds\,dr  \,,
\end{align}
which is the interaction term in~\eqref{DefFullKer}, and define the Volterra-type kernel $\mathsf{V}_t^{\vartheta,\mu} :\R^2 \to [0,\infty]$ associated to $\mu$ by
\begin{align}\label{VolKernelmu}
\mathsf{V}_t^{\vartheta, \mu}(x)
\,:=\,
\int_{\R^2}\mathsf v_t^{\vartheta}(x,y)\,\mu(dy)\,. 
\end{align}
Notice that \( h_t^{\vartheta, \mu}(x) = \int_{\R^2} g_t(x-y)\,\mu(dy) + \mathsf{V}_t^{\vartheta, \mu}(x)\). In the following subsections, we consider some specific choices of the measure
$\mu$ and describe the corresponding explicit forms of $\mathsf{V}_t^{\vartheta, \mu}$ and the associated driving family $h^{\vartheta, \mu}$. We discuss the resulting planar diffusions driven by these families at a formal level.

\subsubsection[CM]{Lebesgue-driven diffusion: recovery of \cite{CM}}\label{LebDiffusion}
For $\mu(dy)=dy$, the Volterra-type kernel $\mathsf{V}_t^{\vartheta,\textup{Leb}}:\R^2\to[0,\infty]$ defined in~\eqref{VolKernelmu} can be rewritten as a single convolution in the time variable $r$ as
\begin{align}\label{DefH}
\mathsf{V}_t^{\vartheta,\textup{Leb}}(x)
\,=\,
\int_0^t \frac{e^{-\frac{|x|^2}{2r}}}{r}\,
\nu\big(\vartheta(t-r)\big)\,dr \, ,
\end{align}
for $x\in\R^2\setminus\{0\}$. The corresponding driving family
$h^{\vartheta,\mathrm{Leb}}:=\mathsf{K}_t^{\vartheta}\mathbf{1}$
is therefore given by
\begin{align}\label{Lebfamily}
h_t^{\vartheta,\mathrm{Leb}}(x)
\,=\,
1
\,+\,
\mathsf{V}_t^{\vartheta,\textup{Leb}}(x)\, ,
\qquad t\in [0,T]\,,\ x\in\R^2\setminus\{0\}\,.
\end{align}
Notice that if $h^\vartheta=\{h_t^\vartheta\}_{t\in[0,T]}$ is any family satisfying the space-time harmonicity condition~\eqref{hspaceTimeHarmonic}, then $h_0^\vartheta(x)\equiv1$ on $\R^2$ if and only if $h^\vartheta=h^{\vartheta, \mathrm{Leb}}$. Thus, the Lebesgue-driven family is the unique family satisfying~\eqref{hspaceTimeHarmonic} with unit initial profile. Moreover, $h^{\vartheta, \mathrm{Leb}}$ is a pre-admissible driving family in the sense of Definition~\ref{SpaceTimeharmonicD}. Indeed, the radial profile satisfies $\bar{h}_t^{\vartheta, \mathrm{Leb}}(r)\sim 2\,\nu(\vartheta T)\log r^{-1}$ as $r\to0$ by~\cite[Prop.~8.11]{CM}, while the asymptotic condition~\eqref{hInfinityAsympt}
follows from \cite[Lem.~8.12]{CM}. The space-time harmonicity relation
\eqref{hspaceTimeHarmonic} and the differential identity~\eqref{DEh} follow, respectively, from the semigroup property of the integral kernel $\mathsf{K}_t^{\vartheta}$ and the diffusion equation satisfied by the Gaussian density $g_t$. The family $h^{\vartheta, \mathrm{Leb}}$ further satisfies all the assumptions of Theorem~\ref{ThmSubMARTD}; for instance, the integrability condition~\eqref{H4GradL2} follows from~\cite[Lem.~3.2]{CM}, while~\eqref{H3Upcross} is verified in~\cite[Lem.~5.3]{CM}. The corresponding planar
diffusion $X$ with law $\mathbb{P}_{x}^{T,\vartheta,\mathrm{Leb}}$ exists by~\cite[Prop.~2.2]{CM} and admits the desired submartingale characterization established in~\cite[Prop.~3.4]{CM}. In this special case, the conclusions of Theorem~\ref{ThmPropertiesRelativeh}, which we record below for comparison, reduce exactly to~\cite[Thm.~2.4]{CM}.

\begin{enumerate}[(i)]
\item  $\displaystyle \mathbb{P}_{x }^{T,\vartheta, \textup{Leb}}[\mathcal{O}_T^c]\,=\,\frac{1}{1+\mathsf{V}_t^{\vartheta,\textup{Leb}}(x)}   $

\item The  coordinate process $\{X_t \}_{t\in [0,T]}$  is a two-dimensional Brownian motion with initial position $x$ under the path measure $\mathbb{P}_{x }^{T,\vartheta, \textup{Leb}}[\,\cdot \mid \mathcal{O}_T^c]$.

\item Let $\mathsf{S}$ be an $\mathcal{F}$-stopping time such that $\mathbb{P}_{x }^{T,\vartheta, \textup{Leb}}[  \mathsf{S}<\tau ]>0$. Then the stopped coordinate process $\{X_{t\wedge \mathsf{S}}\}_{t\in[0,T]}$ has the same law under $\mathbb{P}_{x }^{T,\vartheta, \textup{Leb}}\!\left[\,\cdot\,\middle|\,\mathsf{S}<\tau\right]$ as it does under the path measure
\begin{align*}
\frac{1+\mathsf{V}_{T-\mathsf{S}}^{\vartheta}\!\big(X_{\mathsf{S}}\big)}
{\mathbb{E}_x^T\!\big[1+\mathsf{V}_{T-\mathsf{S}}^{\vartheta}\!\big(X_{\mathsf{S}}\big)\big]}\,
\mathbb{P}_x^T\, ,
\end{align*}
where recall that $\mathbb{P}_x^T$ is the two-dimensional Wiener measure.

\item Using~\eqref{DefH} and a change of variables, we may write
$\mathbb{E}_x^T\!\big[\,\mathsf{V}_{T-t}^{\vartheta}\!\big(X_t\big)\,\big]
= 2\pi \int_t^{T} g_r(x)\,\nu\!\big(\vartheta (T-r)\big)\,dr$.
Consequently, the distribution of the random variable $\tau$ under
$\mathbb{P}_{x}^{T,\vartheta, \textup{Leb}}$ conditional on the event
$\mathcal{O}_T:=\{\tau < T\}$ is absolutely continuous on $(0,T)$ with density
\begin{align}
\frac{\mathbb{P}_{x}^{T,\vartheta, \textup{Leb}}\!\big[\,\tau\in dt\,\big|\,
\mathcal{O}_T\,\big]}{dt}
\,=\,
\frac{2\pi}{\mathsf{V}_T^{\vartheta,\textup{Leb}}(x)}\,
g_t(x)\,\nu\!\big(\vartheta (T-t)\big),
\qquad t\in(0,T).
\nonumber
\end{align}

\end{enumerate}

\subsubsection{Dirac--driven diffusion (pinning at the origin)}

Take $\mu=\delta_{0}$. Then the associated driving family
$h^{\vartheta,\mathrm{Dir}} := \mathsf{K}^{\vartheta}\delta_{0}$ is given by
\begin{align}\label{Dirfamily}
h_t^{\vartheta,\mathrm{Dir}}(x)
\,=\,
g_t(x)
\,+\,
\mathsf{V}_t^{\vartheta,\textup{Dir}}(x)\,,
\qquad t\in [0,T]\,,\ x\in\R^2 \,.
\end{align}
Here $\mathsf{V}_t^{\vartheta,\textup{Dir}}$ denotes the contribution of the point
interaction at the origin obtained by taking $\mu=\delta_0$ in~\eqref{VolKernelmu}, and is given explicitly by
\begin{align}
\mathsf{V}_t^{\vartheta,\textup{Dir}}(x)
\,:=\,
\int_{0}^{t}
g_{r}(x)\,
\nu^{\vartheta,\textup{Dir}}\!\big(t-r\big)\,dr \,,
\qquad t>0\,,\ x\in\R^2 \,,
\nonumber
\end{align}
where the time kernel $\nu^{\vartheta,\textup{Dir}}:(0,\infty)\to(0,\infty]$ associated with the Dirac-driven family and is given by
\begin{align}
\nu^{\vartheta,\textup{Dir}}(a)
\,:=\,
\int_{0}^{a} \frac{\vartheta\,\nu'\!\big(\vartheta(a-s)\big)}{s}\,ds \,.
\nonumber
\end{align}
Since the initial profile satisfies $h_0^{\vartheta, \mathrm{Dir}}=\delta_{0}$, that is, $h_0^{\vartheta,\mathrm{Dir}}(x)=\delta(x)$ in the distributional sense, convolution with the Gaussian heat kernel yields $(h_0^{\vartheta, \mathrm{Dir}}*g_t)(x)=g_t(x)$. The corresponding drift of the conditioned dynamics is therefore
\[
\mathring{b}_t^{\vartheta, \mathrm{Dir}}(x)
\,:=\,
\nabla_x\log\big(h_0^{\vartheta, \mathrm{Dir}}*g_t\big)(x)
\,=\,
\nabla_x\log g_t(x)
\,=\,
-\frac{x}{t}\,,
\]
which coincides with the drift of a two-dimensional Brownian bridge pinned at the origin.

Fix $x\in \R^2\backslash \{0\}$ and suppose that there exists a planar diffusion $X=\{X_t\}_{t\in[0,T]}$, with law $\mathbb{P}_{x}^{T,\vartheta, \mathrm{Dir}}$, whose
transition density kernels are obtained via the Doob transform of the family of integral kernels $\{\mathsf{K}_t^{\vartheta}\}_{t\in[0,T]}$ driven by $h^{\vartheta, \mathrm{Dir}}$. As the existence of such a process is not established \emph{a priori} in this setting, the discussion below is to be understood as a \textbf{formal} characterization. In this sense, an application of Theorem~\ref{ThmPropertiesRelativeh} yields that under $\mathbb{P}_{x}^{T,\vartheta, \mathrm{Dir}}$ the process $X$ exhibits the following characteristics.

\begin{enumerate}[(i)]
\item $\displaystyle
\mathbb{P}_{x}^{T,\vartheta, \mathrm{Dir}}[\mathcal{O}_T^c]
=
\frac{g_T(x)}{g_T(x)+\mathsf{V}_T^{\vartheta,\textup{Dir}}(x)}
$

\item
If $\mathring{\mathbb{P}}_{x}^{T,\vartheta, \mathrm{Dir}}
:=
\mathbb{P}_{x}^{T,\vartheta, \mathrm{Dir}}[\,\cdot \mid \mathcal{O}_T^c]$ denotes the
conditioning of $\mathbb{P}_{x}^{T,\vartheta, \mathrm{Dir}}$ on the event
$\mathcal{O}_T^c$, then under $\mathring{\mathbb{P}}_{x}^{T,\vartheta, \mathrm{Dir}}$
the process $\{X_t\}_{t\in[0,T]}$ satisfies the SDE
\begin{align}\label{SDEBridge}
dX_t
&\,=\,
d\mathring{W}_t^{T,\vartheta}
\,-\,
\frac{X_t}{T-t}\,dt\,,
\qquad t\in[0,T)\,, 
\end{align}
where $\mathring{W}^{T,\vartheta}$ is a
$\mathring{\mathbb{P}}_{x}^{T,\vartheta, \mathrm{Dir}}$-Brownian motion. Consequently, $\mathring{\mathbb{P}}_{x}^{T,\vartheta, \mathrm{Dir}}$ coincides with the law $\mathbb{P}_{x\to 0}^{\,T}$ of a two-dimensional Brownian bridge from $x$ to $0$ over $[0,T]$, under the survival conditioning $\mathcal{O}_T^c$. In particular, the conditional transition density under $\mathring{\mathbb P}_x^{T,\vartheta, \mathrm{Dir}}$
has the form~\eqref{SecondTrans},
\begin{align}
\mathlarger{\mathsf{\mathring{d}}}_{s,t}^{T, \mathrm{Dir}}(x,y)
\,=\,
\frac{g_{T-t}(y)}{g_{T-s}(x)}\, g_{t-s}(x-y)\,,
\qquad 0\le s<t<T\,. \nonumber
\end{align}

\item
Let $\mathring{\mathbb{P}}_{x}^{T,\vartheta, \mathrm{Dir},\mathsf{S}}
:=\mathbb{P}_{x}^{T,\vartheta, \mathrm{Dir}}\!\left[\,\cdot\,\middle|\,\mathsf{S}<\tau\right]$
denote the conditioning of $\mathbb{P}_{x}^{T,\vartheta, \mathrm{Dir}}$ on the event
$\{\mathsf{S}<\tau\}$. Then, on the event $\{\mathsf{S}<\tau\}$, the stopped process
$\{X_{t\wedge \mathsf{S}}\}_{t\in[0,T]}$ has the same law under
$\mathring{\mathbb{P}}_{x}^{T,\vartheta, \mathrm{Dir},\mathsf{S}}$ as it does under
the path measure
\begin{align}
\displaystyle 
\mathbb{E}_{x\to 0}^{\,T}\!\bigg[
\frac{g_{T-\mathsf{S}}\!\big(X_{\mathsf{S}}\big)
      \,+\,
      \mathsf{V}_{T-\mathsf{S}}^{\vartheta,\textup{Dir}}\!\big(X_{\mathsf{S}}\big)}
     {g_{T-\mathsf{S}}\!\big(X_{\mathsf{S}}\big)}
\bigg]^{-1}
\,
\displaystyle 
\frac{g_{T-\mathsf{S}}\!\big(X_{\mathsf{S}}\big)
\,+\,
\mathsf{V}_{T-\mathsf{S}}^{\vartheta,\textup{Dir}}\!\big(X_{\mathsf{S}}\big)}
     {g_{T-\mathsf{S}}\!\big(X_{\mathsf{S}}\big)}
\,\mathbb{P}_{x\to 0}^{\,T}\,
\qquad \text{on }\{\mathsf{S}<\tau\}\, . \nonumber
\end{align}

\item
The conditional distribution of the random variable $\tau$ under
$\mathbb{P}_x^{T,\vartheta,\mathrm{Dir}}$ given the event
$\mathcal O_T:=\{\tau<T\}$ is absolutely continuous on $(0,T)$, with density
\begin{align}
\frac{\mathbb{P}_{x}^{T,\vartheta,\mathrm{Dir}}\!\big[\,\tau\in dt\,\big|\,
\mathcal O_T\,\big]}{dt}
\,=\,
\frac{1}{\mathsf{V}_T^{\vartheta,\textup{Dir}}(x)}\,
g_t(x)\,\nu^{\vartheta,\textup{Dir}}(T-t)\,,
\qquad t\in(0,T)\, .
\nonumber
\end{align}
The above uses the identity $\frac{d}{dt}\,
\big(\mathsf V_{T-t}^{\vartheta,\textup{Dir}}*g_t\big)(x) = -\,g_t(x)\,\nu^{\vartheta,\textup{Dir}}(T-t)$, which in turn is obtained by computing, under the two-dimensional Wiener
measure $\mathbb{P}_x^{T}$,
\begin{align*}
\mathbb{E}_x^{T}\!\big[\mathsf{V}_{T-t}^{\vartheta,\textup{Dir}}(X_t)\big]
\,=\,
\int_t^{T} \, g_{r}(x)\, \nu^{\vartheta,\textup{Dir}}(T-r)\,dr \, .
\end{align*}

\end{enumerate}

\subsubsection{Gaussian-endpoint driven diffusion}

Fix $T,\vartheta>0$ and $\alpha>0$. Consider the measure
$\mu(dy)=g_{\alpha}(y)\,dy$ and the corresponding Gaussian-driven family $h_t^{\vartheta,\mathrm{Gau}}(x) := \int_{\R^2}\mathsf{K}_t^{\vartheta}(x,y)\,g_{\alpha}(y)\,dy$. Then the family $h^{\vartheta,\mathrm{Gau}}$ admits the decomposition
\begin{align}\label{Gaufamily}
h_t^{\vartheta,\mathrm{Gau}}(x)
\,=\,
g_{t+\alpha}(x)
\,+\,
\mathsf V_t^{\vartheta,\mathrm{Gau}}(x),
\qquad t>0,\ x\in\R^2 \setminus \{0\} \, , 
\end{align}
where the function $\mathsf V_t^{\vartheta,\mathrm{Gau}}$, obtained by Gaussian averaging in~\eqref{VolKernelmu}, has the simplified form
\begin{align}
\mathsf V_t^{\vartheta,\mathrm{Gau}}(x)
\,=\,
\int_{0}^{t}
g_{r}(x)\,\nu^{\vartheta,\mathrm{Gau}}_{\alpha}\!\big(t-r\big)\,dr,
\qquad t>0,\ x\in\R^2 \setminus \{0\} \, , \nonumber
\end{align}
with associated time-kernel $\nu^{\vartheta,\mathrm{Gau}}_{\alpha}:(0,\infty)\to(0,\infty]$
defined by
\begin{align}
\nu^{\vartheta,\mathrm{Gau}}_{\alpha}(a)
\,:=\,
\int_{0}^{a}
\frac{\vartheta\,\nu'\!\big(\vartheta(a-s)\big)}{s+\alpha}\,ds \, . \nonumber
\end{align}

Fix $x\in\R^2\setminus\{0\}$ and let $X=\{X_t\}_{t\in[0,T]}$ denote a planar diffusion
with law $\mathbb{P}_{x}^{T,\vartheta,\mathrm{Gau}}$, formally obtained as the Doob transform of the family of integral kernels $\{\mathsf{K}_t^{\vartheta}\}_{t\in[0,T]}$ driven by
$h^{\vartheta,\mathrm{Gau}}$. Since $h_0^{\vartheta,\mathrm{Gau}}=g_{\alpha}$, we have $(h_0^{\vartheta,\mathrm{Gau}}*g_t)(x) =
g_{\alpha+t}(x)$ and hence the corresponding drift of the conditioned dynamics is
\begin{align}
\mathring b_t^{\vartheta, \mathrm{Gau}}(x)
&:=
\nabla_x\log\big((h_0^{\vartheta,\mathrm{Gau}}*g_t)(x)\big)
=
\nabla_x\log\!\big(g_{\alpha+t}(x)\big)
=
-\frac{x}{\alpha+t},
\qquad t>0,\ x\in\R^2 \setminus \{0\}\, . \nonumber
\end{align}
Thus, in this \textbf{formal} setting, Theorem~\ref{ThmPropertiesRelativeh} yields the
following characteristics of $X$ under $\mathbb{P}_{x}^{T,\vartheta,\mathrm{Gau}}$.

\begin{enumerate}[(i)]

\item The terminal survival probability is given by
\begin{align}
\mathbb{P}_{x}^{T,\vartheta,\mathrm{Gau}}[\mathcal O_T^c]
\,=\,
\frac{g_{\alpha+T}(x)}{g_{\alpha+T}(x)+\mathsf V_T^{\vartheta,\mathrm{Gau}}(x)} \, . \nonumber
\end{align}

\item If
\(
\mathring{\mathbb{P}}_{x}^{T,\vartheta, \mathrm{Gau}}
:=
\mathbb{P}_{x}^{T,\vartheta,\mathrm{Gau}}\!\left[\,\cdot\,\middle|\,\mathcal O_T^c\right]
\)
denotes the conditioning of $\mathbb{P}_{x}^{T,\vartheta,\mathrm{Gau}}$ on $\mathcal O_T^c$,
then under $\mathring{\mathbb{P}}_{x}^{T,\vartheta, \mathrm{Gau}}$ the process
$X=\{X_t\}_{t\in[0,T]}$ satisfies the Brownian bridge type SDE
\begin{align}
dX_t
\,=\,
d\mathring{W}_t^{T,\mathrm{Gau}}
\,-\,
\frac{X_t}{\alpha+(T-t)}\,dt,
\qquad t\in[0,T)\,, \nonumber
\end{align}
where $\mathring{W}^{T,\mathrm{Gau}}$ is a $\mathring{\mathbb{P}}_{x}^{T,\vartheta, \mathrm{Gau}}$-Brownian motion.
In contrast to the Brownian bridge, the drift in the above SDE remains bounded as
$t\uparrow T$, so the process is not forced to hit a fixed endpoint and remains diffusive
up to time $T$. Consequently, $\mathring{\mathbb{P}}_{x}^{T,\vartheta, \mathrm{Gau}}$ coincides with
the law $\mathbb{P}_{x}^{T,\alpha\textup{-Br}}$ of a \emph{regularized Brownian bridge}
started from $x$ with smoothing parameter $\alpha>0$, given by
\begin{align}\label{RegBridgeLaw}
\mathbb{P}_{x}^{T,\alpha\textup{-Br}}(A)
\,:=\,
\int_{\R^2}
\mathbb P_{x\to y}^{\,T}(A)\,g_\alpha(y)\,dy,
\qquad A\in\mathcal F_T, 
\end{align}
where $\mathbb P_{x\to y}^{\,T}$ denotes the law of a two-dimensional Brownian bridge
from $x$ to $y$ over $[0,T]$. In particular, the conditional transition density under
$\mathring{\mathbb{P}}_{x}^{T,\vartheta, \mathrm{Gau}}$ admits the Doob transform representation~\eqref{SecondTrans}:
for all $0\le s<t<T$ and $x,y\in\R^2$,
\begin{align}
\mathlarger{\mathsf{\mathring{d}}}_{s,t}^{T,\mathrm{Gau}}(x,y)
\,=\,
\frac{g_{\alpha+T-t}(y)}{g_{\alpha+T-s}(x)}\,
g_{t-s}(x-y)\,. \nonumber
\end{align}

\item
Let $\mathring{\mathbb{P}}_{x}^{T,\vartheta, \mathrm{Gau},\mathsf{S}}
:=
\mathbb{P}_{x}^{T,\vartheta, \mathrm{Gau}}\!\left[\,\cdot\,\middle|\,\mathsf{S}<\tau\right]$
denote the conditioning of $\mathbb{P}_{x}^{T,\vartheta, \mathrm{Gau}}$ on the event
$\{\mathsf{S}<\tau\}$. Then, on the event $\{\mathsf{S}<\tau\}$, the stopped coordinate
process $\{X_{t\wedge \mathsf{S}}\}_{t\in[0,T]}$ has the same law under
$\mathring{\mathbb{P}}_{x}^{T,\vartheta, \mathrm{Gau},\mathsf{S}}$ as it does under the path measure
\begin{align}
\mathbb{E}_{x}^{T,\alpha\textup{-Br}}\!\bigg[\frac{g_{\alpha+T-\mathsf{S}}(X_{\mathsf{S}})+\mathsf V_{T-\mathsf{S}}^{\vartheta,\textup{Gau}}(X_{\mathsf{S}})}{g_{\alpha+T-\mathsf{S}}(X_{\mathsf{S}})}
\bigg]^{-1}\,
\frac{g_{\alpha+T-\mathsf{S}}(X_{\mathsf{S}})+\mathsf V_{T-\mathsf{S}}^{\vartheta,\textup{Gau}}(X_{\mathsf{S}})}{g_{\alpha+T-\mathsf{S}}(X_{\mathsf{S}})}\,
\mathbb{P}_{x}^{T,\alpha\textup{-Br}}\,,
\qquad \text{on }\{\mathsf{S}<\tau\}\, . \nonumber
\end{align}

\item
The conditional distribution of the random variable $\tau$ under
$\mathbb{P}_x^{\mathrm{Gau}}$ given $\mathcal O_T:=\{\tau\le T\}$ is absolutely
continuous on $(0,T)$, with density
\begin{align}
\frac{\mathbb{P}_x^{T,{\mathrm{Gau}}}\!\big[\,\tau\in dt\,\big|\,\mathcal{O}_T\big]}{dt}
\,=\,
\frac{1}{\mathsf V_T^{\vartheta,\textup{Gau}}(x)}\, g_t(x)\,\nu^{\vartheta,\mathrm{Gau}}_{\alpha}(T-t) \,,
\qquad t\in(0,T)\,, \nonumber
\end{align}

\end{enumerate}

\usetikzlibrary{decorations.pathreplacing}

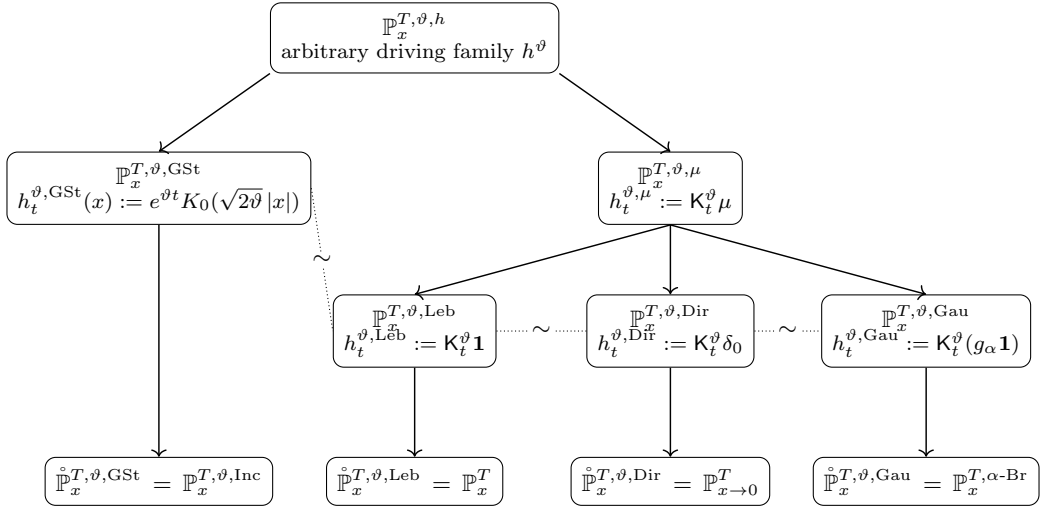
\begin{figure}[t]
\centering
\begin{tikzpicture}[font=\small,
  law/.style={draw, rounded corners, inner sep=4pt, align=center},
  arr/.style={->, line width=0.55pt}
]


\node[law] (Pgen) at (0.0, 4.4)
{$\mathbb{P}^{T,\vartheta,h}_x$\\[-2pt]
{\footnotesize arbitrary driving family $h^\vartheta$}};

\node[law] (PGSt) at (-3.4, 2.4)
{$\mathbb{P}^{T,\vartheta,\textup{GSt}}_x$\\[-2pt]
{\footnotesize $h_t^{\vartheta, \textup{GSt}}(x):=e^{\vartheta t}K_0(\sqrt{2\vartheta}\,|x|)$}};

\node[law] (Pmu) at (3.4, 2.4)
{$\mathbb{P}^{T,\vartheta,\mu}_x$\\[-2pt]
{\footnotesize $h_t^{\vartheta, \mu}:=\mathsf{K}_t^{\vartheta}\mu$}};

\node[law] (PLeb) at (0.0, 0.5)
{$\mathbb{P}^{T,\vartheta,\textup{Leb}}_x$\\[-2pt]
{\footnotesize $h_t^{\vartheta, \textup{Leb}}:=\mathsf{K}_t^{\vartheta}\mathbf{1}$}};

\node[law] (PDir) at (3.4, 0.5)
{$\mathbb{P}^{T,\vartheta,\textup{Dir}}_x$\\[-2pt]
{\footnotesize $h_t^{\vartheta, \textup{Dir}}:=\mathsf{K}_t^{\vartheta}\delta_0$}};

\node[law] (PGau) at (6.8, 0.5)
{$\mathbb{P}^{T,\vartheta,\textup{Gau}}_x$\\[-2pt]
{\footnotesize $h_t^{\vartheta, \textup{Gau}}:=\mathsf{K}_t^{\vartheta}(g_{\alpha}\mathbf{1})$}};

\node[law] (PGStC) at (-3.4, -1.5)
{$\mathring{\mathbb{P}}^{T,\vartheta,\textup{GSt}}_x \,=\, \mathbb{P}^{T,\vartheta,\textup{Inc}}_x$};

\node[law] (PLebC) at (0.0, -1.5)
{$\mathring{\mathbb{P}}^{T,\vartheta,\textup{Leb}}_x \,=\, \mathbb{P}^{T}_x$};

\node[law] (PDirC) at (3.4, -1.5)
{$\mathring{\mathbb{P}}^{T,\vartheta,\textup{Dir}}_x \,=\, \mathbb{P}^{T}_{x\to 0}$};

\node[law] (PGauC) at (6.8, -1.5)
{$\mathring{\mathbb{P}}^{T,\vartheta,\textup{Gau}}_x \,=\, \mathbb{P}^{T,\alpha\textup{-Br}}_x$};


\draw[arr] (Pgen.south west) -- (PGSt.north);
\draw[arr] (Pgen.south east) -- (Pmu.north);

\draw[arr] (Pmu.south) -- (PLeb.north);
\draw[arr] (Pmu.south) -- (PDir.north);
\draw[arr] (Pmu.south) -- (PGau.north);

\draw[arr] (PGSt.south) -- (PGStC.north);
\draw[arr] (PLeb.south) -- (PLebC.north);
\draw[arr] (PDir.south) -- (PDirC.north);
\draw[arr] (PGau.south) -- (PGauC.north);


\draw[densely dotted] (PGSt.east) -- (PLeb.west)
node[midway, fill=white, inner sep=1pt] {$\sim$};

\draw[densely dotted] (PLeb.east) -- (PDir.west)
node[midway, fill=white, inner sep=1pt] {$\sim$};

\draw[densely dotted] (PDir.east) -- (PGau.west)
node[midway, fill=white, inner sep=1pt] {$\sim$};

\end{tikzpicture}
\caption{The above diagram summarizes the relations among the laws appearing in Section~\ref{ExamplesDiffusion} and represents a \textbf{formal} application of Theorems~\ref{ThmExistenceRelativeh} and~\ref{ThmPropertiesRelativeh}; the notation is gathered in Notation~\ref{LawsNotation} for convenience.}
\label{LawsRelations}
\end{figure}

\begin{notation}\label{LawsNotation}
Figure~\ref{LawsRelations} summarizes the relations between the following laws.
\begin{itemize}
\renewcommand{\labelitemi}{---}

\item $\mathbb{P}^{T,\vartheta,h}_x$ denotes the law from Proposition~\ref{PropCPD}, with the dependence on the driving family $h$ suppressed when unambiguous.

\item
$\mathbb{P}^{T,\vartheta,\textup{Leb}}_x$,
$\mathbb{P}^{T,\vartheta,\textup{Dir}}_x$, and
$\mathbb{P}^{T,\vartheta,\textup{Gau}}_x$
denote the laws associated with the driving families
$h^{\vartheta,\textup{Leb}}$,
$h^{\vartheta,\textup{Dir}}$, and
$h^{\vartheta,\textup{Gau}}$
defined in~\eqref{Lebfamily}, \eqref{Dirfamily}, and \eqref{Gaufamily},
respectively, which are generated by the Lebesgue, Dirac, and Gaussian measures.

\item
The laws $\mathbb{P}^{T}_x$, $\mathbb{P}^{T}_{x\to y}$, and
$\mathbb{P}^{T,\alpha\textup{-Br}}_x$ denote, respectively, the two--dimensional
Wiener measure, Brownian bridge, and regularized Brownian bridge over $[0,T]$,
with the latter defined in~\eqref{RegBridgeLaw}. None of these laws depends on
$\vartheta$.

\item $\mathbb{P}^{T,\vartheta,\textup{GSt}}_x$ denotes the law associated to the ground-state-driven family $h^{\vartheta, \textup{GSt}}$ given in~\eqref{GStFamily}.

\item $\mathbb{P}^{T,\vartheta,\textup{Inc}}_x$ denotes the law corresponding to the family $h_t^{\textup{Inc}}(x):=\mathrm{e}^{\vartheta t}K_0(\sqrt{2\vartheta}\,|x|,\vartheta t)$,
where $K_0(z,y)$ is the incomplete Bessel function defined in~\eqref{DefIncBessel}.

\item $\mathring{\mathbb{P}} := \mathbb{P}[\,\cdot \mid \tau > T]$ denotes the law obtained by conditioning a given reference measure $\mathbb{P}$ on survival up to time $T$, that is, on the event that the diffusion does not hit the origin before time $T$.

\end{itemize}
\end{notation}

\section{Kolmogorov continuity criterion} \label{MarkoveFamilyProof}

We will prove Lemma~\ref{LemmaKolmogorovD} and Proposition~\ref{PropCPD} in this section.
Before doing so, we record the following observation, which will be used in the proof.

If $f:[0,\infty)\times\R^2\to\R$ is a twice-differentiable function with compact spatial support, then the forward equation~\eqref{KolmogorovForD} implies the backward identity
\begin{align}\label{KolmogorovForDBack}
\frac{\partial}{\partial h}\,
\int_{\R^2}\,  \mathlarger{\mathsf{d}}_{t,t+h}^{T,\vartheta}(x,y)\, f(t+h,y)\,dy
\Big|_{h=0}
\,=\,
\big(\mathscr{L}_x^{T-t,\vartheta} \, +\, \partial_t \big)\, f(t,x) \, ,
\end{align}
where the backward operator is given by $\mathscr{L}_x^{t,\vartheta}
:=\frac{1}{2}\Delta_x + b_{t}^\vartheta(x)\cdot\nabla_x$. Indeed, differentiating under the integral sign and using the forward
equation~\eqref{KolmogorovForD}, together with
$\mathlarger{\mathsf{d}}_{t,t}^{T,\vartheta}(x,\cdot)=\delta_x$,
yield the above backward relation.

\subsection[Properties]{Properties of $\mathsf{H}_t^\vartheta(x)$}

For $x\neq0$, let $P_x:=\frac{xx^\dagger}{|x|^2}$ denote the orthogonal projection onto $\mathrm{span}\{x\}$. A
direct calculation shows that the Jacobian matrix $\mathsf J_t^\vartheta(x)
:=
\nabla_x^\dagger \mathsf{H}_t^\vartheta(x)$ of $\mathsf{H}_t^\vartheta(x)$ given in~\eqref{DefUpsilonD}
admits the following orthogonal decomposition
\begin{align}\label{JacobianRepresentation}
\mathsf J_t^\vartheta(x)
 \, = \, 
\lambda_t^\vartheta(x)\,P_x
 \, + \, 
\lambda_t^{\vartheta,\perp}(x)\,(I_2-P_x) \, ,
\end{align}
where $\lambda_t^\vartheta(x)$ and $\lambda_t^{\vartheta,\perp}(x)$ are the
radial and tangential eigenvalues, respectively, given by,
\[
\lambda_t^{\vartheta,\perp}(x)=\bar{\mathsf{H}}_t^\vartheta(|x|),
\qquad
\lambda_t^\vartheta(x)
 \, = \, 
\bar{\mathsf{H}}_t^\vartheta(|x|)
 \, + \, 
x\cdot\nabla_x\bar{\mathsf{H}}_t^\vartheta(|x|)\,,
\]
where recall that $\bar{\mathsf{H}}_t^\vartheta$ is given in~\eqref{DefUpsilonDRadial}. The following Lemma is a direct consequence of the explicit formula
for $\bar{\mathsf{H}}_t^\vartheta$ and the logarithmic and large-distance asymptotics
\eqref{hSmallaAsympt}-\eqref{hInfinityAsympt} of the driving function
$h^\vartheta$.

\begin{lemma}\label{LemmaEigenvalueBound}
For every $L>0$ there exists a constant $C_L>0$ such that the radial and
tangential eigenvalues of the Jacobian matrix $\mathsf J_t^\vartheta(x)$ satisfy
\begin{align}\label{EigenvalueBoundD}
\lambda_T^{\vartheta,\perp}(x),\ \lambda_T^\vartheta(x)
\,\le\,
C_L,
\qquad
x\in\R^2,\quad T,\vartheta>0\ \text{with } \vartheta T \, \le \,  L  \, .
\end{align}
\end{lemma}

The proof of the following lemma follows from differentiating the explicit representation~\eqref{DefUpsilonD} of $\mathsf{H}_t^\vartheta$ and applying a standard second-order chain rule.
\begin{lemma}\label{LemmaFlowPDEBullets} 
The map $x\mapsto \mathsf{H}_t^\vartheta(x)$ satisfies the following
properties.
\begin{enumerate}[(i)]
\item 
The $\R^2$-valued map $(t,x)\mapsto \mathsf{H}_t^\vartheta(x)$ satisfies
\begin{align}\label{DiffEquationD}
\partial_t \mathsf{H}_t^\vartheta(x)
\,=\,\mathscr{L}_x^{t,\vartheta}\mathsf{H}_t^\vartheta(x),
\qquad (t,x) \, \in \, (0,T) \, \times \, (\R^2\setminus\{0\})\,,
\end{align}
where $\mathscr{L}_x^{t,\vartheta}:=\frac12\Delta_x+b_t^\vartheta(x)\cdot\nabla_x$
with $b_t^\vartheta(x)=\nabla_x\log h_t^\vartheta(x)$.

\item 
For every $G\in C^2(\R^2)$ and $(t,x)\in(0,T)\times(\R^2\setminus\{0\})$,
\begin{align}\label{ChainRuleD}
\Big(\mathscr{L}_{x}^{t,\vartheta}-\partial_t\Big)\,
G\!\big(\mathsf{H}_{t}^{\vartheta}(x)\big)
& \, = \, 
\frac12\Big[
\big(\lambda_t^\vartheta(x)\big)^2
\big(\hat x\cdot\nabla_z\big)^2 G(z)
 \, + \, 
\big(\lambda_t^{\vartheta,\perp}(x)\big)^2
\big(\hat x^\perp\cdot\nabla_z\big)^2 G(z)
\Big]\Big|_{z=\mathsf{H}_{t}^{\vartheta}(x)}\,,
\end{align}
where $\hat x:=x/|x|$ and $\hat x^\perp:=x^\perp/|x|$, with $x^\perp$ denoting the
vector obtained by a $90^\circ$ clockwise rotation of $x$, so that
$\{\hat x,\hat x^\perp\}$ forms an orthonormal basis of $\R^2$.
\end{enumerate}
\end{lemma}
In particular, since $\big(b_t^\vartheta(x)\cdot\nabla\big)\mathsf{H}_t^\vartheta(x)
=\mathsf J_t^\vartheta(x)\,b_t^\vartheta(x)$, the diffusion equation~\eqref{DiffEquationD} can be rewritten as
\begin{align}\label{UspsilonPDE}
\partial_t\mathsf{H}_t^\vartheta(x)
 \, = \, 
\frac12\,(\Delta\mathsf{H}_t^\vartheta)(x)
 \, + \, 
 \mathsf J_t^\vartheta(x)\,b_t^\vartheta(x) \, .
\end{align}

\subsection{Proof of lemma~\ref{LemmaKolmogorovD}}\label{LemmaKolmogorovDProof}

\begin{proof}
Fix $0\le s<t\le T$. Since the transition density
$\mathlarger{\mathsf{d}}_{s,t}^{T,\vartheta}(x,y)$ defined in~\eqref{FirstTrans}
satisfies $\mathlarger{\mathsf{d}}_{s,t}^{T,\vartheta}
=\mathlarger{\mathsf{d}}_{0,t-s}^{T-s,\vartheta}$, it suffices to prove the claim
for $s=0$. Let $m\in\N_0$ and define
\[
G_m^{T,\vartheta}(x,z):=\big|z-\mathsf{H}_T^\vartheta(x)\big|^{2m},
\qquad
\mathbf H_m^{T,\vartheta}(t,x)
 \, := \, 
 \int_{\R^2} \, \mathlarger{\mathsf d}^{T,\vartheta}_{0,t}(x,y)\,
G_m^{T,\vartheta}\!\big(x,\mathsf{H}_{T-t}^\vartheta(y)\big)\,dy  \, .
\]

Let $C_L>0$ be the constant from~\eqref{EigenvalueBoundD} and set
$\mathbf c_L:=2C_L^2$.
We first show that for all $x\in\R^2$, $m\in\N$, and $T,\vartheta>0$ with
$\vartheta T\le L$,
\begin{align}\label{HSUFFICIENTD}
\partial_t\, \mathbf{H}_{m}^{T,\vartheta}(t,x)
\,\le\,
\mathbf{c}_L \,m^2\,\mathbf{H}_{m-1}^{T,\vartheta}(t,x)\,.
\end{align}
Since $\mathbf H_0^{T,\vartheta}\equiv 1$ and $\mathbf H_m^{T,\vartheta}(0,x)=0$
for $m\ge1$, iterating \eqref{HSUFFICIENTD} yields
$\mathbf H_m^{T,\vartheta}(t,x)\le \mathbf c_L^{\,m}(m!)^2\,t^m$, which is exactly
the desired estimate (with $t$ replaced by $t-s$).

To show \eqref{HSUFFICIENTD}, differentiate under the integral sign and use the
forward equation~\eqref{KolmogorovForD} together with the backward identity
\eqref{KolmogorovForDBack} (i.e.\ integration by parts in $y$) to obtain
\begin{align*}
\partial_t \mathbf H_m^{T,\vartheta}(t,x)
& \, = \, 
\int_{\R^2}\mathlarger{\mathsf d}^{T,\vartheta}_{0,t}(x,y)\,
\Big(\mathscr{L}_y^{T-t,\vartheta}+\partial_t\Big)
\,G_m^{T,\vartheta}\!\big(x,\mathsf{H}_{T-t}^\vartheta(y)\big)\,dy  \, .
\end{align*}
Applying the chain rule \eqref{ChainRuleD} with $G(z)=G_m^{T,\vartheta}(x,z)$
gives
\begin{align*}
\partial_t \mathbf H_m^{T,\vartheta}(t,x)
&=
\frac12\int_{\R^2}\mathlarger{\mathsf d}^{T,\vartheta}_{0,t}(x,y)\,
\Big[
\big(\lambda_{T-t}^\vartheta(y)\big)^2
\big(\hat y\cdot\nabla_z\big)^2 G_m^{T,\vartheta}(x,z)
+
\big(\lambda_{T-t}^{\vartheta,\perp}(y)\big)^2
\big(\hat y^\perp\cdot\nabla_z\big)^2 G_m^{T,\vartheta}(x,z)
\Big]\Big|_{z=\mathsf{H}_{T-t}^\vartheta(y)}dy ,
\end{align*}
where $\hat y:=y/|y|$ and $\hat y^\perp:=y^\perp/|y|$ (for $y\neq0$).
Since $z\mapsto G_m^{T,\vartheta}(x,z)=|z-\mathsf{H}_T^\vartheta(x)|^{2m}$ is convex for
$m\ge1$, its second directional derivatives are nonnegative. Hence, using
\eqref{EigenvalueBoundD} and $\mathbf c_L=2C_L^2$, we obtain
\begin{align*}
\partial_t  \, \mathbf H_m^{T,\vartheta}(t,x)
& \, \le \, 
\frac12\int_{\R^2}\mathlarger{\mathsf d}^{T,\vartheta}_{0,t}(x,y)\,
C_L^2\Big[
\big(\hat y\cdot\nabla_z\big)^2 G_m^{T,\vartheta}(x,z)
 \, + \, 
\big(\hat y^\perp\cdot\nabla_z\big)^2 G_m^{T,\vartheta}(x,z)
\Big]\Big|_{z=\mathsf{H}_{T-t}^\vartheta(y)} \, dy \\
& \, = \, 
\frac{C_L^2}{2}\int_{\R^2}\mathlarger{\mathsf d}^{T,\vartheta}_{0,t}(x,y)\,
\Delta_z G_m^{T,\vartheta}(x,z)\Big|_{z=\mathsf{H}_{T-t}^\vartheta(y)} \, dy  \, .
\end{align*}
A direct computation yields $\Delta_z G_m^{T,\vartheta}(x,z)
=
4m^2\,|z-\mathsf{H}_T^\vartheta(x)|^{2(m-1)}
=
4m^2\,G_{m-1}^{T,\vartheta}(x,z)$, and therefore
\[
\partial_t  \, \mathbf H_m^{T,\vartheta}(t,x)
 \, \le \, 
2C_L^2\,m^2 
\int_{\R^2} \, \mathlarger{\mathsf d}^{T,\vartheta}_{0,t}(x,y)\,
G_{m-1}^{T,\vartheta}\!\big(x,\mathsf{H}_{T-t}^\vartheta(y)\big)\,dy
 \, = \, 
\mathbf c_L\,m^2\,\mathbf H_{m-1}^{T,\vartheta}(t,x) \, ,
\]
which is \eqref{HSUFFICIENTD}. This completes the proof.
\end{proof}

\subsection{Proof of Proposition~\ref{PropCPD}}\label{PropCPDProof}

\begin{proof}[Proof of Proposition~\ref{PropCPD}]
Fix $T,\vartheta>0$ and $x\in\R^2$, and assume Hypothesis~\ref{H2}.
For $0\le t_1<\cdots<t_n\le T$ and a bounded Borel function
$\varphi:(\R^2)^n\to\R$, define
\begin{align}\label{FDDdef}
\mathbb{E}^{T,\vartheta}_{x}\big[\varphi(X_{t_1},\ldots,X_{t_n})\big]
 \, := \, 
\int_{(\R^2)^n} \, 
\varphi(x_1,\ldots,x_n)\,
\prod_{k=1}^n \, 
\mathlarger{\mathsf{d}}_{t_{k-1},t_k}^{T,\vartheta}(x_{k-1},x_k)\,
dx_1\cdots dx_n \, ,
\end{align}
with $t_0:=0$ and $x_0:=x$.
The family of finite-dimensional distributions in~\eqref{FDDdef} is consistent
by Lemma~\ref{LemTranKern}. Hence, by Kolmogorov's extension theorem, there
exists a probability measure $\mathbb{P}^{T,\vartheta}_x$ on
$\big(\boldsymbol{\Upsilon},\mathcal{B}(\boldsymbol{\Upsilon})\big)$ with these
finite-dimensional marginals. In particular, $X_0=x$
$\mathbb{P}^{T,\vartheta}_x$--a.s.

Define the transformed process $Y_t^{T,\vartheta}:=\mathsf{H}_{T-t}^\vartheta(X_t)$ for
$t\in[0,T]$. By Lemma~\ref{LemmaKolmogorovD}, for every $L>0$ there exists
$\mathbf{c}_L>0$ such that, whenever $\vartheta T\le L$,
$$
\mathbb{E}^{T,\vartheta}_{x}\big[\lvert Y_t^{T,\vartheta}-Y_s^{T,\vartheta}\rvert^{2m}\big]
 \, \le \,  \mathbf{c}_L^{\,m}(m!)^2 (t-s)^m,
\qquad 0\le s<t\le T \, ,\ m\in\N \,  .
$$
Kolmogorov's continuity criterion therefore yields a modification of
$\{Y_t^{T,\vartheta}\}_{t\in[0,T]}$ with continuous sample paths. By
Hypothesis~\ref{H2}, the inverse maps $\mathsf{G}_{T-t}^\vartheta$ are continuous
and locally Lipschitz on compact subsets of $\R^2\setminus\{0\}$, uniformly in
$t\in[0,T]$. It follows that $X_t=\mathsf{G}_{T-t}^\vartheta\!\big(Y_t^{T,\vartheta}\big)$,
$t\in[0,T]$, admits a continuous modification. We henceforth identify
$\mathbb{P}^{T,\vartheta}_x$ with this continuous-path version, i.e.\ as a
probability measure on $\boldsymbol{\Upsilon}=C([0,T];\R^2)$.

The Markov property follows from the construction of the finite-dimensional
distributions. Indeed, let $0\le s<t\le T$ and let $F$ be a bounded
$\mathcal{F}_s$-measurable random variable. For any bounded Borel function
$\psi:\R^2\to\R$, the definition of the finite-dimensional distributions
\eqref{FDDdef} and the consistency relation of the kernels from
Lemma~\ref{LemTranKern} yield
$$
\mathbb{E}^{T,\vartheta}_{x}\big[F\,\psi(X_t)\big]
 \, = \, 
\mathbb{E}^{T,\vartheta}_{x}\Big[
F\int_{\R^2} \, 
\mathlarger{\mathsf{d}}_{s,t}^{T,\vartheta}(X_s,y)\,\psi(y)\,dy
\Big] \, .
$$
This identifies the conditional expectation $\mathbb{E}^{T,\vartheta}_{x}\!\big[\psi(X_t)\mid \mathcal{F}_s\big]
=
\int_{\R^2}
\mathlarger{\mathsf{d}}_{s,t}^{T,\vartheta}(X_s,y)\,\psi(y)\,dy$,
and shows that $X$ is a time-inhomogeneous Markov process with transition
density $\mathlarger{\mathsf{d}}_{s,t}^{T,\vartheta}$.

Uniqueness of $\mathbb{P}^{T,\vartheta}_x$ follows because any probability measure
on $(\boldsymbol{\Upsilon},\mathcal{B}(\boldsymbol{\Upsilon}))$ under which $X_0=x$
and $X$ is Markov with transition density $\mathlarger{\mathsf{d}}^{T,\vartheta}$
must have the finite-dimensional distributions given by~\eqref{FDDdef}.

To see the continuous dependence on parameters, let
$(x_n,T_n,\vartheta_n)\to(x,T,\vartheta)$ with $T_n,\vartheta_n>0$.
By the explicit representation of the kernels
$\mathlarger{\mathsf{d}}^{T,\vartheta}$ in~\eqref{FirstTrans},
the continuity of $h^\vartheta$ in $(x,\vartheta,t)$ from
Definition~\ref{SpaceTimeHEFDef}\emph{(iv)},
and dominated convergence justified by
Convention~\ref{ConventionAnalytic},
the finite-dimensional distributions~\eqref{FDDdef}
depend continuously on $(x,T,\vartheta)$. Moreover,
Lemma~\ref{LemmaKolmogorovD} yields Kolmogorov moment bounds that are uniform on
compact parameter ranges $\{(T,\vartheta):\vartheta T\le L\}$, and therefore imply
tightness of $\{\mathbb{P}^{T_n,\vartheta_n}_{x_n}\}_n$ on $\boldsymbol{\Upsilon}$.
Consequently, $\mathbb{P}^{T_n,\vartheta_n}_{x_n}\Rightarrow \mathbb{P}^{T,\vartheta}_{x}$
as $n\to\infty$.

Finally, as $\vartheta\searrow 0$, the kernels $\mathlarger{\mathsf{d}}^{T,\vartheta}$
converge, in the sense of finite-dimensional distributions, to the Gaussian heat
kernel transition densities, and hence
$\mathbb{P}^{T,\vartheta}_{x}\Rightarrow \mathbb{P}_x$, the Wiener measure started
from $x$, on $\boldsymbol{\Upsilon}$.
\end{proof}

\section[SDE]{A weak solution to the SDE (\ref{SDEToSolve})}\label{AppendixWeakSolConstD}
Given a Borel probability measure $\mu$ on $\R^2$, consider the filtered
probability space
\[(\boldsymbol{\Upsilon} \, ,\mathcal{B}^{T,\vartheta}_\mu \, ,
\{\mathcal{F}_t^{T,\vartheta,\mu}\}_{t\in[0,T]} \, ,\mathbb P^{T,\vartheta}_\mu)\,,
\]
where the augmented $\sigma$-algebra $\mathcal{B}^{T,\vartheta}_\mu$ and the
augmented filtration $\{\mathcal{F}_t^{T,\vartheta,\mu}\}_{t\in[0,T]}$ are defined
as in~\eqref{AugmentedD}. The following lemma and its corollary concern the
construction of the two-dimensional Brownian motion $W^{T,\vartheta}$ appearing
in Proposition~\ref{PropStochPreD}.

\begin{lemma}\label{LemmaMartID}
Fix $T,\vartheta>0$, and let $\mu$ be a Borel probability measure on $\R^2$ such that
$\int_{\R^2}|x|^2\,\mu(dx)<\infty$.
Assume Hypothesis~\ref{H2}.
Define $Y_t^{T,\vartheta}:=\mathsf{H}_{T-t}^\vartheta(X_t)$ for $t\in[0,T]$.
Then $\{Y_t^{T,\vartheta}\}_{t\in[0,T]}$ is a continuous square-integrable
$\mathbb{P}^{T,\vartheta}_\mu$-martingale with respect to
$\{\mathcal{F}_t^{T,\vartheta,\mu}\}_{t\in[0,T]}$.
Moreover, for every $u\in\R^2$, the quadratic variation of the real-valued martingale $u\cdot Y^{T,\vartheta}$ is given by
\[
\big\langle u\cdot Y^{T,\vartheta}\big\rangle_t
 \, = \, 
\int_0^t  \, \big\|\mathsf{J}_{T-s}^\vartheta(X_s)\,u\big\|_2^2\,ds \, ,
\qquad t\in[0,T] \, .
\]
\end{lemma}

\begin{proof}
Fix $t\in[0,T]$. By Minkowski's inequality we obtain the first inequality below.
\begin{align*}
\mathbb{E}_\mu^{T,\vartheta}\big[|Y_t^{T,\vartheta}|^2\big]^{1/2}
& \, \le \, 
\mathbb{E}_\mu^{T,\vartheta}\big[|Y_0^{T,\vartheta}|^2\big]^{1/2}
 \, + \, 
\mathbb{E}_\mu^{T,\vartheta}\big[|Y_t^{T,\vartheta}-Y_0^{T,\vartheta}|^2\big]^{1/2}\\
& \, \le \, 
C_L\Big(\int_{\R^2}|x|^2\,\mu(dx)\Big)^{1/2}
 \, + \, 
\Big(\int_{\R^2}\!\int_{\R^2}\mathlarger{\mathsf{d}}_{0,t}^{T,\vartheta}(x,y)\,
\big|\mathsf{H}_{T-t}^\vartheta(y)-\mathsf{H}_T^\vartheta(x)\big|^2\,dy\,\mu(dx)\Big)^{1/2}
 \, < \, \infty  \, .
\end{align*}
The second inequality uses that
$|\mathsf{H}_T^\vartheta(x)|=\lambda_T^{\vartheta,\perp}(x)\,|x|\le C_L|x|$ by Lemma~\ref{EigenvalueBoundD}.
Finiteness of the right--hand side follows from the integrability assumption on
$\mu$ and Lemma~\ref{LemmaKolmogorovD}. Hence $\{Y_t^{T,\vartheta}\}_{t\in[0,T]}$ is square-integrable.

Next, for $0\le s<t\le T$, by the Markov property of $X$ under
$\mathbb{P}_\mu^{T,\vartheta}$, we have the second equality below.
\begin{align*}
\mathbb{E}_\mu^{T,\vartheta}\!\big[Y_t^{T,\vartheta}\mid \mathcal{F}_s^{T,\vartheta,\mu}\big]
\,=\,
\mathbb{E}_\mu^{T,\vartheta}\!\big[\mathsf{H}_{T-t}^\vartheta(X_t)\mid \mathcal{F}_s^{T,\vartheta,\mu}\big]
\,=\,
\int_{\R^2}\mathsf{H}_{T-t}^\vartheta(y)\,
\mathlarger{\mathsf{d}}_{s,t}^{T,\vartheta}(X_s,y)\,dy \,
\end{align*}
By the space-time harmonic relation~\eqref{SpaceTimeharmonicD} the right--hand side equals
$\mathsf{H}_{T-s}^\vartheta(X_s)=Y_s^{T,\vartheta}$.
Thus $\{Y_t^{T,\vartheta}\}_{t\in[0,T]}$ is a
$\mathbb{P}_\mu^{T,\vartheta}$-martingale. Continuity of $Y^{T,\vartheta}$ follows from continuity of $X$ and the joint
continuity, guaranteed by Hypothesis~\ref{H2}, of the extension of
$(r,x)\mapsto\mathsf{H}_r^\vartheta(x)$ to $(0,T]\times\R^2$.

For the quadratic variation of $Y^{T,\vartheta}$, fix $u\in\R^2$ and set $M_t:=u\cdot Y_t^{T,\vartheta}$.
Then $M$ is a continuous square-integrable martingale, hence its bracket
$\langle M\rangle$ is absolutely continuous and, for a.e.\ $t\in(0,T)$,
\begin{align}\label{PiffleD}
\frac{d}{dt}\langle M\rangle_t
 \, = \, 
\lim_{h\downarrow0}\frac1h\,
\mathbb{E}_\mu^{T,\vartheta}\!\left[\big(M_{t+h}-M_t\big)^2
\mid \mathcal{F}_t^{T,\vartheta,\mu}\right]
\qquad \mathbb{P}_\mu^{T,\vartheta}\text{-a.s.}
\end{align}
Moreover, by the Markov property and time-inhomogeneity, for $h>0$,
\begin{align*}
\mathbb{E}_\mu^{T,\vartheta}\!\left[\big(M_{t+h}-M_t\big)^2
\mid \mathcal{F}_t^{T,\vartheta,\mu}\right]
&=
\int_{\R^2}\mathlarger{\mathsf{d}}_{0,h}^{\,T-t,\vartheta}(X_t,y)\,
\Big(u\cdot\mathsf{H}_{T-t-h}^\vartheta(y)
-
u\cdot\mathsf{H}_{T-t}^\vartheta(X_t)\Big)^2\,dy .
\end{align*}
Applying the backward Kolmogorov identity~\eqref{KolmogorovForDBack}
together with the chain rule identity~\eqref{ChainRuleD}, we obtain that,
for a.e.\ $t\in(0,T)$ and $\mathbb{P}_\mu^{T,\vartheta}$--a.s.\,,
\begin{align*}
\frac{d}{dt}\langle M\rangle_t
\,=\,
\big(\lambda_{T-t}^\vartheta(X_t)\big)^2
\|P_{X_t}u\|_2^2
\,+\,
\big(\lambda_{T-t}^{\vartheta,\perp}(X_t)\big)^2
\|(I-P_{X_t})u\|_2^2
\,=\,
\big\|\mathsf{J}_{T-t}^\vartheta(X_t)\,u\big\|_2^2 \,, 
\end{align*}
which yields the claimed quadratic variation identity.

\end{proof}

Note that the $2\times 2$ matrix $\mathsf{J}_t^\vartheta(x)$ admits the spectral representation~(\ref{JacobianRepresentation})
and is therefore invertible whenever $x\neq 0$ and
$\lambda_t^\vartheta(x)\neq 0$, $\lambda_t^{\vartheta,\perp}(x)\neq 0$, with
\begin{align}\label{JacobianRepresentationInverse}
    \big(\mathsf{J}_t^\vartheta(x)\big)^{-1}
\,=\,
\frac{1}{\lambda_t^\vartheta(x)}\,P_x
\,+\,
\frac{1}{\lambda_t^{\vartheta,\perp}(x)}\,(I-P_x)\,.
\end{align}
In the following corollary, this inverse is only used on the set $\{X_t\neq 0\}$,
on which the above representation applies.

\begin{corollary}\label{CorollaryMartID}
Fix $T,\vartheta>0$ and a Borel probability measure $\mu$ on $\R^2$.
Assume Hypothesis~\ref{H2}, and suppose that $\int_{\R^2}|x|^2\,\mu(dx)<\infty$.
Let the $\R^2$-valued process $\{Y_t^{T,\vartheta}\}_{t\in[0,T]}$ be defined as in
Lemma~\ref{LemmaMartID}.  Define the process
$\{W_t^{T,\vartheta}\}_{t\in[0,T]}$ by
\[
W_t^{T,\vartheta}
\,:=\,
\int_0^t \mathbf{1}_{\{X_s\neq 0\}}\,
\big(\mathsf{J}_{T-s}^{\vartheta}(X_s)\big)^{-1}\,dY_s^{T,\vartheta},
\qquad t\in[0,T]\,.
\]
Then $\{W_t^{T,\vartheta}\}_{t\in[0,T]}$ is a two-dimensional standard Brownian
motion with respect to the filtration
$\{\mathcal{F}_t^{T,\vartheta,\mu}\}_{t\in[0,T]}$ under $\mathbb{P}_\mu^{T,\vartheta}$.
Moreover, the processes $Y^{T,\vartheta}$ and $W^{T,\vartheta}$ satisfy, $\mathbb{P}_\mu^{T,\vartheta}$--a.s.\,,
\begin{align*}
dY_t^{T,\vartheta}
\,=\,
\mathsf{J}_{T-t}^{\vartheta}(X_t)\,dW_t^{T,\vartheta}\,,
\qquad t\in[0,T]\,.
\end{align*}
\end{corollary}

\begin{proof}
By Lemma~\ref{LemmaMartID}, $Y^{T,\vartheta}$ is a continuous square-integrable
$\mathbb{P}_\mu^{T,\vartheta}$-martingale. On $\{X_s\neq 0\}$ the matrix
$\mathsf{J}_{T-s}^\vartheta(X_s)$ is invertible, hence $W^{T,\vartheta}$ is a well-defined
continuous local martingale with respect to $\{\mathcal{F}_t^{T,\vartheta,\mu}\}_{t\in[0,T]}$
under $\mathbb{P}_\mu^{T,\vartheta}$. By L\'evy's characterization, it suffices to show that for each $v\in\R^2$, $\big\langle v\cdot W^{T,\vartheta}\big\rangle_t=t\|v\|_2^2$ for $t\in[0,T]$.
Set $M_t:=v\cdot W_t^{T,\vartheta}$. Since $M$ is the stochastic integral of the
predictable process
$s\mapsto \mathbf{1}_{\{X_s\neq 0\}}(\mathsf{J}_{T-s}^{\vartheta}(X_s))^{-1}v$
against the $\R^2$-valued martingale $Y^{T,\vartheta}$, we have
\[
\langle M\rangle_t
 \, = \, 
\int_0^t  \, \mathbf{1}_{\{X_s\neq 0\}}\,
\Big\langle (\mathsf{J}_{T-s}^{\vartheta}(X_s))^{-1} \, v \,,\,
d\langle Y^{T,\vartheta}\rangle_s\, (\mathsf{J}_{T-s}^{\vartheta}(X_s))^{-1} \, v
\Big\rangle \, .
\]

Moreover, $d\langle u\cdot Y^{T,\vartheta}\rangle_s=\|\mathsf{J}_{T-s}^{\vartheta}(X_s)u\|_2^2\,ds$ for every $u\in\R^2$ by Lemma~\ref{LemmaMartID},
hence the matrix-valued bracket satisfies $d\langle Y^{T,\vartheta}\rangle_s
=
\mathsf{J}_{T-s}^{\vartheta}(X_s)\,\big(\mathsf{J}_{T-s}^{\vartheta}(X_s)\big)^\dagger\,ds$. Therefore,
\begin{align*}
\langle M\rangle_t
&=
\int_0^t \mathbf{1}_{\{X_s\neq 0\}}\,
\Big\langle (\mathsf{J}_{T-s}^{\vartheta}(X_s))^{-1} \, v \,,\,
\mathsf{J}_{T-s}^{\vartheta}(X_s)\big(\mathsf{J}_{T-s}^{\vartheta}(X_s)\big)^\dagger
(\mathsf{J}_{T-s}^{\vartheta}(X_s))^{-1} \, v
\Big\rangle\,ds\\
& \, = \, 
\int_0^t  \, \mathbf{1}_{\{X_s\neq 0\}}\,
\big\langle v \,,\, v\big\rangle\,ds \nonumber \\
 &\, = \, 
\|v\|_2^2 \,  \int_0^t \,  \mathbf{1}_{\{X_s\neq 0\}}\,ds \, .
\end{align*}

It remains to show that $\int_0^t \mathbf{1}_{\{X_s=0\}}\,ds=0$
$\mathbb{P}_\mu^{T,\vartheta}$--a.s. For each $s\in(0,T]$ and $x\in\R^2$, under
$\mathbb{P}_x^{T,\vartheta}$ the random variable $X_s$ has Lebesgue density
$\mathlarger{\mathsf{d}}_{0,s}^{T,\vartheta}(x,\cdot)$, hence $\mathbb{P}_x^{T,\vartheta}[X_s=0]
=
\int_{\R^2}\mathlarger{\mathsf{d}}_{0,s}^{T,\vartheta}(x,y)\,\mathbf{1}_{\{y=0\}}\,dy
=0$.
By Tonelli's theorem,
\[
\mathbb{E}_\mu^{T,\vartheta}\!\left[\int_0^t \mathbf{1}_{\{X_s=0\}}\,ds\right]
 \, = \, 
\int_{\R^2} \, \int_0^t  \, \mathbb{P}_x^{T,\vartheta}[X_s=0]\,ds\,\mu(dx)
 \, = \, 0 \, ,
\]
so $\int_0^t \mathbf{1}_{\{X_s=0\}}\,ds=0$ $\mathbb{P}_\mu^{T,\vartheta}$--a.s.
Consequently, $\langle v\cdot W^{T,\vartheta}\rangle_t=t\|v\|_2^2$ for all
$t\in[0,T]$, and thus $W^{T,\vartheta}$ is a two-dimensional standard Brownian motion.

Finally, by definition,
$dW_t^{T,\vartheta}
=\mathbf{1}_{\{X_t\neq 0\}}(\mathsf{J}_{T-t}^{\vartheta}(X_t))^{-1}\,dY_t^{T,\vartheta}$,
so $\mathbf{1}_{\{X_t\neq 0\}}\,dY_t^{T,\vartheta}
=\mathsf{J}_{T-t}^{\vartheta}(X_t)\,dW_t^{T,\vartheta}$.
Since $\int_0^T\mathbf{1}_{\{X_t=0\}}\,dt=0$ $\mathbb{P}_\mu^{T,\vartheta}$--a.s.,
the indicator may be dropped inside stochastic integrals, yielding
$dY_t^{T,\vartheta}=\mathsf{J}_{T-t}^{\vartheta}(X_t)\,dW_t^{T,\vartheta}$
for $t\in[0,T]$, $\mathbb{P}_\mu^{T,\vartheta}$--a.s. \vspace{.3cm}
\end{proof}

\begin{notation}\label{DerivativeConvention}
For a twice continuously differentiable map
$f:\R^2\to\R^2$, we write $\nabla^{\dagger} f(x)$ for the $2\times2$ Jacobian
matrix of first derivatives at $x\in\R^2$. The second derivative $Df(x)$ is
viewed as an element of the tensor product $\R^2\otimes M^{2\times2}$, with the
convention that for $v\in\R^2$, $(v^{\dagger}\otimes I_2)\,Df(x)$ is the Hessian matrix of the real-valued function $v^{\dagger}f=v\cdot f$.
Under these conventions, $\Tr_2(Df)(x)=\Delta f(x)$ componentwise.

If $g$ denotes the local inverse of $f$ and $\nabla^{\dagger}f(x)$ is invertible,
then a direct application of the second-order chain rule yields
\begin{align}\label{PipD}
\big(D g\big)\big(f(x)\big)
\,=\,
-\Big(\nabla^{\dagger}f(x)\otimes\nabla f(x)\Big)^{-1}
\big(D f\big)(x)
\Big(I_2\otimes\nabla^{\dagger}f(x)\Big)^{-1} \, ,
\end{align}
where $I_2$ denotes the $2\times2$ identity matrix and $\nabla f(x)$ is the
transpose of the Jacobian $\nabla^{\dagger}f(x)$.
\end{notation}

\begin{proof}[Proof of Proposition~\ref{PropStochPreD}]

Let $W^{T,\vartheta}$ be the two-dimensional Brownian motion constructed in
Corollary~\ref{CorollaryMartID}. For each $t\in(0,T]$, let
$\mathsf{G}_t^\vartheta:\R^2\setminus\{0\}\longrightarrow \R^2\setminus\{0\}$
denote the inverse of $\mathsf{H}_t^\vartheta$ from Hypothesis~\ref{H2}. Then, for
every $t\in[0,T]$,
\[
X_t \, = \, \mathsf{G}_{T-t}^\vartheta\!\big(Y_t^{T,\vartheta}\big)
\qquad\text{on \,  }\{X_t\neq0\} \, = \, \{Y_t^{T,\vartheta}\neq0\} \, .
\]

Fix $t\in(0,T]$ and $x\neq0$. Differentiating the identity
$\mathsf{G}_t^\vartheta(\mathsf{H}_t^\vartheta(x))=x$ with respect to $x$ and using
$\nabla^{\dagger}\mathsf{H}_t^\vartheta(x)=\mathsf{J}_t^\vartheta(x)$ yields
\begin{align}\label{PDEUp1D_corr}
\big(\nabla^\dagger \mathsf{G}_t^\vartheta\big)\big(\mathsf{H}_t^\vartheta(x)\big)\,
\mathsf J_t^\vartheta(x)
\,=\, I_2\, .
\end{align}
Differentiating the same identity with respect to $t$ gives
\begin{align}
\big(\partial_t\mathsf{G}_t^\vartheta\big)\big(\mathsf{H}_t^\vartheta(x)\big)
\,=\, -\big(\nabla^{\dagger}\mathsf{G}_t^\vartheta\big)\big(\mathsf{H}_t^\vartheta(x)\big)\,
\big(\partial_t\mathsf{H}_t^\vartheta\big)(x)
\,=\,-(\mathsf{J}_t^\vartheta(x))^{-1}\,(\partial_t\mathsf{H}_t^\vartheta)(x)\,, \nonumber
\end{align}
where the second equality uses~\eqref{PDEUp1D_corr} and the fact that
$\mathsf{J}_t^\vartheta(x)$ is invertible for $x\neq0$ on the set where it is
used; see~\eqref{JacobianRepresentationInverse}. Using~\eqref{UspsilonPDE}, we
therefore obtain
\begin{align}
\big(\partial_t\mathsf{G}_t^\vartheta\big)\big(\mathsf{H}_t^\vartheta(x)\big)
\,=\,
-\frac12\,\big(\mathsf J_t^\vartheta(x)\big)^{-1}(\Delta\mathsf{H}_t^\vartheta)(x)
 \, - \, 
 b_t^\vartheta(x) \nonumber
\end{align}

Next, recall the derivative convention introduced in
Notation~\ref{DerivativeConvention}: for $f:\R^2\to\R^2$ of class $C^2$, the
second derivative $Df(x)$ is viewed as an element of $\R^2\otimes M^{2\times2}$,
and $\Tr_2$ denotes the trace over the matrix component. Under this convention,
we have the componentwise identity $\Tr_2(Df)=\Delta f$. Applying~\eqref{PipD}
with $f=\mathsf{H}_t^\vartheta$ and $g=\mathsf{G}_t^\vartheta$, and using
$(B\otimes C)^{-1}=B^{-1}\otimes C^{-1}$ for invertible matrices together with
cyclicity of the trace, yields
\begin{align}
\Tr_2\!\Big[\Big(I\otimes (\mathsf{J}_t^\vartheta(x))^{\dagger}\Big)
\big(D\mathsf{G}_t^\vartheta\big)\big(\mathsf{H}_t^\vartheta(x)\big)
\Big(I\otimes \mathsf{J}_t^\vartheta(x)\Big)\Big]
\,=\,
-(\mathsf{J}_t^\vartheta(x))^{-1}\,\Tr_2\!\big[(D\mathsf{H}_t^\vartheta)(x)\big] \,=\,
-(\mathsf{J}_t^\vartheta(x))^{-1}\,(\Delta \mathsf{H}_t^\vartheta)(x)\, . \nonumber
\end{align}
Combining the last two displays yields the following identity.
\begin{align}
\big(\partial_t\, \mathsf{G}_t^\vartheta\big)\big(\mathsf{H}_t^\vartheta(x)\big)
&\,=\,
-b_t^\vartheta(x)
\,+\,
\frac12\,\Tr_2\!\Big[\Big(I\otimes (\mathsf{J}_t^\vartheta(x))^{\dagger}\Big)
\big(D\mathsf{G}_t^\vartheta\big)\big(\mathsf{H}_t^\vartheta(x)\big)
\Big(I\otimes \mathsf{J}_t^\vartheta(x)\Big)\Big]
\label{PDEUp2D_corr}
\end{align}

Fix $\varepsilon>0$ and let
$\{\rho_{n}^{\varepsilon\downarrow 0}\}_{n\in\mathbb{N}}$ and
$\{\rho_{n}^{0\uparrow\varepsilon}\}_{n\in\mathbb{N}_0}$ be the stopping times
defined in~\eqref{VARRHOSD}. 
For each $n\in\mathbb{N}$, on the interval
$t\in\big(\rho_{n-1}^{0\uparrow\varepsilon},\,\rho_{n}^{\varepsilon\downarrow 0}\big)$
we have $X_t\neq 0$ by definition of $\rho_{n}^{\varepsilon\downarrow 0}$.
Equivalently, $Y_t^{T,\vartheta}\neq 0$ on this interval, since
$\{X_t\neq0\}=\{Y_t^{T,\vartheta}\neq0\}$.
Hence the map $(t,y)\longmapsto \mathsf{G}_{T-t}^\vartheta(y)$ is of class $C^{1,2}$
on $(0,T)\times(\R^2\setminus\{0\})$.
It\^o's formula therefore applies componentwise to
$t\mapsto \mathsf{G}_{T-t}^\vartheta(Y_t^{T,\vartheta})$ on this interval;
the required integrability is ensured by Hypothesis~\ref{H3}(i).

Applying It\^o's formula yields, for
$t\in(\rho_{n-1}^{0\uparrow\varepsilon},\rho_{n}^{\varepsilon\downarrow 0})$,
\begin{align}\label{Ito_for_X}
dX_t
\,=\,&
d\big(\mathsf{G}_{T-t}^\vartheta(Y_t^{T,\vartheta})\big) \nonumber\\
\,=\,&
-\big(\partial_s \mathsf{G}_s^\vartheta\big)\big(Y_t^{T,\vartheta}\big)\big|_{s=T-t}\,dt
\,+\,
\big(\nabla_y^\dagger \mathsf{G}_{T-t}^\vartheta\big)(Y_t^{T,\vartheta})\,dY_t^{T,\vartheta}
\nonumber\\
\,+\,&
\frac12\,\Tr_2\!\Big[\big(I\otimes (dY_t^{T,\vartheta})^\dagger\big)\,
\big(D_y\mathsf{G}_{T-t}^\vartheta\big)(Y_t^{T,\vartheta})\,
\big(I\otimes dY_t^{T,\vartheta}\big)\Big] \, .
\end{align}
Using $dY_t^{T,\vartheta}=\mathsf{J}_{T-t}^\vartheta(X_t)\,dW_t^{T,\vartheta}$
(by Corollary~\ref{CorollaryMartID}) and writing everything at
$Y_t^{T,\vartheta}=\mathsf{H}_{T-t}^\vartheta(X_t)$, we obtain for the second term above
\begin{align*}
\big(\nabla_y^\dagger \mathsf{G}_{T-t}^\vartheta\big)\big(\mathsf{H}_{T-t}^\vartheta(X_t)\big)\,
dY_t^{T,\vartheta}
\,=\,
\big(\nabla^\dagger \mathsf{G}_{T-t}^\vartheta\big)\big(\mathsf{H}_{T-t}^\vartheta(X_t)\big)\,
\mathsf{J}_{T-t}^\vartheta(X_t)\,dW_t^{T,\vartheta} 
\,=\,
dW_t^{T,\vartheta} \, ,
\end{align*}
where the last equality follows from~\eqref{PDEUp1D_corr} with $t$ replaced by $T-t$
and $x=X_t$.
For the drift term, note that
$\partial_t\big(\mathsf{G}_{T-t}^\vartheta(y)\big)
=
-\big(\partial_s \mathsf{G}_s^\vartheta\big)(y)\big|_{s=T-t}$, and therefore
\begin{align*}
-\big(\partial_s \mathsf{G}_s^\vartheta\big)\big(\mathsf{H}_s^\vartheta(X_t)\big) \, \Big|_{s=T-t}
& \, = \, 
b_{T-t}^\vartheta(X_t)
-\frac12\,\Tr_2\!\Big[\Big(I\otimes (\mathsf{J}_{T-t}^\vartheta(X_t))^{\dagger}\Big)
\big(D\mathsf{G}_{T-t}^\vartheta\big)\big(\mathsf{H}_{T-t}^\vartheta(X_t)\big)
\Big(I\otimes \mathsf{J}_{T-t}^\vartheta(X_t)\Big)\Big] \, ,
\end{align*}
where we used~\eqref{PDEUp2D_corr} with $t$ replaced by $T-t$.
The quadratic-variation term above cancels with that in~\eqref{Ito_for_X}, and therefore,
for
$t\in\big(\rho_{n-1}^{0\uparrow\varepsilon},\,\rho_{n}^{\varepsilon\downarrow 0}\big)$,
we conclude that
$dX_t=dW_t^{T,\vartheta}+ b_{T-t}^\vartheta(X_t)\,dt$.

Thus $X$ satisfies~\eqref{SDEToSolve} on each excursion interval away from $0$; since
these intervals cover $\{t\in[0,T]:X_t\neq0\}$, the SDE holds on $\{X_t\neq0\}$.
Therefore the process
\[
R_t \, := \, X_t \, - \, X_0 \, - \, W_t^{T,\vartheta} \, - \, \int_0^t  \, b_{T-s}^\vartheta(X_s)\,ds,
\qquad t\in[0,T] \, ,
\]
has zero differential on $\{X_t\neq0\}$, and hence is constant on each excursion of
$X$ away from $0$. Let $\Lambda_s^\varepsilon$ be the indicator of the union of
the ``small'' intervals
$\big(\rho_{n}^{\varepsilon\downarrow 0},\rho_{n}^{0\uparrow\varepsilon}\big)$,
so that $\Lambda_s^\varepsilon\le \mathbf{1}_{\{|X_s|\le\varepsilon\}}$. Then
$R_t=\int_0^t \Lambda_s^\varepsilon\,dR_s$ for $t\in[0,T]$ and therefore
\[
\sup_{0\le t\le T}|R_t|
 \, \le \, 
\sup_{0\le t\le T}\Big|\int_0^t  \, \Lambda_s^\varepsilon\,dX_s\Big|
 \, + \, 
 \sup_{0\le t\le T}\Big|\int_0^t \,  \Lambda_s^\varepsilon\,dW_s^{T,\vartheta}\Big|
 \, + \, 
 \int_0^T \,  \Lambda_s^\varepsilon\,|b_{T-s}^\vartheta(X_s)|\,ds \, .
\]
By continuity of $X$, each interval
$\big(\rho_{n}^{\varepsilon\downarrow 0},\rho_n^{0\uparrow\varepsilon}\big)$
contributes an increment of $X$ of magnitude at most $2\varepsilon$, hence
\[
\sup_{0\le t\le T}\Big|\int_0^t \,  \Lambda_s^\varepsilon\,dX_s\Big|
 \, \le \,  2\varepsilon\,N_T^\varepsilon \, .
\]
Taking expectations and using~\eqref{H3Upcross} yields
\[
\mathbb{E}_\mu^{T,\vartheta}\!\Big[\sup_{0\le t\le T}\Big|\int_0^t  \, \Lambda_s^\varepsilon\,dX_s\Big|\Big]
 \, \le \, 
2\varepsilon\,\mathbb{E}_\mu^{T,\vartheta}\!\big[N_T^\varepsilon\big]
 \, \le \, 
2C\,\varepsilon\,\psi(\varepsilon)
\xrightarrow{\varepsilon\downarrow0}0 \, ,
\]
where the final convergence follows from~\eqref{H3PsiCond}. For the second term, Doob's maximal inequality together with It\^o's isometry,
and the definition of $\Lambda_s^\varepsilon$, yield the first two inequalities below,
\[
\mathbb{E}_\mu^{T,\vartheta}\Big[\sup_{0\le t\le T}\Big|\int_0^t \,  \Lambda_s^\varepsilon\,dW_s^{T,\vartheta}\Big|^2\Big]
 \, \le \, 
 4\,\mathbb{E}_\mu^{T,\vartheta}\Big[\int_0^T \,  \Lambda_s^\varepsilon\,ds\Big]
 \, \le \, 
 4\,\mathbb{E}_\mu^{T,\vartheta}\Big[\int_0^T  \, \mathbf{1}_{\{|X_s|\le\varepsilon\}}\,ds\Big]
 \, \le \,  4C\,\eta(\varepsilon) \, ,
\]
where the last inequality follows from~\eqref{H3OccAndDriftSmallBall}. The right-hand
side converges to $0$ as $\varepsilon\downarrow0$ by~\eqref{H3PsiCond}. Similarly,
using~\eqref{H3OccAndDriftSmallBall} together with~\eqref{H3PsiCond}, we obtain
\[
\mathbb{E}_\mu^{T,\vartheta}\Big[\int_0^T  \, \Lambda_s^\varepsilon\,|b_{T-s}^\vartheta(X_s)|\,ds\Big]
 \, \le \, 
\mathbb{E}_\mu^{T,\vartheta}\Big[\int_0^T  \, \mathbf{1}_{\{|X_s|\le\varepsilon\}}\,|b_{T-s}^\vartheta(X_s)|\,ds\Big]
 \, \le  \, 
 C\,\eta(\varepsilon) \xrightarrow{\varepsilon\downarrow0}0 \, .
\]
Hence $\sup_{0\le t\le T}|R_t|\to0$ in probability as $\varepsilon\downarrow0$, and,
since $R$ is continuous, this implies that $R_t=0$ for all $t\in[0,T]$ almost surely.
Therefore~\eqref{SDEToSolve} holds on all of $[0,T]$.

\end{proof}

\section{Proofs of the hitting time results}
\label{SectionOriginEventDProof}

In this section we will prove Proposition~\ref{PropSubMartD} and Theorem~\ref{ThmSubMARTD}, respectively.

\subsection{Submartingale characterization of the pint interaction} \label{PropSubMartDProof}

Recall that the $\mathbb{P}_{\mu}^{T,\vartheta}$-submartingale $\{\mathbf{S}^{T,\vartheta}_t\}_{t\in [0,\infty)}$ is given by
$ \mathbf{S}^{T,\vartheta}_t:=\mathsf{p}_{T-t}^{\vartheta}(X_t)$, where the function $\mathsf{p}_{r}^{\vartheta}:\R^2\rightarrow [0,1]  $ is defined in~\eqref{DefPD}.

\begin{proof}[Proof of Proposition~\ref{PropSubMartD}]

We first show that $\mathbf{S}^{T,\vartheta}$ is a submartingale.
Fix $0\le s<t\le T$. By the Markov property of $X$ under
$\mathbb{P}^{T,\vartheta}_\mu$ and the existence of the transition density
$\mathlarger{\mathsf{d}}^{T,\vartheta}_{s,t}$, we have
\begin{align*}
\mathbb{E}^{T,\vartheta}_\mu\!\left[\mathbf{S}_t^{T,\vartheta}\mid\mathcal{F}^{T,\vartheta,\mu}_s\right]
\,=\,
\int_{\R^2} d^{T,\vartheta}_{s,t}(X_s,y)\,\mathsf{p}^\vartheta_{T-t}(y)\,dy 
\,=\,
\frac{1}{h_{T-s}^\vartheta(X_s)}
\int_{\R^2} \mathsf{K}^\vartheta_{t-s}(X_s,y)\,
(h_0^\vartheta*g_{T-t})(y)\,dy \,.
\end{align*}
Using the inequality $f_r^\vartheta(x,y)\ge g_r(x-y)$ and the semigroup property
of the Gaussian kernel, we obtain
\[
\mathbb{E}^{T,\vartheta}_\mu\!\left[\mathbf{S}_t^{T,\vartheta}\mid\mathcal{F}^{T,\vartheta,\mu}_s\right]
 \, \ge \, 
\frac{1}{h_{T-s}^\vartheta(X_s)}
\int_{\R^2}  \, h_0^\vartheta(z)\,g_{T-s}(X_s-z)\,dz
 \, = \, 
\mathsf{p}^\vartheta_{T-s}(X_s)
 \, = \, 
\mathbf{S}_s^{T,\vartheta} \, .
\]
Hence $\mathbf{S}^{T,\vartheta}$ is a bounded (by 1) submartingale.

Next, we note that $X$ has continuous sample paths.
By Definition~\ref{SpaceTimeHEFDef}, the logarithmic singularity of
$h_0^\vartheta$ is locally integrable, so the convolution
$x\mapsto(h_0^\vartheta*g_t)(x)$ is finite and continuous on $\R^2$ for every
$t>0$. Since $h_t^\vartheta(x)\to\infty$ as $x\to0$, it follows that
$\mathsf{p}_t^\vartheta(x)\to0$ as $x\to0$, and defining
$\mathsf{p}_t^\vartheta(0):=0$ yields a continuous extension to $\R^2$.
Consequently, $\mathbf{S}^{T,\vartheta}$ is continuous. Thus the Doob-Meyer decomposition (see~\cite[Thm.\ 4.10]{Karatzas}) applies.

Notice that using the It\^o isometry, we get
\begin{align*}
\mathbb{E}_\mu^{T,\vartheta}\!\left[
\left|\int_0^t \, 
\nabla\mathsf{p}_{T-s}^{\vartheta}(X_s)\cdot dW_s^{T,\vartheta}\right|^2
\right]
\,=\,
\mathbb{E}_\mu^{T,\vartheta}\!\left[
\int_0^t\,
\big|\nabla\mathsf{p}_{T-s}^{\vartheta}(X_s)\big|^2\,ds
\right] \, ,
\end{align*}
which is finite for all $t\le T$ by Hypothesis~\ref{H4}(i).
Hence the process $\mathbf{M}^{T,\vartheta}$ given by~\eqref{Martingale}
is a square-integrable martingale.

To identify the support of  $\mathbf{A}^{T,\vartheta}:=\mathbf{S}^{T,\vartheta}-\mathbf{M}^{T,\vartheta}$, let $[u,v]\subset[0,T]$
be such that $X_r\neq0$ for all $r\in[u,v]$.
By continuity of $X$, there exists $\varepsilon>0$ with
$|X_r|\ge\varepsilon$ on $[u,v]$. Hence the function $(t,x)\mapsto \mathsf{p}_{T-t}^\vartheta(x)$ is $C^{1,2}$ on
$[u,v]\times\{x\in\R^2:|x|\ge\varepsilon\}$. Applying It\^o's formula to the process $t\mapsto \mathsf{p}_{T-t}^\vartheta(X_t)$
and using that $dX_t = dW_t^{T,\vartheta} + b_{T-t}^{\vartheta}(X_t)\,dt$ for $t\in[0,T]$, yields
\begin{align}\label{Eq:ItoAway0}
\mathsf{p}_{T-v}^\vartheta(X_v) \, - \, \mathsf{p}_{T-u}^\vartheta(X_u)
\,=\,&
\int_u^v  \,  \nabla\mathsf{p}_{T-s}^\vartheta(X_s)\cdot dW_s^{T,\vartheta}  \nonumber\\
\,+\,&
\int_u^v
\Big(
-\partial_s\mathsf{p}_{T-s}^\vartheta(X_s)
 \, + \, 
 \frac12\Delta\mathsf{p}_{T-s}^\vartheta(X_s)
 \, + \, 
 b_{T-s}^\vartheta(X_s)\cdot \nabla\mathsf{p}_{T-s}^\vartheta(X_s)
\Big)\,ds  \, .
\end{align}
The second (drift) integral vanishes because $\mathsf{p}^\vartheta$ solves the backward
Kolmogorov equation~\eqref{PartialForPD} on $\R^2\setminus\{0\}$ and $X_s\neq0$ on $[u,v]$. Therefore,
\[
\mathbf{S}_v^{T,\vartheta} \, - \, \mathbf{S}_u^{T,\vartheta}
 \, = \, 
\int_u^v  \, \nabla\mathsf{p}_{T-s}^\vartheta(X_s)\cdot dW_s^{T,\vartheta}
 \, = \, 
\mathbf{M}_v^{T,\vartheta} \, - \, \mathbf{M}_u^{T,\vartheta} \, ,
\]
where the first equality uses the definition of $\mathbf{S}^{T,\vartheta}$ and the second
uses~\eqref{Martingale}. Hence $\mathbf{A}_v^{T,\vartheta}-\mathbf{A}_u^{T,\vartheta}=0$,
so $\mathbf{A}^{T,\vartheta}$ is constant on $[u,v]$. Equivalently,
\[
\int_0^{T}  \, \mathbf{1}_{\{X_t\neq 0\}}\,d\mathbf{A}_t^{T,\vartheta} \, = \, 0 \, ,
\qquad \mathbb{P}_\mu^{T,\vartheta}\text{-a.s.}
\]

We now show that the process $\mathbf{A}^{T,\vartheta}$ is increasing. Fix $\varepsilon\in(0,1]$ and consider the up- and down-crossing stopping times
$\{\rho_n^{\varepsilon\downarrow 0}\}_{n\ge1}$ and $\{\rho_n^{0\uparrow\varepsilon}\}_{n\ge0}$
defined in~\eqref{VARRHOSD}. For $t\in[0,T]$ set $N_t^\varepsilon:=\sum_{n\ge1}\mathbf{1}_{\{\rho_n^{\varepsilon\downarrow 0}\le t\}}$,
so that the sums below are finite a.s. Define $\mathbf{S}_t^{T,\vartheta}
=
\mathbf{M}_t^{T,\vartheta,\varepsilon}
+
\mathbf{A}_t^{T,\vartheta,\varepsilon}$,
where
\begin{align*}
\mathbf{M}_t^{T,\vartheta,\varepsilon}
\,=\,
\mathsf{p}_T^\vartheta(X_0)
\,+\,\sum_{n=1}^{N_t^\varepsilon+1}
\Big(
\mathbf{S}_{t\wedge\rho_n^{\varepsilon\downarrow 0}}^{T,\vartheta}
\,-\,
\mathbf{S}_{t\wedge\rho_{n-1}^{0\uparrow\varepsilon}}^{T,\vartheta}
\Big)\,,\quad \textup{ and } \quad
\mathbf{A}_t^{T,\vartheta,\varepsilon}
\,=\,
\sum_{n=1}^{N_t^\varepsilon}
\mathbf{S}_{\rho_n^{0\uparrow\varepsilon}}^{T,\vartheta}\,.
\end{align*}
(Here we used that $\mathbf{S}_{\rho_n^{\varepsilon\downarrow 0}}^{T,\vartheta}
=\mathsf{p}_{T-\rho_n^{\varepsilon\downarrow 0}}^\vartheta(0)=0$.) Fix $\delta\in(0,T)$. On the interval $[0,T-\delta]$, we have
$T-\rho_n^{0\uparrow\varepsilon}\in[\delta,T]$ whenever $\rho_n^{0\uparrow\varepsilon}\le T-\delta$.
Hence
\begin{align}\label{Eq:SmallDecreaseBound}
0 \, \le  \,  \mathbf{S}_{\rho_n^{0\uparrow\varepsilon}}^{T,\vartheta}
 \, = \, 
\mathsf{p}_{T-\rho_n^{0\uparrow\varepsilon}}^\vartheta(X_{\rho_n^{0\uparrow\varepsilon}})
 \, \le \, 
\sup_{r\in[\delta,T]}
\frac{\sup_{x\in\R^2}(h_0^\vartheta*g_r)(x)}{\bar h_r^\vartheta(\varepsilon)}
 \, \le \, 
\frac{C_\delta}{\inf_{r\in[\delta,T]}\bar h_r^\vartheta(\varepsilon)}
 \, \le \, 
\frac{C_\delta}{\bar h_\delta^\vartheta(\varepsilon)}\,,
\end{align}
where $\bar h_r^\vartheta$ denotes the radial profile of $h_r^\vartheta$, and $C_\delta$ exists by
the local integrability in \eqref{hSmallaAsympt}, the tail control \eqref{hInfinityAsympt}, and the joint continuity in Definition~\ref{SpaceTimeHEFDef}(v). Moreover, the infimum is positive since $r\mapsto\bar h_r^\vartheta(\varepsilon)$ is continuous and strictly positive on $[\delta,T]$. In particular, $\mathbf{A}^{T,\vartheta,\varepsilon}$ can decrease from its previous running maximum on $[0,T-\delta]$
by at most $C_\delta\,\bar h_\delta^\vartheta(\varepsilon)^{-1}$. Moreover, by~\eqref{hSmallaAsympt}, $\bar h_\delta^\vartheta(\varepsilon)^{-1}\to0$ as $\varepsilon\downarrow0$. By the $L^2$ convergence of $\mathbf{M}^{T,\vartheta,\varepsilon}\to\mathbf{M}^{T,\vartheta}$
(which follows as in the estimate below),
we also have $\sup_{t\le T}|\mathbf{A}_t^{T,\vartheta,\varepsilon}-\mathbf{A}_t^{T,\vartheta}|\to0$ in $L^2$.
Letting $\varepsilon\downarrow0$ in the inequality
$\mathbf{A}_t^{T,\vartheta,\varepsilon}\ge \mathbf{A}_s^{T,\vartheta,\varepsilon}-C_\delta\,\bar h_\delta^\vartheta(\varepsilon)^{-1}$
for $0\le s<t\le T-\delta$ yields $\mathbf{A}_t^{T,\vartheta}\ge \mathbf{A}_s^{T,\vartheta}$.
Since $\delta\in(0,T)$ is arbitrary and $\mathbf{A}^{T,\vartheta}$ is continuous, $\mathbf{A}^{T,\vartheta}$ is increasing on $[0,T]$.

It remains only to show that $\mathbf{M}^{T,\vartheta,\varepsilon}\to\mathbf{M}^{T,\vartheta}$  converges in $L^2$. Define the adapted process $\{\Lambda_s^\varepsilon\}_{s\in[0,T]}$ by $\Lambda_s^\varepsilon
:=
\sum_{n\ge1}
\mathbf{1}_{[\rho_{n-1}^{0\uparrow\varepsilon},\,\rho_n^{\varepsilon\downarrow 0})}(s)$, that is, $\Lambda_s^\varepsilon=1$ when $s$ belongs to an excursion of $X$
outside the ball $\{|x|\le\varepsilon\}$ and $\Lambda_s^\varepsilon=0$ otherwise.
Since $\mathbf{A}^{T,\vartheta}$ is constant on excursion intervals, we have
$\Lambda_s^\varepsilon\,d\mathbf{A}_s^{T,\vartheta}=0$, and therefore
\[
\mathbf{M}_t^{T,\vartheta,\varepsilon}
 \, = \, 
\mathsf{p}_T^\vartheta(X_0)
 \, + \, 
\int_0^t  \, \Lambda_s^\varepsilon\,d\mathbf{S}_s^{T,\vartheta}
 \, = \, 
\mathsf{p}_T^\vartheta(X_0)
 \, + \, 
\int_0^t \, 
\Lambda_s^\varepsilon
\nabla\mathsf{p}_{T-s}^\vartheta(X_s)\cdot dW_s^{T,\vartheta} \, ,
\]
where the last equality follows from the definition~\eqref{Martingale} of
$\mathbf{M}^{T,\vartheta}$.
By It\^o's isometry and the inequality
$1-\Lambda_s^\varepsilon\le\mathbf{1}_{\{|X_s|\le\varepsilon\}}$, we obtain
\[
\mathbb{E}_\mu^{T,\vartheta}
\!\left[
\big|\mathbf{M}_t^{T,\vartheta,\varepsilon}-\mathbf{M}_t^{T,\vartheta}\big|^2
\right]
 \, = \, 
\mathbb{E}_\mu^{T,\vartheta}
\!\left[
\int_0^t
(1-\Lambda_s^\varepsilon)
\big|\nabla\mathsf{p}_{T-s}^\vartheta(X_s)\big|^2\,ds
\right]
 \, \le \, 
\mathbb{E}_\mu^{T,\vartheta}
\!\left[
\int_0^T \, 
\mathbf{1}_{\{|X_s|\le\varepsilon\}}
\big|\nabla\mathsf{p}_{T-s}^\vartheta(X_s)\big|^2\,ds
\right] \, ,
\]
which converges to $0$ as $\varepsilon\downarrow0$ by Hypothesis~\ref{H4}(ii).
Applying Doob's $L^2$ inequality yields
\[
\mathbb{E}_\mu^{T,\vartheta}
\!\left[
\sup_{t\le T}
\big|\mathbf{M}_t^{T,\vartheta,\varepsilon}-\mathbf{M}_t^{T,\vartheta}\big|^2
\right]
 \, \le \, 
4\,
\mathbb{E}_\mu^{T,\vartheta}
\!\left[
\big|\mathbf{M}_T^{T,\vartheta,\varepsilon}-\mathbf{M}_T^{T,\vartheta}\big|^2
\right]
\longrightarrow0 \quad\text{as }\varepsilon\downarrow0\,.
\]
This completes the proof.

\end{proof}

\subsection[Structural]{Structural properties of planar $h^\vartheta$-diffusions}\label{SubsectionThmSubMARTD}

Recall that $\mathcal O_T\in\mathcal{B}(\boldsymbol{\Upsilon})$ denotes the event
that the process $X$ visits the origin at some time in the interval $[0,T]$ and
$\tau := \inf\{t\in[0,T]:\, X_t=0\}$ denotes the first hitting time of the origin
on $[0,T]$, with the convention that $\tau=\infty$ if $X_t\neq0$ for all
$t\in[0,T]$. We have $\mathcal O_T=\{\tau\le T\}$ and
$\mathcal O_T^c=\{\tau>T\}$.

In the proof below, we work with the usual augmentation
$\mathcal{F}^{T,\vartheta,x}$ of the natural filtration $\mathcal{F}$ under
$\mathbb P_x^{T,\vartheta}$. All martingale, stopping-time, and Markov properties
established with respect to $\mathcal{F}$ therefore remain valid, and conditional
expectations with respect to $\mathcal{F}_{\mathsf{S}}$ are unchanged up to
$\mathbb P_x^{T,\vartheta}$-null sets (see, e.g.,
\cite[Section~2.7(A-B)]{Karatzas}).

Define the adapted process $\{\Lambda_t^{T,\vartheta}\}_{t\in[0,T]}$ by
\begin{align}\label{RNDvartheta}
\Lambda_t^{T,\vartheta}
\,:=\,
\frac{h_{T-t}^{\vartheta}\!\big(X_t\big)}
{(h_0^{\vartheta} * g_{T-t})\!\big(X_t\big)} \, ,
\qquad t\in[0,T] \, .
\end{align}
For any $\mathcal{F}$-stopping time $\mathsf{S}$, we write
$\Lambda_{\mathsf{S}}^{T,\vartheta}:=\Lambda_t^{T,\vartheta}\big|_{t=\mathsf{S}}$
on $\{\mathsf{S}<\tau\}$. Let $\mathring{\mathbb{P}}_{x }^{T,\vartheta}
:=\mathbb{P}_{x }^{T,\vartheta}[\,\cdot\mid\mathcal O_T^c]$ and denote by
$\mathring{\mathbb{E}}_{x }^{T,\vartheta}$ the corresponding expectation.
For any $\mathcal{F}$-stopping time $\mathsf{S}$ such that
$\mathring{\mathbb{E}}_{x }^{T,\vartheta}\!\big[\Lambda_{\mathsf{S}}^{T,\vartheta}\big]>0$,
define the probability measure $\mathbb{Q}_{x}^{T,\vartheta,\mathsf{S}}$
on $\mathcal{F}_{\mathsf{S}}$ by
\begin{align}\label{RNQ}
\mathbb{Q}_{x}^{T,\vartheta,\mathsf{S}}(d\omega)
\,:=\,
\frac{\Lambda_{\mathsf{S}}^{T,\vartheta}(\omega)}
{\mathring{\mathbb{E}}_{x }^{T,\vartheta}\!\big[\Lambda_{\mathsf{S}}^{T,\vartheta}\big]}\,
\mathring{\mathbb{P}}_{x }^{T,\vartheta}(d\omega) \, ,
\qquad \text{on }\{\mathsf{S}<\tau\} \, .
\end{align}

\begin{proof}[Proof of Theorem~\ref{ThmSubMARTD}]
\noindent Part (i): On the event $\{\tau>T\}=\mathcal O_T^c$ we have $\mathbf{S}^{T,\vartheta}_{\tau\wedge T} = \mathsf{p}_0^\vartheta(X_T) =1$. Conversely, on $\{\tau\le T\}$ we have $\mathbf{S}^{T,\vartheta}_{\tau\wedge T}
=\mathsf{p}_{T-\tau}^\vartheta(0)=0$ since $h_{T-\tau}^\vartheta(0)=\infty$ whereas $(h_0^\vartheta*g_{T-\tau})(0)<\infty$.
Therefore,
\begin{align}\label{1OCompD}
\mathbf{S}^{T,\vartheta}_{\tau\wedge T}
 \, = \, \mathbf{1}_{\{\tau>T\}}
 \, = \, \mathbf{1}_{\mathcal O_T^c} \, ,
\qquad \mathbb{P}_x^{T,\vartheta}\text{-a.s.}
\end{align}
By Proposition~\ref{PropSubMartD}, the increasing component
$\mathbf{A}^{T,\vartheta}$ in the Doob-Meyer decomposition
$\mathbf{S}^{T,\vartheta}
=\mathbf{M}^{T,\vartheta}+\mathbf{A}^{T,\vartheta}$
is constant on every time interval during which $X$ remains away from the origin.
In particular, $\mathbf{A}^{T,\vartheta}$ is constant on $[0,\tau)$, and hence
$\mathbf{A}_t^{T,\vartheta}=\mathbf{A}_0^{T,\vartheta}=0$ for all $t<\tau$.
Since $\mathbf{A}^{T,\vartheta}$ is continuous, it follows that
$\mathbf{A}_\tau^{T,\vartheta}=0$ as well.
Applying the optional stopping theorem to the martingale
$\mathbf{M}^{T,\vartheta}
=\mathbf{S}^{T,\vartheta}-\mathbf{A}^{T,\vartheta}$
at the bounded stopping time $\tau\wedge T$ yields
\begin{align}
\mathbb{E}_x^{T,\vartheta}\!\big[
\mathbf{S}^{T,\vartheta}_{\tau\wedge T}
\big]
\,=\,
\mathbb{E}_x^{T,\vartheta}\!\big[
\mathbf{S}^{T,\vartheta}_{\tau\wedge T}
-\mathbf{A}^{T,\vartheta}_{\tau\wedge T}
\big]
\,=\,
\mathbf{S}^{T,\vartheta}_0
\,=\,
\mathsf{p}_T^\vartheta(x)\,, \nonumber
\end{align}
where we used $\mathbf{A}^{T,\vartheta}_{\tau\wedge T}=0$. Combining the above with~\eqref{1OCompD} yields
$\mathbb{P}_x^{T,\vartheta}[\mathcal O_T^c]
=\mathsf{p}_T^\vartheta(x)$, which proves~\textup{(i)}. \vspace{.3cm}

\noindent\textup{Part (ii):}
Define the bounded process $\{M_t^{T,\vartheta}\}_{t\in[0,T]}$ by
$M_t^{T,\vartheta}
:=
\mathsf{p}_T^\vartheta(x)^{-1}\,\mathbf{S}^{T,\vartheta}_{\tau\wedge t}$.
By~\eqref{1OCompD} and part~\textup{(i)}, we have
$M_T^{T,\vartheta}
=\mathbb{P}_x^{T,\vartheta}[\mathcal O_T^c]^{-1}\,\mathbf{1}_{\mathcal O_T^c}$.
Recalling that $\mathring{\mathbb{P}}^{T,\vartheta}_x
:=\mathbb{P}_x^{T,\vartheta}[\,\cdot\,\mid \mathcal O_T^c]$, it follows that
\[
\frac{d\mathring{\mathbb{P}}^{T,\vartheta}_x}{d\mathbb{P}_x^{T,\vartheta}}\Big|_{\mathcal{F}_T}
 \, = \, 
M_T^{T,\vartheta} \, .
\]
Moreover, since $\mathbf{S}^{T,\vartheta}=\mathbf{M}^{T,\vartheta}+\mathbf{A}^{T,\vartheta}$ with
$\mathbf{A}^{T,\vartheta}$ constant on $[0,\tau)$, we have
$\mathbf{S}^{T,\vartheta}_{\tau\wedge t}=\mathbf{M}^{T,\vartheta}_{\tau\wedge t}$ and hence
$\{M_t^{T,\vartheta}\}_{t\in[0,T]}$ is a mean-one $\mathbb{P}_x^{T,\vartheta}$-martingale
by optional stopping for the martingale $\mathbf{M}^{T,\vartheta}$ at the bounded times $\tau\wedge t$. Since $\mathbf{S}^{T,\vartheta}=\mathbf{M}^{T,\vartheta}+\mathbf{A}^{T,\vartheta}$ and
$\mathbf{A}^{T,\vartheta}$ does not increase on $\{t<\tau\}$, we have
\[
d\mathbf{S}_t^{T,\vartheta}
 \, = \, 
\mathbf{1}_{\{t<\tau\}}\,
\nabla \mathsf{p}_{T-t}^\vartheta(X_t)\cdot dW_t^{T,\vartheta} \, ,
\qquad 0 \le t\le T \, ,
\]
where $W^{T,\vartheta}$ is the $\mathbb{P}_x^{T,\vartheta}$-Brownian motion from
Proposition~\ref{PropStochPreD}. Therefore
\begin{align}\label{Emmy}
dM_t^{T,\vartheta}
 \, = \, 
\frac{1}{\mathsf{p}_T^\vartheta(x)}\,\,d\mathbf{S}^{T,\vartheta}_{\tau\wedge t}
 \, = \, 
\mathbf{1}_{\{t<\tau\}}\,
\frac{1}{\mathsf{p}_T^\vartheta(x)}\,\,
\nabla \mathsf{p}_{T-t}^\vartheta(X_t)\cdot dW_t^{T,\vartheta}
 \, = \, 
\mathbf{1}_{\{t<\tau\}}\,M_t^{T,\vartheta}\,
\nabla\log\mathsf{p}_{T-t}^\vartheta(X_t)\cdot dW_t^{T,\vartheta}\,,
\end{align}
where the last equality uses that on $\{t<\tau\}$ we have
$M_t^{T,\vartheta}
=\frac{1}{\mathsf{p}_T^\vartheta(x)}\,\mathsf{p}_{T-t}^\vartheta(X_t)$
and $\nabla \mathsf{p} = \mathsf{p}\,\nabla\log \mathsf{p}$. Thus, by Girsanov's theorem~\cite[Thm.\ 5.1 \& Cor.\ 5.13]{Karatzas}, under
$\mathring{\mathbb{P}}_x^{T,\vartheta}$ the process
\begin{align}\label{TwoWs}
\mathring W_t^{T,\vartheta}
:=
W_t^{T,\vartheta}
-
\int_0^t \mathbf{1}_{\{s<\tau\}}\nabla\log\mathsf{p}_{T-s}^\vartheta(X_s)\,ds
\end{align}
is a two-dimensional standard Brownian motion.
Since under $\mathbb{P}_x^{T,\vartheta}$ the coordinate process satisfies
$dX_t=dW_t^{T,\vartheta}+b_{T-t}^\vartheta(X_t)\,dt$ (Proposition~\ref{PropStochPreD}),
it follows from~\eqref{TwoWs} that under $\mathring{\mathbb{P}}_x^{T,\vartheta}$,
\[
dX_t
 \, = \, 
d\mathring W_t^{T,\vartheta}
 \, + \, 
\Big(b_{T-t}^\vartheta(X_t)+\mathbf{1}_{\{t<\tau\}}\nabla\log\mathsf{p}_{T-t}^\vartheta(X_t)\Big)\,dt
 \, = \, 
d\mathring W_t^{T,\vartheta}
 \, + \, 
\nabla\log(h_0^\vartheta*g_{T-t})(X_t)\,dt \, .
\]
Equivalently, with $\mathring b_t^\vartheta(x):=\nabla_x\log(h_0^\vartheta*g_t)(x)$, the SDE
$dX_t=d\mathring W_t^{T,\vartheta}+\mathring b_{T-t}^\vartheta(X_t)\,dt$ holds
for $0\le t\le T$ under $\mathring{\mathbb{P}}_x^{T,\vartheta}$, as claimed. \vspace{.3cm}

\noindent Part (iii):
Let $F$ be a bounded $\mathcal{F}_{\mathsf{S}}$-measurable random variable.
Set $\mathcal O_T^c:=\{\tau>T\}$. By Part~(i),
we have
\begin{align}\label{Eq:CondSurvival_S}
\mathbb{P}_x^{T,\vartheta}\!\big[\mathcal O_T^c\mid \mathcal{F}_{\mathsf{S}}\big]
 \, = \, 
\mathbf{1}_{\{\mathsf{S}<\tau\}}\,
\frac{(h_0^\vartheta*g_{T-\mathsf{S}})(X_{\mathsf{S}})}{h_{T-\mathsf{S}}^\vartheta(X_{\mathsf{S}})}
 \, = \, 
\mathbf{1}_{\{\mathsf{S}<\tau\}}\,
\big(\Lambda_{\mathsf{S}}^{T,\vartheta}\big)^{-1} \, ,
\qquad \mathbb{P}_x^{T,\vartheta}\text{-a.s.}\,,
\end{align}
where recall that $\Lambda_{\mathsf{S}}^{T,\vartheta}$ is given by~\eqref{RNDvartheta}. In particular, on $\{\mathsf{S}<\tau\}$ the conditional probability in~\eqref{Eq:CondSurvival_S}
is strictly positive. Using the tower property we obtain the first equality below.
\begin{align}\label{Eq:ES_identity}
\mathbb{E}_x^{T,\vartheta}\!\big[F\,\mathbf{1}_{\{\mathsf{S}<\tau\}}\big]
\,=\,
\mathbb{E}_x^{T,\vartheta}\!\Bigg[
F\,\mathbf{1}_{\{\mathsf{S}<\tau\}}\,
\frac{\mathbf{1}_{\mathcal O_T^c}}{\mathbb{P}_x^{T,\vartheta}\!\big[\mathcal O_T^c\mid \mathcal{F}_{\mathsf{S}}\big]}
\Bigg]
\,=\,
\mathbb{E}_x^{T,\vartheta}\!\Big[
F\,\mathbf{1}_{\mathcal O_T^c}\,
\Lambda_{\mathsf{S}}^{T,\vartheta}
\Big]\,
\,=\,
\mathbb{P}_x^{T,\vartheta}\!\big[ \mathcal O_T^c\big]\, 
\mathbb{\mathring{E}}_x^{T,\vartheta}\!\Big[ F\,
\Lambda_{\mathsf{S}}^{T,\vartheta}
\Big]\, ,
\end{align}
The second equality above uses~\eqref{Eq:CondSurvival_S} and the last equality uses the definition of $\mathring{\mathbb{P}}_x^{T, \vartheta}:=\mathbb{P}_x^{T,\vartheta}[\cdot\mid \mathcal O_T^c]$. In particular $F=\mathbf{1}_{\boldsymbol{\Upsilon}}$ in~\eqref{Eq:ES_identity} yields that $\mathbb{P}_x^{T, \vartheta}\!\big[ \mathsf{S}<\tau\big]
= \mathbb{P}_x^{T, \vartheta}\!\big[ \mathcal O_T^c\big]\,
\mathring{\mathbb{E}}_x^{T,\vartheta}\!\big[
\,\Lambda_{\mathsf{S}}^{T,\vartheta}
\big]$. Using this together with~\eqref{Eq:ES_identity}, we get the desired equality as following.
\begin{align*}
\mathbb{E}_x^{T,\vartheta}\!\big[F\mid \mathsf{S}<\tau\big]
& \, = \, 
\frac{\mathbb{E}_x^{T,\vartheta}\!\big[F\,\mathbf{1}_{\{\mathsf{S}<\tau\}}\big]}
{\mathbb{P}_x^{T,\vartheta}[\mathsf{S}<\tau]}
 \, = \, 
\frac{\mathring{\mathbb{E}}_x^{T,\vartheta}\!\big[F\,\Lambda_{\mathsf{S}}^{T,\vartheta}\big]}
{\mathring{\mathbb{E}}_x^{T,\vartheta}\!\big[\Lambda_{\mathsf{S}}^{T,\vartheta}\big]}
\end{align*}
Since this holds for every bounded $\mathcal{F}_{\mathsf{S}}$-measurable $F$,
it follows that the restriction of $\mathring{\mathbb{P}}_x^{T,\vartheta,\mathsf{S}}
:=\mathbb{P}_x^{T,\vartheta}[\,\cdot\,\mid \mathsf{S}<\tau]$ to $\mathcal{F}_{\mathsf{S}}$
coincides with the restriction of the measure $\mathbb{Q}_x^{T,\vartheta,\mathsf{S}}$
defined by~\eqref{RNQ}. Consequently, the stopped process $\{X_{t\wedge \mathsf{S}}\}_{t\in[0,T]}$ has the same law
under $\mathring{\mathbb{P}}_x^{T,\vartheta,\mathsf{S}}$ as under $\mathbb{Q}_x^{T,\vartheta,\mathsf{S}}$. \vspace{.3cm}

\noindent\textup{Part (iv):}
Fix $t\in(0,T)$.
Since $\{\tau\le t\}=\{\tau\le T\}\setminus\{t<\tau\le T\}$, we have
\begin{align}\label{CondDensitytau}
\mathbb{P}_x^{T,\vartheta}[\tau\le t\mid \mathcal O_T]
& \, = \, 
1-\frac{\mathbb{P}_x^{T,\vartheta}[t<\tau\le T]}{\mathbb{P}_x^{T,\vartheta}[\mathcal O_T]}
 \, = \, 
\frac{\mathbb{P}_x^{T,\vartheta}[\tau\le t]}{\mathbb{P}_x^{T,\vartheta}[\mathcal O_T]} \, .
\end{align}

Next, observe that
\[
\mathbb{P}_x^{T,\vartheta}[\tau\le t]
 \, = \, 
1 \, - \, \mathbb{P}_x^{T,\vartheta}[t<\tau]
 \, = \, 
1 \, - \, \mathbb{P}_x^{T,\vartheta}[\mathcal O_T^c]\,
\mathring{\mathbb{E}}_x^{T,\vartheta}\!\big[\Lambda_t^{T,\vartheta}\big] \, ,
\]
where the last equality follows from~\eqref{Eq:ES_identity} applied with the
deterministic stopping time \(\mathsf{S}:=t\) and \(F:=1\).
Substituting this into~\eqref{CondDensitytau} yields
\begin{align}
    \mathbb{P}_x^{T,\vartheta}[\tau\le t\mid \mathcal O_T]
\,=\,
\frac{1-\mathsf{p}_T^\vartheta(x)\,
\mathring{\mathbb{E}}_x^{T,\vartheta}\!\big[\Lambda_t^{T,\vartheta}\big]}
{1-\mathsf{p}_T^\vartheta(x)}\,, \nonumber
\end{align}
where we also used \(\mathbb{P}_x^{T,\vartheta}[\mathcal O_T^c]=\mathsf{p}_T^\vartheta(x)\).
Using~\eqref{RNDvartheta} and~\eqref{SecondTrans}, we compute
\(\mathring{\mathbb{E}}_x^{T,\vartheta}\!\big[\Lambda_t^{T,\vartheta}\big]\) as follows.
\begin{align}
\int_{\R^2}
\frac{h_{T-t}^{\vartheta}(y)}{(h_0^{\vartheta} * g_{T-t})(y)}\,
\mathlarger{\mathsf{\mathring{d}}}_{0,t}^{T,\vartheta}(x,y)\,dy
\,=\,
\frac{1}{(h_0^\vartheta*g_T)(x)}
\int_{\R^2} h_{T-t}^{\vartheta}(y)\,g_t(x-y)\,dy
\,=\,
\frac{(h_{T-t}^\vartheta*g_t)(x)}{(h_0^\vartheta*g_T)(x)}\,
\nonumber
\end{align}
Substitution the above into \(\mathbb{P}_x^{T,\vartheta}[\tau\le t\mid \mathcal O_T]\),
followed by differentiation in \(t\), gives the desired conditional density of
\(\tau\) under \(\mathbb{P}_x^{T,\vartheta}[\,\cdot\,\mid\mathcal O_T]\).

\end{proof}

\begin{appendix}


\section{Auxiliary Lemmas}\label{AuxLemmas}

\begin{lemma}\label{LemmaSpaceTimeHarmonicHath}
Fix $T,\vartheta>0$ and let $h_0^\vartheta:\R^2\to(0,\infty)$ be radially symmetric. For all $t\in(0,T]$ and all $x\in\R^2$,
\begin{enumerate}[(i)]
\item
\(
\int_{\R^2} g_{t}(x-y)\,(h_0^\vartheta*g_{T-t})(y)\,dy
=
(h_0^\vartheta*g_{T})(x),
\)
where $g_t$ denotes the two-dimensional Gaussian heat kernel.

\item
\(
\int_{\R^2} \mathsf{K}_{t}^\vartheta(x,y)\,(\hat h_0^\vartheta*g_{T-t})(y)\,dy
=
(\hat h_0^\vartheta*g_{T})(x),
\)
where $\hat h_0^\vartheta(x):=x\,h_0^\vartheta(x)$ and $\mathsf{K}_{t}^\vartheta$ is
given by~\eqref{DefFullKer}.
\end{enumerate}
\end{lemma}

\begin{proof}
Part~(i) follows from the semigroup property of the Gaussian heat kernel. \vspace{.3cm}

\noindent Part~(ii): Fix $0\le s<t\le T$ and $x\in\R^2$. Since
$\mathsf{K}^{\vartheta}_t(x,y):=g_t(x-y)+ \mathsf{v}_t^{\vartheta}(x,y)$, where $\mathsf{v}_t^{\vartheta}(x,y)$ is defined in~\eqref{DefFullKerSecondTerm}, we have 
\begin{align}
\int_{\R^2}\, \mathsf{K}^{\vartheta}_t(x,y)\,(\hat h_0^\vartheta*g_{T-t})(y)\,dy
\,=\,
\int_{\R^2}\, g_t(x-y)\,(\hat h_0^\vartheta*g_{T-t})(y)\,dy
\,+\, 
\int_{\R^2}\, \mathsf{v}^{\vartheta}_t(x,y)\,(\hat h_0^\vartheta*g_{T-t})(y)\,dy\,.
\nonumber
\end{align}
By the semigroup property of the Gaussian heat kernel, we get the second equality below.
\begin{align}
    \int_{\R^2}\, g_t(x-y)\,(\hat h_0^\vartheta*g_{T-t})(y)\,dy
\,=\,&
\int_{\R^2}  z \, h_0^{\vartheta}(z)\, \int_{\R^2}\, g_{t-s}(x-y)\,   g_{T-t}(y-z)\,dy\, dz \nonumber \\
\,=\,&
\int_{\R^2} z\,h_0^\vartheta(z)\,g_{T-s}(x-z)\,dz
\,=\,
(\hat h_0^\vartheta*g_{T-s})(x)\, \nonumber
\end{align}
Using~\eqref{DefFullKerSecondTerm}, we further have
\begin{align*}
\int_{\R^2}\mathsf v_t^{\vartheta}(x,y)\,(\hat h_0^\vartheta*g_{T-t})(y)\,dy
\,=\,
2\pi\,\vartheta
\int_{0<r<s<t}
g_r(x)\,\nu'\!\big(\vartheta(s-r)\big)
\bigg[\int_{\R^2} g_{t-s}(y)\,(\hat h_0^\vartheta*g_{T-t})(y)\,dy\bigg]
\,ds\,dr 
\,=\,0
\end{align*}
since the inner integral is vanishes.

\end{proof}

\subsection{Modified Bessel function}

Recall that the modified Bessel function of the second kind of order $\nu\in\R$ admits the integral representation (see, for instance,~\cite[p.\ 183]{Watson})
\begin{align}\label{DefBessel}
K_\nu(z)
\,:=\,
\frac{1}{2}\Big(\frac{z}{2}\Big)^{\nu}
\int_0^\infty
a^{-\nu-1}\,
e^{-a-\frac{z^2}{4a}}\,
da,
\qquad z>0 \, .
\end{align}
Next, consider the corresponding incomplete modified Bessel function, defined as
the upper tail of the above integral:
\begin{align}\label{DefIncBessel}
K_\nu(z,y)
\,:=\,
\frac{1}{2}\Big(\frac{z}{2}\Big)^{\nu}
\int_{y}^{\infty}
a^{-\nu-1}\,
e^{-a-\frac{z^2}{4a}}\,
da,
\qquad z>0\,,\ y>0\,.
\end{align}

\begin{lemma}\label{LemRealUnnormEigen}
Given $t,\vartheta>0$, let
$\mathsf{K}_t^{\vartheta}:\mathbb R^2\times\mathbb R^2\to[0,\infty]$
be defined as in~\eqref{DefFullKer} and $K_0(z,y)$ denote the incomplete modified Bessel function of the second kind of order $0$.
Then, for all $x\in\mathbb R^2\setminus\{0\}$,
\begin{enumerate}[(i)]

\item $\int_{\R^2}\, g_t(x-y)\,K_0\big(\sqrt{2\vartheta}  |y|\big)\,dy \,=\, e^{ \vartheta t }\, K_0\big(\sqrt{2 \vartheta} |x|\,,\vartheta t\big)$

\item $\int_{\R^2}\, g_{t}(x-y) \, K_0\big(\sqrt{2\vartheta}|y| \,,\vartheta (T-t) \big) \, dy \,=\, e^{\vartheta t} K_0\big(\sqrt{2\vartheta}  |x|\,,\vartheta T\big)$

\item $\int_{\mathbb R^2} \mathsf{K}_t^{\vartheta}(x,y)\, K_0\!\big(\sqrt{2\vartheta}\,|y|\big)\,dy = e^{\vartheta t}\, K_0\!\big(\sqrt{2\vartheta}\,|x|\big)$

\end{enumerate}

\end{lemma}

\begin{proof}

\noindent Part (i):  Applying~(\ref{DefBessel}) in particular for $z=\sqrt{2 \vartheta}|x|$ and $\nu=0$ , we obtain
\begin{align}\label{DefBessel2}
K_0\big(\sqrt{2 \vartheta}|x| \big)\,=\,\int_0^{\infty}\, \frac{1}{2a}e^{-a-\frac{\vartheta |x|^2}{2a} }da\,\,=\,\int_0^{\infty}\, \frac{1}{2a}e^{-\vartheta a-\frac{|x|^2}{2a} }da\,,
\end{align}
where the second representation follows from the change of variables
$a\mapsto \vartheta a$. Using the Gaussian heat kernel
$g_t(x)=\frac{1}{2\pi t}\,e^{-\frac{|x|^2}{2t}}$ together with
\eqref{DefBessel2}, we get the first two equalities below.
\begin{align}
\int_{\R^2}\, g_t(x-y)\,K_0\big(\sqrt{2\vartheta}  |y|\big)\,dy
\,=\,
&\int_{\R^2}\,\frac{ e^{ -\frac{ |x-y|^2  }{ 2t }  }
 }{ 2\pi t}\,K_0\big(\sqrt{2 \vartheta} |y|\big)\, \nonumber \\
 \,=\,&\,\frac{ 1  }{2   } \,\int_0^{\infty}\,e^{- \vartheta a }\,
\int_{\R^2}\,\frac{e^{ -\frac{ |x-y|^2  }{ 2t }  }
 }{ 2\pi t}\,\frac{e^{-\frac{ |y|^2  }{ 2a } }  }{ a }\,dy \, da  \nonumber  \\
\,=\, &\,\frac{ 1 }{2   } \,\int_0^{\infty}e^{-\vartheta  a}\,
\frac{e^{-\frac{ |x|^2  }{ 2(t+a) } }  }{ t+a } \, da  \nonumber \\
\,=\, &\,\frac{1}{2}\,e^{\vartheta t } \,\int_t^{\infty}\,e^{- \vartheta a }\,
\frac{e^{-\frac{ |x|^2  }{ 2a } }  }{ a }  \,da \nonumber \\
\,=\,& e^{ \vartheta t }\, K_0\big(\sqrt{2 \vartheta} |x|\,,\vartheta t\big)\, \nonumber 
\end{align}
Here, the third equality follows from the convolution identity for Gaussian
kernels, the fourth equality is obtained by the change of variables
$a\mapsto a-t$, and the final equality uses the definition of the incomplete modified Bessel function~\eqref{DefIncBessel}. \vspace{.3cm}

\noindent\textup{Part (ii).} Using part~(i), we get the first and third equalities below.
\begin{align}
\int_{\R^2}\, g_{t}(x-y) \, K_0\big(\sqrt{2\vartheta}|y| \,,\vartheta (T-t) \big) \, dy
\,=\,&
e^{-\vartheta (T-t)}\, \int_{\R^2} \, K_0\big(\sqrt{2\vartheta}|z| \big) \,
\int_{\R^2}\, g_{t}(x-y) \, \, g_{T-t}(y-z) \,dy\, dz \nonumber \\
\,=\, &e^{-\vartheta (T-t)}\, \int_{\R^2}\,K_0\big(\sqrt{2\vartheta}  |z|\big)\,  g_{T}(x-z) \, dz \nonumber \\
\,=\, &e^{\vartheta t}\,K_0\big(\sqrt{2\vartheta}  |x|\,,\vartheta T\big)\,, \nonumber 
\end{align}
where the second equality uses the Gaussian convolution. \vspace{.3cm}

\noindent\textup{Part (iii).}
Since
$\mathsf{K}^{\vartheta}_t(x,y):=g_t(x-y)+ \mathsf{v}_t^{\vartheta}(x,y)$, where $\mathsf{v}_t^{\vartheta}(x,y)$ is defined in~\eqref{DefFullKerSecondTerm},
we may decompose
\begin{align}
\int_{\R^2}\, \mathsf{K}^{\vartheta}_t(x,y)\,
K_0\big(\sqrt{2\vartheta}\,|y|\big)\,dy
\,=\,
\int_{\R^2}\, g_t(x-y)\,
K_0\big(\sqrt{2\vartheta}\,|y|\big)\,dy
\,+\, 
\int_{\R^2}\, \mathsf{v}^{\vartheta}_t(x,y)\,
K_0\big(\sqrt{2\vartheta}\,|y|\big)\,dy .
\nonumber
\end{align}
Using part~(i) together with~\eqref{DefBessel2}, the first term can be written as
\begin{align}\label{ForI}
\int_{\R^2}\, g_t(x-y)\,
K_0\big(\sqrt{2\vartheta}\,|y|\big)\,dy
\,=\,
e^{\vartheta t}\,
K_0\big(\sqrt{2\vartheta}\,|x|\big)
-\pi\int_0^{t}\, g_a(x)\, e^{\vartheta(t-a)}\,da .
\end{align}
Next, using~\eqref{DefFullKerSecondTerm} and~\eqref{DefBessel2}, we obtain
\begin{align}
\int_{\R^2}\, \mathsf{v}^{\vartheta}_t(x,y)\,
K_0\big(\sqrt{2\vartheta}\,|y|\big)\,dy
\,=\,
&\pi \vartheta 
\int_{0<r<s<t}
g_r(x)\,
\nu'\big(\vartheta(s-r)\big)\,
\int_0^{\infty} e^{-\vartheta a}
\int_{\R^2}\, g_{t-s}(y)\,
\frac{e^{-\frac{|y|^2}{2a}}}{a}
\,dy\,da\,ds\,dr .
\nonumber
\end{align}
The inner integrals over $y\in\R^2$ and $a\in[0,\infty)$ can be evaluated as
\begin{align*}
\int_0^{\infty} e^{-\vartheta a}\,\frac{1}{a+t-s}\,da
\,=\,
e^{\vartheta (t-s)}
\int_{\vartheta (t-s)}^{\infty} \frac{e^{-b}}{b}\,db
\,=\,
e^{\vartheta (t-s)}\,E_1\!\big(\vartheta (t-s)\big)
\,=:\,
\boldsymbol{E}\!\big(\vartheta (t-s)\big),
\end{align*}
where, for the first equality, we have made the change of integration variable
$b=\vartheta(a+t-s)$, the function $E_1(x):=\int_x^{\infty}\frac{e^{-b}}{b}\,db$ for $x>0$ denotes the classical exponential integral, and
$\boldsymbol{E}(x):=e^{x}E_1(x)$ is its exponentially renormalized version. Therefore,
\begin{align}\label{ForII}
\int_{\R^2}\, \mathsf{v}^{\vartheta}_t(x,y)\,
K_0\big(\sqrt{2\vartheta}\,|y|\big)\,dy
\,=\,
&\pi \vartheta
\int_{0<r<s<t}\, g_r(x)\,
\nu'\big(\vartheta(s-r)\big)\,
\boldsymbol{E}\big(\vartheta(t-s)\big)\,ds\,dr
\nonumber\\
\,=\,
&\pi
\int_0^{t} \, g_r(x)\,
\int_0^{\vartheta(t-r)}
\nu'(b)\,
\boldsymbol{E}\big(\vartheta(t-r)-b\big)\,db\,dr
\nonumber\\
\,=\,
&\pi
\int_0^{t}\, g_r(x)\,
e^{\vartheta(t-r)}\,dr ,
\end{align}
where we have used the change of variables $b=\vartheta(s-r)$ and applied the
identity (see, for instance,~\cite[Lem.\ ~8.9]{CM}): for all $t,\vartheta>0$,
\[
\int_0^{\vartheta(t-r)}
\nu'(b)\,
\boldsymbol{E}\big(\vartheta(t-r)-b\big)\,db
\,=\,
e^{\vartheta(t-r)},
\qquad 0\le r<t .
\]
By adding~\eqref{ForI} and~\eqref{ForII} we get the desired identity.

\end{proof}

\end{appendix}

\end{document}